\documentclass[12pt]{amsart}
\usepackage[utf8]{inputenc}
\usepackage{amssymb}
\usepackage{amsmath}
\usepackage{amsthm}
\usepackage{graphicx}
\usepackage{amscd}
\usepackage{epic}
\usepackage{xypic}
\usepackage[colorlinks, citecolor=blue]{hyperref}
\usepackage[left=1.5in,right=1.5in,top=1.5in,bottom=1.5in]{geometry}
\usepackage{soul}
\usepackage{comment}
\usepackage{enumitem}
\usepackage{mathrsfs}
\theoremstyle{plain}
\newtheorem{theorem}{Theorem}[section]
\newtheorem{corollary}[theorem]{Corollary}
\newtheorem{lemma}[theorem]{Lemma}

\newtheorem{proposition}[theorem]{Proposition}
\newtheorem{definition-lemma}[theorem]{Definition-Lemma}
\theoremstyle{remark}
\newtheorem{remark}[theorem]{Remark}

\theoremstyle{definition}

\newcommand{\red}[0]{\operatorname{red}}

\def\ddbar{\partial\bar\partial}
\def\ve{\varepsilon}

\def\BC{\operatorname{BC}}
\def\NA{\overline{\operatorname{NA}}}
\def\Null{\operatorname{Null}}
\newcommand{\<}{\leq}
\def\>{\geq}
\newcommand{\mbQ}{\mathbb{Q}}
\newcommand{\mbR}{\mathbb{R}}

\def\mcO{\mathcal{O}}
\newcommand{\num}{\equiv}

\newcommand{\OO}{{\mathcal{O}}}
\newcommand{\Q}{{\mathbb{Q}}}

\newcommand{\R}{{\mathbb{R}}}

\newcommand{\mult}{{\rm mult}}
\newcommand{\Supp}{{\rm Supp}}

\newcommand{\mbC}{\mathbb{C}}
\newcommand{\mbZ}{\mathbb{Z}}

\newcommand{\bir}{\dashrightarrow}
\newcommand{\vphi}{\varphi}
\def\injective{\hookrightarrow}
\def\surjective{\twoheadrightarrow}

\def\lrd{\lfloor}
\def\rrd{\rfloor}

\def\>{\geq}
\def\ve{\varepsilon}

\def\mcO{\mathcal{O}}

\def\mcC{\mathcal{C}}

\def\mcE{\mathcal{E}}

\def\mcL{\mathcal{L}}

\def\eps{\epsilon}

\def\lru{\lceil}
\def\rru{\rceil}
\def\lrd{\lfloor}
\def\rrd{\rfloor}

\def\bbeta{\boldsymbol{\beta}}
\def\aalpha{\boldsymbol{\alpha}}
\def\bomega{\boldsymbol{\omega}}
\def\pt{\operatorname{pt}}
\def\Ex{\operatorname{Ex}}
\def\dim{\operatorname{dim}}
\def\codim{\operatorname{codim}}
\def\sing{\operatorname{\textsubscript{sing}}}
\def\sm{\operatorname{\textsubscript{sm}}}

\def\NA{\operatorname{\overline{NA}}}

\def\Im{\operatorname{Im}}

\def\Spec{\operatorname{Spec}}

\def\ddbar{\partial\bar{\partial}}
\theoremstyle{definition}
\newtheorem{definition}[theorem]{Definition}
\theoremstyle{definition}

\numberwithin{equation}{section}

\theoremstyle{remark}
\newtheorem{claim}[theorem]{Claim}

\title{On the Minimal Model Program for K\"ahler 3-folds}
\author{Omprokash Das}
\address{School of Mathematics\\
Tata Institute of Fundamental Research\\
Homi Bhabha Road, Navy Nagar\\
Colaba, Mumbai 400005}
\email{omdas@math.tifr.res.in}
\email{omprokash@gmail.com}
\thanks{Omprokash Das was partially supported by the Start--Up Research Grant(SRG), Grant No. \# SRG/2020/000348 of the Science and Engineering Research Board (SERB), Govt. Of India.}

\author{Christopher Hacon}
\address{Department of Mathematics\\
University of Utah\\
155 S 1400 E\\
Salt Lake City, Utah 84112, USA}
\email{hacon@math.utah.edu}
\thanks{Christopher Hacon was partially supported by the NSF research grants no: DMS-2301374, DMS-
1952522, DMS-1801851, and a grant from the Simons foundation No. 256202.}

\date{}

\begin{document}
\maketitle

\begin{abstract}
    In this article we prove the existence of pl-flipping and divisorial contractions and pl flips in dimension $n$ for compact K\"ahler varieties, assuming results of the minimal model program in dimension $n-1$. We also give a self contained proof of the cone theorem, the existence of flipping and divisorial contractions, of flips and minimal models in dimension 3.
\end{abstract}

\tableofcontents

\section{Introduction}
In recent years there has been substantial progress towards the minimal model program for K\"ahler varieties.
The 3-fold minimal model program is established in increasing generality in \cite{CP97}, \cite{Pet98}, \cite{Pet01}, \cite{HP15}, \cite{HP16}, \cite{CHP16}, \cite{DO23}, \cite{DH20}, and \cite{DHY23}.
Further progress was then made in \cite{Fuj22}, \cite{DHP22}, and \cite{LM22}, where the minimal model program for projective morphisms between analytic varieties is established in arbitrary dimensions. Note that flipping contractions are, by definition, projective morphisms and hence the above references imply the existence of flips in all dimensions. Unluckily, we do not have a K\"ahler analog of the base point free and cone theorems and hence this is the main difficulty in proving the minimal model program in higher dimensions. 

In dimension 3, the existence of flipping and divisorial contractions for terminal pairs was claimed in \cite{HP16}. Flipping contractions and divisorial contractions to a point for klt 3-folds follow from \cite{CHP16} and klt divisorial contractions to a curve follow from \cite{DH20}.
Unluckily, the proof of the existence of divisorial contractions to a curve  claimed in \cite{HP16} seems to be incomplete (and relies on \cite{AT84} which also seems to contain an error). To make matters more complicated, the arguments of \cite{CHP16} heavily rely on those of \cite{HP16}, and \cite{DH20} also indirectly relies on  \cite{HP16} and \cite{CHP16}.

 The purpose of this paper is to give a brief, unified, and mostly self contained treatment of the minimal model program for klt K\"ahler 3-folds. In particular, we give new proofs of the cone theorem, and the existence of flipping contractions and divisorial contractions for K\"ahler klt 3-folds.  In the process, we also establish an inductive framework that we expect to be useful in proving the existence of flipping contractions and divisorial contractions for higher dimensional K\"ahler klt pairs. The strategy of this paper is primarily inspired by the approaches of \cite{BCHM10}, \cite{DH20}, and \cite{DHP22}. %Another advantage of our approach is to give a new proof of the existence of 3-fold klt flipping contractions, which is independent (and we believe simplifies) the proof given in \cite{CHP16}. 

We will now describe the main results contained in this paper. We begin by proving the existence of pl-divisorial contractions and pl-flipping contractions and the corresponding flips in dimension $n$ assuming a base point free result in dimension $n-1$.
\begin{theorem}\label{t-pl}
    Let $(X,S+B)$ be a $\mbQ$-factorial compact K\"ahler plt pair of arbitrary dimension, and $\pi :S\to T$ a contraction such that $-(K_X+S+B)|_S$ and $-S|_S$ are $\pi$-ample. Then there exists a projective bimeromorphic morphism $p : X \to Z$ with connected fibers such that $p|_S = \pi$
and $p|_{X\setminus S}$ is an isomorphism. If $p$ is of flipping type, then the flip $p^+:X^+\to Z$ exists.
\end{theorem}
In fact we prove a more general version of this results that applies to generalized plt pairs, see Theorems \ref{t-ext} and \ref{t-flip}.
Note that the existence of flips for plt pairs was shown in \cite{DHP22} and \cite{Fuj22}.
In order to apply the above result, we must guarantee the existence of the morphism $\pi :S\to T$. If there exists a K\"ahler form $\omega$ such that $\alpha =[K_X+S+B+\omega]$ is nef and $R=\alpha ^\perp\cap \overline{\rm NA}(X)$ is an extremal ray, then we must show that $R$ can be contracted in $S$.
To this end, we consider the pair $K_S+B_S+\omega _S=(K_S+S+B+\omega)|_S$ induced by adjunction, and we aim to prove that there is a morphism $\pi:S\to T$ such that $K_S+B_S+\omega _S\equiv \pi ^* \omega _T$ where $\omega _T$ is a K\"ahler form on $T$. 
In dimension $\dim S \leq 3$, $\pi$ exists by \cite{DHY23} and therefore Theorem \ref{t-pl} holds in dimension $\dim X\leq 4$. (We remark that the results of \cite{DHY23} only rely on the 3-fold MMP established in this paper.)
We also note that if such a morphism $\pi$ exists, then $-(K_S+B_S)\equiv _T \omega_S$ is ample over $T$ and hence this morphism is projective and so, a posteriori, the existence of $\pi$ is implied by the usual base point free theorem.

We now turn our attention to the full minimal model program for dlt K\"ahler 3-folds. We first show the existence of flipping and divisorial contractions (see Theorems \ref{t-flipcont} and \ref{t-divcont}).
\begin{theorem}\label{thm:contractions}
    Let $(X,B)$ be a $\mbQ$-factorial compact K\"ahler 3-fold dlt pair such that $K_X+B$ is pseudo-effective. Let $\omega$ be a K\"ahler class such that $\alpha=K_X+B+\omega$ is a nef and big class and $\alpha ^\perp \cap \overline {\rm NA}(X)=R$ is an extremal ray. Then there exists a projective bimeromorphic morphism $f:X\to Z$ with connected fibers such that 
    $\alpha = f^*\alpha _Z$, where $\alpha _Z$ is a K\"ahler class on $Z$. Moreover, if $f$ is a divisorial contraction, then $(Z, f_*B)$ has $\mbQ$-factorial klt singularities. 
\end{theorem}
% \begin{theorem}\label{thm:contractions}
%    Let $(X, B)$ be a $\mbQ$-factorial compact K\"ahler $3$-fold dlt pair such that $K_X+B$ is pseudo-effective. If $R$ is a $(K_X+B)$-negative extremal ray, then there is a projective bimeromorphic morphsim $f:X\to Z$ with connected fibers such that a curve $C\subset X$ is contracted by $f$ if and only if $[C]\in R$.  Moreover, if $f$ is a divisorial contraction, then $(Z, f_*B)$ has $\mbQ$-factorial klt singularities, and if $f$ is a flipping contraction, then the corresponding flip $f^+:X^+\to Z$ also exists.   
% \end{theorem}

Using only a special case of the above theorem (namely, assuming $X$ is \textit{strongly} $\mbQ$-factorial), we give a quick proof of the cone theorem when $K_X+B$ is pseudo-effective. 
%Note that contractions to a point are typically easier since in many cases one can apply a result of Grauert \cite{Gra62}. For example, in Corollary \ref{c-1} we show the existence of extremal contractions of a divisor to a point in arbitrary dimension. We are u 
\begin{theorem}  \label{t-cone}  
Let $(X,B)$ be a compact K\"ahler $3$-fold dlt pair such that $K_X+B$ is pseudo-effective. Then there are at most countably many rational curves $\{\Gamma _i\}_{i\in I}$ in $X$ such that $-(K_X+B)\cdot \Gamma _i\leq 6$ for all $i\in I$ and 
 \[\overline{\rm NA}(X)=\overline{\rm NA}(X)_{(K_X+B) \geq 0}+\sum _{i\in I}\mathbb R ^+[\Gamma _i].\]
 % Moreover, if $K_X+B$ is big, then $I$ is finite.
 \end{theorem}

Finally, using the cone and contraction theorems above we prove the existence of minimal models when $K_X+B$ is pseudo-effective. 
\begin{theorem}\label{t-mmp-mfs}
    Let $(X,B)$ be a $\mbQ$-factorial compact K\"ahler $3$-fold dlt pair such that $K_X+B$ is pseudo-effective. Then after finitely many $(K_X+B)$-flips and divisorial contractions
    \[X=X_0\dasharrow X_1\dasharrow X_2\dasharrow \ldots \dasharrow X_n\]
    we obtain that $K_{X_n}+B_n$ is nef.   
\end{theorem}
The existence of Mori-fiber spaces also follows, but we do not pursue it here as it also relies on the results of \cite{Bru06}, \cite{HP15}, and \cite{DH20}.\\

\begin{remark}
We also note that Theorems 2.28, 2.30 and 2.31 in \cite{DO23} rely on the contraction results of \cite{HP16} and \cite{CHP16}, however these theorems are also implied by Theorem \ref{thm:contractions}, and so the main results of \cite{DO23} are unaffected.

 %   In \cite{DO23} the authors O. Das and W. Ou proved contractions theorems (\cite[Theorems 10.14, 10.16 and 10.17]{DO23}) for dlt pairs $(X, B)$ such that $K_X+B$ is pseudo-effective, using the contraction results of \cite{HP16} and \cite{CHP16}. So these results of \cite{DO23} should be replaced by our Theorem \ref{thm:contractions} to make the main theorems of \cite{DO23} error free. 
\end{remark}

This article is organized in the following manner. In Section \ref{sec:prelim} we recall some definitions and provide appropriate references, in Section \ref{sec:point-contractions} we prove some contraction theorems when the exceptional locus maps to finitely many points. Section \ref{sec:main-contractions} is the core of this article. In this section we prove Theorems \ref{t-pl} and \ref{thm:contractions}. Finally, in Section \ref{sec:cone-and-mm} we prove the cone Theorem \ref{t-cone} and the existence of minimal models, Theorem \ref{t-mmp-mfs}.\\

\noindent
{\bf Acknowledgement } We would like to thank Paolo Cascini, Mihai P\u{a}un, J\'anos Koll\'ar, and Matei Toma for useful comments, references and suggestions. % for pointing out the errors in \cite[Proposition 7.1]{HP16} and \cite{AT84}. 

\section{Preliminaries}\label{sec:prelim}
We will follow the usual conventions of the minimal model program. In particular, we refer the reader to \cite[Chapter 2]{KM98} for the definition of pairs and their singularities (klt, lc, plt etc.), \cite[Subsection 2.1]{DHY23} for the definitions of generalized pairs and related singularities (glc, gklt, gdlt, etc.), \cite[Chapters 2, 3, 11]{Fuj22} for many key concepts of the minimal model program for analytic varieties, to \cite[Definition 2.8]{DHP22} for the definitions of nef, minimal and log terminal models, \cite{HP16} for a discussion of Bott-Chern cohomology $H^{1,1}_{\rm BC}(X)$, the  K\"ahler and nef cones $\mathcal K\subset \bar {\mathcal K}$, the Mori cone and the cone of positive closed currents $\overline{\rm NE}(X)\subset \overline {\rm NA}(X)\subset N_1(X)$.  Analytic pl-flipping contractions and pl-flips are defined in \cite[Chapter 15]{Fuj22}. 

\begin{definition}\label{def:fujiki-class-c}
    A compact analytic variety $X$ is said to be in Fujiki's class $\mcC$ if $X$ is bimeromorphic to a compact K\"ahler manifold.
\end{definition}

\begin{definition}\label{def:modified-kahler}
    Let $X$ be a normal compact analytic variety. A closed positive $(1,1)$ current $T$ on $X$ with local potentials is called a K\"ahler current if $T\>\omega$ for some smooth Hermitian form $\omega$ on $X$. A $(1,1)$ class $\alpha\in H^{1,1}_{\BC}(X)$ is called big if it contains a K\"ahler current. A class $\beta\in H^{1,1}_{\BC}(X)$ is called a \emph{modified K\"ahler class} if there is a K\"ahler current $T$ such that the Lelong numbers satisfy $\nu(T, P)=0$ for all prime Weil divisors $P\subset X$.  
\end{definition}

\begin{remark}\label{rmk:big-class-in-fujiki}
    Note that if $X$ is a compact complex manifold in Fujiki's class $\mcC$, then clearly $X$ carries a big $(1,1)$ class $\alpha\in H^{1,1}_{\BC}(X)$. However, if $X$ is singular, then it is not clear whether $X$ carries any big $(1,1)$ class or not, the issue here is that the pushforward of a smooth K\"ahler form from a resolution of $X$ is a positive current on $X$ which may not have local potentials.      
\end{remark}

\begin{definition}\label{def:mod-nef}
Let $X$ be a normal compact analytic variety and $\alpha\in H^{1,1}_{\BC}(X)$ a pseudo-effective class. Then the \emph{negative part} $N(\alpha)$ of the Boucksom-Zariski decomposition of $\alpha$ is an effective $\mbR$-divisor (see \cite[Defintion A.6]{DHY23} and \cite[Definition 3.7]{Bou04}). We note that in \cite[Definition  A.6]{DHY23} it is assumed that $X$ is K\"ahler, however, it was never used.
If $X$ is $\mbQ$-factorial, $N(\alpha)$ is $\mbR$-Cartier, and so $\alpha-[N(\alpha)]$ is a $(1,1)$ class in $H^{1,1}_{\BC}(X)$. In this case we define $P(\alpha):=\alpha-[N(\alpha)]$ and call it the \emph{positive part} of $\alpha$. Lemma \ref{lem:positive-part} below shows that $P(\alpha)$ is a modified nef class (see \cite[Definition 2.2(ii)]{Bou04}) when $X$ is a K\"ahler variety. We call $\alpha=[N(\alpha)]+P(\alpha)$, the \emph{Boucksom-Zariski decomposition} of $\alpha$.  
\end{definition}

\begin{lemma}\label{lem:existence-of-modified-kahler}
    Let $X$ be a normal $\mbQ$-factorial compact analytic variety in Fujiki's class $\mcC$, and    $\alpha\in H^{1,1}_{\BC}(X)$ a big $(1,1)$ class. Then $X$ contains a modified K\"ahler class. 
\end{lemma}

\begin{proof}
    Let $f:Y\to X$ be a resolution of singularities of $X$ such that $f$ is an isomorphism over $X_{\sm}$. By Demailly's regularization theorem, there is a K\"ahler current $T$ in $f^*\alpha\in H^{1,1}_{\BC}(Y)$ with analytic singularities such that $T\geq \omega_Y$ for some smooth Hermitian form $\omega_Y$ on $Y$. Let $T=D+R$ be the Siu decomposition of $T$, where $D$ is an effective $\mbR$-divisor and $R$ is the residue current. Then $R\geq \omega_Y$. Let $\omega$ be a smooth Hermitian form on $X$, then there is a $C>0$ such that $\omega_Y\geq Cf^*\omega$. Note that by \cite[Proposition 4.6.3(i)]{BG13}, $f_*T$ has local potentials on $X$ and $[f_*T]=\alpha$. Now  since $X$ is $\mbQ$-factorial, $f_*D$ is $\mbR$-Cartier, and hence $f_*R$ has local potentials on $X$. However, $f_*R\geq C\omega$, and the generic Lelong numbers $\nu(f_*R, P)=0$ for all prime Weil divisors $P$ on $X$. Therefore $[f_*R]\in H^{1,1}_{\BC}(X)$ is a modified K\"ahler class.
\end{proof}

\begin{lemma}\label{lem:positive-part}
Let $X$ be a normal $\mbQ$-factorial compact analytic variety in Fujiki's class $\mcC$, such that $H^{1,1}_{\BC}(X)$ contains a big $(1,1)$ class. If $\alpha\in H^{1,1}_{\BC}(X)$ is a pseudo-effective class, then the positive part $P(\alpha)$ of the Boucksom-Zariski decomposition of $\alpha$ defined above is a modified nef class. 
\end{lemma}

\begin{proof}
    Let $f:Y\to X$ be a resolution of singularities of $X$ such that $Y$ is a K\"ahler manifold and $f^*\alpha=N(f^*\alpha)+P(f^*\alpha)$ the Boucksom-Zariski decomposition of $f^*\alpha$ as in \cite[Definition 3.7]{Bou04}. Then from \cite[Definition A.6]{DHY23} it follows that $P(\alpha)=f_*(P(f^*\alpha))$. Note that $P(f^*\alpha)$ is modified nef by \cite[Proposition 3.8(i)]{Bou04}. Let $\beta$ be a modified K\"ahler class on $X$, its existence is guaranteed by Lemma \ref{lem:existence-of-modified-kahler}. It suffices to show that $P(\alpha)+\beta$ is modified K\"ahler, as the modified nef cone is the closure of the modified K\"ahler cone (see \cite[\S 2.7]{Bou04}). Let $T$ be a K\"ahler current in $\beta$ with generic Lelong numbers $\nu(T, P)=0$ for all prime Weil divisors $P$ on $X$. Let $f^*T=D+R$ be the Siu decomposition, where $D$ is the divisorial part and $R$ is the residue current. Then $D$ is a $\mbR$-Cartier divisor such that $\Supp(D)\subset \Ex(f)$ and $R$ is a K\"ahler current. In particular, $[R]\in H^{1,1}_{\BC}(Y)$ is a modified K\"ahler class. Then by \cite[Lemma 2.35]{DH20} applied to the class $[R]$ and replacing $Y$ by a higher resolution we may assume that there is a K\"ahler class $\omega_Y$ and an effective $f$-exceptional $\mbR$-divisor $E$ on $Y$ such that $f^*\beta=\omega_Y+E$. Since the modified nef cone is the closure of the modified K\"ahler cone, it follows that $P(f^*\alpha)+\omega_Y$ is modified K\"ahler on $Y$. Let $\Theta$ be a K\"ahler current in $P(f^*\alpha)+\omega_Y$ such that the generic Lelong numbers $\nu(\Theta, Q)=0$ for all prime Weil divisors $Q$ on $Y$. Note that $P(f^*\alpha)=f^*P(\alpha)+F$ for some $f$-exceptional $\mbR$-divisor $F$ (but not necessarily effective). Then the current $\Theta+E+F$ is in the class $f^*(P(\alpha)+\beta)+[F]$. Let $U:=X\setminus f(\Ex(f))$; then $(f_*\Theta)|_U$ is a closed positive $(1,1)$ current with local potentials representing the class $(P(\alpha)+\beta)|_U$. Since $\codim_Xf(\Ex(f))\geq 2$, by the same argument as in the proof of \cite[Proposition 4.6.3(i)]{BG13} it follows that $f_*\Theta|_U$ extends to a unique closed positive $(1,1)$ current $f_*\Theta$ with local potentials representing the class $P(\alpha)+\beta$. From the definition of $\Theta$ it follows that $\nu(f_*\Theta, P)=0$ for all prime Weil divisors $P$ on $X$. Thus $P(\alpha)+\beta$ is a K\"ahler class.    
\end{proof}

\section{Singularities of Currents on Complex Spaces}

\begin{definition}
Let $X$ be a normal analytic variety and $T$ a closed positive $(1, 1)$ current on $X$ with local potentials. Then for each $c>0$ we define
\[ 
E_c:=\{x\in X\;:\; \nu(T, x)>c\},
\]
where $\nu(T, x)$ denote the Lelong number of $T$ at $x\in X$.\\ 
This is an analytic subset of $X$ by \cite{Siu74}, also see \cite[Lemma 4.16]{HP24}. We define 
\[
E_+(T)=\cup_{c\in \mbR^+}E_c.
\]
If additionally $X$ is compact and $\alpha\in H^{1,1}_{\BC}(X)$ is a big class, then we define the \emph{non-K\"ahler} locus of $\alpha$ as follows:
\[
E_{nK}(\alpha ):=\cap_{T\in \alpha} E_+(T),
\]
where $T$ runs over all K\"ahler currents contained in the class $\alpha$.
\end{definition}

We will also need to consider the currents with \emph{admissible singularities} and \emph{weakly analytic singularities} as in \cite{HP24}.
\begin{definition}\label{def:admissible-singularities}\cite[Definition 4.10]{HP24}
    Let $X$ be a normal analytic variety and $\vphi:X\to [-\infty, \infty)$ a function on $X$. We say that $\vphi$ has \textit{admissible singularities} if 
    \[
    \vphi:=\max\{\vphi_1, \vphi_2,\ldots, \vphi_k\},
    \]
where each $\vphi_i:X\to [-\infty, \infty)$ is a function with weakly analytic singularities in the sense \cite{Dem92}, i.e. locally
\[
\vphi_i=\lambda_i\log\left(\sum_{j=1}^{n_j}|f_j|^2 \right)+g_i,
\]
where $\lambda_i\>0$, $f_j$ are holomorphic functions and $g_i$ are bounded functions.

In this case we say that the positive current $T=\alpha+i\ddbar\vphi$, where $\alpha$ is a real closed smooth $(1,1)$ form, has \emph{admissible singularities}. 
\end{definition}

\begin{definition}\label{def:weakly-analytic-singular-current}\cite[Definition 4.11]{HP24}
    Let $X$ be a normal analytic variety, and $T=\alpha+i\ddbar\vphi$ a closed positive $(1,1)$ current on $X$, where $\alpha$ is a smooth closed positive $(1,1)$ form on $X$ with local potentials. We say that $T$ has \textit{weakly analytic singularities} if the following conditions are satisfied:
    \begin{enumerate}
        \item there is a proper bimeromrophic morphism $\pi:\widehat X\to X$ from a normal analytic variety $\widehat X$, and
        \item a closed positive $(1,1)$ current $\widehat T$ on $\widehat X$ such that 
        \[
        \widehat T=\pi^*\alpha+i\ddbar\psi\geq 0,
        \]
        where $\psi$ has admissible singularities and $\pi_*\widehat T=T$.
    \end{enumerate}
From \cite[Remark 4.12]{HP24} it follows that $\widehat T=\pi^*T$. Also, note that if $T$ has weakly analytic singularities, then $E_+(T)$ is a closed analytic subset of $X$ by \cite[Lemma 4.16]{HP24}.
%If $\nu :X'\to \widehat X$ is a resolution of $\widehat X$ and of the ideal $\mathcal I _{\widehat T}\subset \OO _{\widehat X}$ of singularities of $\widehat T$, then we say that the composition $X'\to X$ is a resolution of the singularities of $T$.
\end{definition}

Next we define the \textit{restricted} non-K\"ahler locus of a current as in \cite{HP24}.
\begin{definition}\label{def:analytic-non-kahler-lcous}\cite[Definition 4.18]{HP24}
    Let $X$ be a normal compact K\"ahler analytic variety and $\alpha\in H^{1,1}_{\BC}(X)$ a nef and big class. The restricted non-K\"ahler locus of $\alpha$ is the following set
    \[
    E^\text{as}_{nK}(\alpha):=\bigcap_{T\in\alpha} E_+(T),
    \]
    where $T$ is a K\"ahler current with weakly analytic singularities.

    \end{definition}

Next we recall two very important results from \cite{HP24}.

\begin{lemma}\label{lem:as-locus-is-analytic}\cite[Corollary 4.20]{HP24}
    Let $X$ be a normal compact K\"ahler variety and $\alpha\in H^{1,1}_{\BC}(X)$ a big class. There is a K\"ahler current $T$ on $X$ with weakly analytic singularities such that $E^\text{as}_{nK}(\alpha)=E_+(T)$; in particular, $E^\text{as}_{nK}(\alpha)$ is a closed analytic subset of $X$.   
\end{lemma}

\begin{theorem}\label{thm:non-kahler-equal-null}\cite[Theorem 4.21]{HP24}
    Let $X$ be a normal compact K\"ahler variety and $\alpha\in H^{1,1}_{\BC}(X)$ a nef and big class. Then $E^\text{as}_{nK}(\alpha)=\Null(\alpha)$, in particular, $\Null(\alpha)$ is an analytic set. 
\end{theorem}

We will also need the following lemmas. 

\begin{lemma}\label{lem:resolution-of-current-singularities}%\footnote{Om: I didn't use Mihai's argument here, found shorter argument using his ideas.}
  Let $X$ be normal compact K\"ahler variety, $\alpha$ a real smooth closed $(1,1)$ form, and $T=\alpha+i\ddbar\vphi\geq \omega$ a closed positive $(1,1)$ current with weakly analytic singularities, where $\omega$ is a fixed K\"ahler form on $X$. Let $Z:=E_+(T)$. %, and $f:Y\to X$ a resolution of the singularities of $X$ and $T$. 
 Then there exists a resolution $f:Y\to X$ such that 
  \[
  f^*T=F+\Theta\geq f^*\omega
  \]
 where $F$ is an effective $\mbR$-divisor with $f(\Supp F)=Z$, and $\Theta\>f^*\omega$ is a closed positive $(1,1)$ current such that the classes $[\Theta]\mbox{ and } [\Theta-f^*\omega]\in H^{1,1}_{\BC}(Y)$ are both nef.
 %\footnote{Doesn't this actually show that $\Theta -f^*\omega=\beta +i\ddbar\rho$ is nef? Om: Yes, I believe so, because $E_+(R)=\emptyset$ as $\rho$ is a bounded function, and then by the same theorem of Mihai as before the class $[R]=[\Theta-f^*\omega]$ is nef. But $\Theta-f^*\omega$ not a smooth form though as $\rho$ is only a bounded function.}
\end{lemma}

\begin{proof} 
Since $T$ has weakly analytic singularities, we may assume that there is a resolution $f:Y\to X$ such that the Siu decomposition is $f^*T=F+R_0$, where $F$ is an effective $\mbR$-divisor and $R_0$ is the residue current whose  local potentials are bounded functions.
Since $f^*T\geq f^*\omega$, then $R_0\geq f^*\omega$ and so the local potentials of $R:=R_0-f^*\omega$ are also bounded functions.
%Now let $f^*(T-\omega)=F+R$ be the Siu decomposition of the positive current $f^*(T-\omega)$, where $F$ is an effective $\mbR$-divisor and $R$ is the residue current. Now since $T-\omega$ has admissible singularities and $f^{-1}Z$ is a SNC divisor, it follows from the computation of $i\ddbar(\vphi\circ f)$ that the local potentials of $R$ are bounded functions. 
In particular, the Lelong numbers satisfy $\nu(R, y)=0$ for all $y\in Y$, and we can write $R=\beta+i\ddbar\rho\geq 0$ (see \cite[\S 3]{DP04}), where $\beta$ is a real smooth closed $(1,1)$ form and $\rho:Y\to \mbR$ is a bounded quasi-psh function. Therefore we have $f^*T=F+(\beta+f^*\omega)+i\ddbar\rho=F+\Theta\>f^*\omega$, where $\Theta:=\beta+f^*\omega+i\ddbar\rho=R+f^*\omega\>f^*\omega$. Since $E_+(\Theta-f^*\omega)=E_+(R)=\emptyset$ by our construction, by \cite[Theorem 2, Page 418]{Pau98} the class $[\Theta-f^*\omega]\in H^{1,1}_{\BC}(Y)$ is nef, and so $[\Theta]=[\Theta-f^*\omega]+[f^*\omega]$ is also nef. 
   
\end{proof}

\begin{lemma}\label{lem:non-kahler-locus}
Let $f:Y\to X$ be a proper bimeromorphic morphism from a compact complex manifold $Y$ in Fujiki's class $\mcC$ to a normal compact analytic variety $X$ with rational singularities. Let $\alpha$ be a $(1,1)$ big class on $X$. Then $f^{-1}(E_{nK}(\alpha))\cup\Ex(f)\subset E_{nK}(f^*\alpha)$.
\end{lemma}

\begin{proof}
Supposing  that $y\notin E_{nK}(f^*\alpha)$, we must show that
$y\notin f^{-1}(E_{nK}(\alpha))\cup\Ex(f)$. Since $Y$ is a compact complex manifold in Fujiki's class $\mcC$, by \cite[Theorem 3.27(ii)]{Bou04} (also see \cite[Theorem 2.2]{CT15}), there is a K\"ahler current $T$ in the class $f^*\alpha$ with analytic singularities such that $E_{nK}(f^*\alpha)=E_+(T)$. Thus $y\notin E_{nK}(f^*\alpha)$ implies that $T$ is a smooth form near $y$. 
Suppose that  $y\in \Ex(f)$. We can write $T=f^*\theta+i\ddbar\vphi$, where $\theta$ is a real closed smooth $(1,1)$ form on $X$ with local potentials representing the class $\alpha$ and $\vphi$ is a quasi-psh function on $Y$. Since $T$ is a K\"ahler current, there is a smooth Hermitian form $\omega_Y$ such that $f^*\theta+i\ddbar\vphi\>\omega_Y$. Let $E$ be an irreducible component of $f^{-1}(f(y))$ containing $y$. Since $T$ is smooth near $y$, $E$ is not contained in the pluripolar locus $\{y'\in Y\;|\; \vphi(y')=-\infty\}$, in particular, $T|_{E}$ is well defined and we have 
\begin{equation}
    T|_{E}=(f^*\theta+i\ddbar\vphi)|_{E}=i\ddbar(\vphi|_{E})\>\omega_Y|_{E}.
\end{equation}
Thus $\vphi|_{E}$ is a strictly psh function. Since $E$ is compact, from the maximal principle of psh functions it follows that $\vphi|_{E}$ is a constant function, which is a contradiction to the above inequality. In particular, $y\not\in\Ex(f)$, and so
$f$ is an isomorphism near $y\in Y$. By \cite[Lemma 3.4]{HP16}, $S:=f_*T$ is a closed positive $(1,1)$ current on $X$ with local potentials such that $[S]=\alpha$; clearly $S$ is a K\"ahler current. Since $f$ is an isomorphism near $y\in Y$, it follows that $S$ is a smooth form near $x=f(y)\in X$. 
In particular, $x\not\in E_{nK}(\alpha)$, and the proof of the inclusion  $f^{-1}(E_{nK}(\alpha))\cup\Ex(f)\subset E_{nK}(f^*\alpha)$ is complete. %\footnote{CH: Don't we also need to show that ${\rm Ex}(f)\subset E_{nK}(f^*\alpha)$? This is the easy part because otherwise, pick $x\in {\rm Ex}(f)\setminus E_{nK}(f^*\alpha)$, then there is a an exceptional curve $x\in C$ and a K\"ahler current $\beta \equiv f^*\alpha$ which is smooth at $x$, but then $\int _C\beta >0$ which is impossible as $C\cdot f^* \alpha=f_*C\cdot \alpha =0$.}\\ 

% For the reverse inequality, suppose that $y\not\in f^{-1}(E_{nK}(\alpha))\cup\Ex(f)$. Then $f$ is an isomorphism near $y$ and $x=f(y)\not\in E_{nK}(\alpha)$. Thus there is a K\"ahler current $T\in \alpha$ such that $T$ is smooth near $x$, in particular, $f^*T\in f^*\alpha$ is K\"ahler current on $Y$ which is smooth near $y$. Hence $y\not\in E_{nK}(f^*\alpha)$. This completes our proof. 
\end{proof}

\begin{remark}
    With the same notations and hypothesis as in Lemma \ref{lem:non-kahler-locus} if additionally we assume that $X$ is K\"ahler and $\alpha$ is nef (in addition to being big), then using \cite[Corollary 4.20]{HP24} it follows easily that $E^\text{as}_{nK}(f^*\alpha)=f^{-1}(E^\text{as}_{nK}(\alpha))\cup \Ex(f)$. 
\end{remark}

\begin{lemma}\label{lem:null-locus}
    Let $f:Y\to X$ be a proper bimeromorphic morphism between two normal compact analytic varieties and $\alpha$ a  nef and big $(1,1)$-class on $X$. Then $\Null(f^*\alpha)=f^{-1}(\Null(\alpha))\cup \Ex(f)$. 
\end{lemma}

\begin{proof}
First we will show that $\Null(f^*\alpha)\subset f^{-1}(\Null(\alpha))\cup \Ex(f)$. To that end assume that $y\notin f^{-1}(\Null(\alpha))\cup \Ex(f)$. Then $f$ is an isomorphism near $y\in Y$. 
Assume by contradiction that $y\in \Null(f^*\alpha)$, then there is a subvariety $V\subset Y$ containing $y$ such that $(f^*\alpha)^{\dim V}\cdot V=0$. Then by the projection formula we have $\alpha^{\dim V}\cdot f_*V=0$. However, since $f$ is an isomorphism near $y$, $f|_V:V:\to f(V)$ is bimeromorphic, and hence $f_*V\neq 0$. Therefore $f(y)\in f(V)\subset \Null(\alpha)$, and so $y\in f^{-1}(\Null(\alpha))$, which is a contradiction.

Now we will show the reverse inclusion. Let $y\in \Ex(f)$ and $E$ an irreducible component of $f^{-1}(f(y))$ containing $y$. Then clearly, $(f^*\alpha)^{\dim E}\cdot E=0$, and thus $y\in \Null(f^*\alpha)$. Now choose $y\in f^{-1}(\Null(\alpha))\setminus\Ex(f)$; then $f$ is an isomorphism near $y$. Since $f(y)\in \Null(\alpha)\setminus f(\Ex(f))$, there is a subvariety $f(y)\in W\subset X$ not contained in $f(\Ex(f))$ such that $\alpha^{\dim W}\cdot W=0$. Let $V\subset Y$ be the strict transform of $W$ under $f$. Then by the projection formula we have $(f^*\alpha)^{\dim V}\cdot V=0$; since $y\in V$, it follows that $y\in \Null(f^*\alpha)$. This completes the proof.
         
\end{proof}

\subsection{Generalized pairs}
\begin{definition}\label{d-gpair}
 	Let $\pi :X\to S$ be a proper morphism of normal K\"ahler
varieties such that $S$ is relatively compact and  $\nu :X'\to X$ a resolution of singularities, $B'$ an $\R$-divisor on $X'$ with simple normal crossings support, and $\beta '$ a real closed smooth $(1,1)$-form on $X'$, such that
\begin{enumerate}
\item $B:=\nu _*B'\geq 0$,
%\item $\beta'$ is a positive current,  
\item $[ \beta ']\in H^{1,1}_{\rm BC}(X')$ is nef over $S$, and 
\item $[K_{X'}+B'+\beta ']=\nu^* \gamma $, where $\gamma \in H^{1,1}_{\rm BC}(X)$.
\end{enumerate}
  Then we let $\beta =\nu _*\beta '$ and we say that $\nu: (X',B'+\beta')\to (X,B+\beta)$  is a generalized pair.
We will often abuse notation and say that $(X,B+\beta)$ (or $(X,B+\beta /S)$) is a generalized pair (over $S$) and $\nu: (X',B'+\beta')\to (X,B+\beta)$ is a log resolution.
We will often assume that   $X$ is compact and $S$ is a point (or that $X=S$) and we omit $\pi :X\to S$.
\end{definition}
Note that given $\beta'$, we can define the corresponding nef b-(1,1) current $\boldsymbol{\beta}:=\overline{\beta'}$ as follows. For any bimeromorphic morphism $p :X''\to X'$ we define $\boldsymbol{\beta}_{X''}=p ^* \beta '$ and for any bimeromorphic morphism $q :X''\to X'''$ we let $\boldsymbol{\beta}_{X'''}=q_*\boldsymbol{\beta}_{X''}$ (in general this is a closed (1,1)-current, not necessarily a (1,1)-form). Using the projection formula, one easily checks that $q_*\boldsymbol{\beta}_{X''}$ is well defined (i.e. $\boldsymbol{\beta}_{X'''}$ does not depend on the choice of the common resolution $X''$ of $X'$ and $X''$) and that for any bimeromorphic morphism $r:X_1\to X_2$ of bimeromorphic models of $X$, we have $r_* \boldsymbol{\beta}_{X_1}=\boldsymbol{\beta}_{X_2}$.
We say that $\overline{\beta'}$ descends to $X'$. Note that for any bimeromorphic morphism $p :X''\to X'$, we also have
$\overline{\beta'}=\overline{\beta''} $ where $\beta '' =\boldsymbol{\beta}_{X''}$, and so $\boldsymbol{\beta}$ also descends to $X''$.

Similarly, if $\nu :Y\to X'$ is a proper bimeromorphic morphism, then write $K_Y+B_Y=\nu ^*(K_{X'}+B')$. For any proper bimeromorphic morphism $\mu:Y\to Y'$ we let $B_{Y'}=\mu _* B_Y$. In this way we have defined a b-divisor  $\mathbf B$ (whose trace $\mathbf B_Y$ on $Y$ is $B_Y$). Since the b-divisor $\mathbf {K+B}=\overline{K_{X'}+B_{X'}}$ and the b-$(1,1)$-form {\mathversion{bold}$\beta$}$=\overline{\beta _{X'}}$ descend to $X'$, we say that the generalized pair $(X,B+\beta)$ descends to $X'$.

     We will often denote the generalized pair $\nu: (X',B'+\beta')\to (X,B+\beta)$ by $(X,B+\bbeta )$ where $\bbeta =\overline{\beta '}$.
     Note that then $\beta '=\bbeta _{X'}$ and $B'=\nu ^*(K_X+B+\beta)-(K_{X'}+\beta ')$.

We define the \emph{generalized discrepancies} $a(P;X,B+\bbeta)=-{\rm mult}_P(\mathbf B_Y)$, where $P$ is a prime divisor on a bimeromorphic model $Y$ of $X$.
We say that $(X,B+\bbeta )$ is \emph{generalized klt} or generalized Kawamata log terminal (resp. \emph{generalized lc} or generalized log canonical) if for any log resolution $\nu:X'\to X$, we have $\lfloor \mathbf B_{X'} \rfloor\leq 0$, i.e. $a(P;X,B+\bbeta)>-1$ for all prime divisors $P$ over $X$ (resp. $a(P;X,B+\bbeta)\geq -1$ for all prime divisors $P$ over $X$). This can be checked on a single given log resolution.
We say that $(X,B+\bbeta )$ is \emph{generalized dlt} (divisorially log terminal) if there is an open subset $U\subset X$
such that $(U,(B+\bbeta)|_U)$ is a log resolution (of itself) and $a(P;X,B+\bbeta)\geq -1$ for every prime divisor $P$ over $X$ such that the generic point of $\rm{center}_X(P)$ is contained in $U$
and $a(P;X,B+\bbeta)>-1$ for every prime divisor $P$ over $X$ with ${\rm center}_X(P)\subset X\setminus U$.

We remark that if $\alpha'$ is a real closed smooth (1,1)-form in the class  $[\beta']$ and $\aalpha =\overline{ \alpha '}$, then we have $a(P;X,B+\bbeta)=a(P;X,B+\aalpha)$, i.e. the generalized discrepancies do not depend on  the choice of the representative of $[\beta']$ (recall that $\bbeta$ descends to $\beta'$).

\begin{remark}
    Following standard convention, in the rest of the article we will say that a generalized pair $(X, B+\bbeta)$ is sub-gklt, sub-glc, etc. if $B$ is not necessarily effective, and gklt, glc, etc. if $B$ is effective.   
\end{remark}

\subsection{Adjunction for Generalized Pairs}\label{subsec:adjunction} In this subsection we will define adjunction for generalized pairs.\\
\begin{itemize}
    \item Recall that, if $(X, S+B)$ is a log canonical pair such that $\lrd S+B\rrd=S$ and $S^\nu\to S$ is the normalization, then by adjunction (see \cite[Chapter 4]{Kol13}) there is an effective $\mbQ$-divisor ${\rm Diff}_{S^\nu}(B)\>0$ on $S^\nu$, called the \emph{different}, such that 
\[  
(K_X+S+B)|_{S^\nu}=K_{S^\nu}+{\rm Diff}_{S^\nu}(B).
\]
\end{itemize}
Now let $(X, B+\bbeta)$ be a generalized pair such that the coefficients of $B$ are $\leq  1$ and $S$ a component of $\lrd B\rrd$. Let $f:X'\to X$ be a log resolution of the $(X, B+\bbeta)$, $S'$ is the strict transform of $S$ and $K_{X'}+B'+\bbeta_{X'}=f^*(K_X+B+\bbeta_X)$. 
Let $S^n\to S$ be the normalization morphism; then $f|_{S'}:S'\to S$ factors through $S^\nu$, and we denote the induced morphism by $g:S'\to S^n$. We define
\[K_{S^n}+B_{S^n}+\bbeta_{S^n}:=(K_X+B+\bbeta_X)|_{S^n}=g_*(K_{S'}+(B'-S')|_{S'}+\bbeta_{X'}|_{S'}).\]

It is easy to see that this definition does not depend on the log resolution $f$, and  $(S^n, B_{S^n}+\bbeta_{S^n})$ is a generalized pair. If $(X, B+\bbeta)$ is a glc pair (in particular, $B$ is a boundary divisor), then from Proposition 4.5 and 4.7 of \cite{Kol13} it follows that $B_{S^\nu}$ is a boundary divisor. Readers may compare this with the more familiar case of adjunction for generalized pairs on algebraic varieties, see \cite[Remark 4.8]{BZ16}.

\begin{remark}
 If $(X, B+\bbeta)$ is a glc pair and $K_X+B$ is $\R$-Cartier, then we can write $K_{X'}+B'-E=f^*(K_X+B)$, where $0\<E=f^*\bbeta _X-\bbeta _{X'}$ is an effective $f$-exceptional $\mbR$-divisor (by the negativity lemma). It then follows easily from the definition above that 
 \[B_{S^n}={\rm Diff}_{S^n}(B)+g_*(E|_{S'})\geq {\rm Diff}_{S^n}(B).\]
\end{remark}

\begin{remark}\label{r-diff}
 Note that $B_{S^n}$ is determined by a surface computation, and therefore if $(X,B+\bbeta )$ is generalized log canonical in codimension 2, then for the purpose of computing $B_{S^n}$ we may assume that $X$ is a surface. A simple computation shows that $(X, B)$ has numerically log canonical singularities, so that $X$ is $\mbQ$-Gorenstein, $K_X+B$ is $\mbR$-Cartier and $(X, B)$ is log canonical.
%Suppose that $f:X'\to X$ is a log resolution. Then $\bbeta _{X'}\equiv _X-\sum p_iP_i$, where $p_i\in \R^{\geq 0}$ and $P_i$ are Cartier. Thus by classification of log canonical surface singularities (working locally over $X$), we may assume that $K_X$ has Cartier index $n\in \mathbb N$. %%and the coefficients of $E:=f^*\bbeta _X-\bbeta _{X'}$ are of the form $\frac{\sum s_ip_i}n$, where $s_i\in \mathbb N$.  Therefore the coefficients of $B_{S^n}$ are of the form $1-\frac 1 n +s_ip_i$ where $s_i\in \mathbb N$.\footnote{CH: Not $1-\frac 1 n +\frac{\sum s_ip_i}n$. Om: Why? The elements of the derived set $D(I)$ is of the form $1-\frac 1 n +\frac{\sum s_ip_i}n$.}
\end{remark}

We will also need the following lemma.
\begin{lemma}\label{l-dlt}
    Let $(X,B+\bbeta _X)$ be a generalized dlt pair. Then for every point $x\in X$, there is a relatively compact Stein open neighborhood $x\in U\subset X$ and a $\mbR$-divisor $\Delta\>0$ on $U$ such that $(U,\Delta)$ is a klt pair.
\end{lemma}
\begin{proof}
By \cite[Remark 2.18]{DHY23} each point $x\in X$ has a relatively compact Stein open neighborhood $U\subset X$ satisfying Property \textbf{P} (see \cite[Definition 2.17]{DHY23}). Thus replacing $X$ by $U$ we may assume that $X$ is a relative compact Stein space satisfying Property \textbf{P}. By standard arguments (cf. the proof of \cite[Proposition 2.43]{KM98}), we may assume that $(X,B+\bbeta _X)$ is a generalized klt pair. Then by \cite[Theorem 2.19(4)]{DHY23} it follows that there is a $\Delta\>0$ such that $(U, \Delta)$ is klt. 
    % Let $\nu:X'\to X$ be a log resolution of $(X,B+\bbeta _X)$. Write $K_{X'}+B'+\bbeta _{X'}=\nu ^*(K_X+B+\beta_X)$, then $-(K_{X'}+B')\equiv _X\bbeta _{X'}$
    % is nef and big over $X$ and hence $-(K_{X'}+B')\equiv_X C'\>0$, where $C'$ is an effective $\mbR$-divisor such that $(X',B'+C')$ is sub-klt, and hence $(X,\Delta:=\nu _*(B'+C'))$ is klt by the Base-point free theorem \cite[Theorem 8.1]{Fuj22} applied to the $\nu$-nef $\mbR$-Cartier divisor $K_{X'}+B'+C'$.
\end{proof}

\subsection{Varieties in Fujiki's class $\mcC$} Recall that an analytic variety $X$ is said to be in Fujiki's class $\mcC$ if $X$ is bimeromorphic to a K\"ahler manifold. In what follows, we will study some basic properties of these varieties which will be useful in running the MMP.

\begin{lemma}\label{lem:pullback-nef}
Let $X$ be a normal compact analytic variety of dimension $3$ in Fujiki's class $\mcC$ and $f:Y\to X$  a proper bimeromorphic morphism from a normal variety $Y$. Then a class $\alpha\in H^{1,1}_{\BC}(X)$ is nef if and only if $f^*\alpha$ in nef on $Y$
\end{lemma}
\begin{proof}
    This is \cite[Lemma 3.13]{HP16}.
\end{proof}

\begin{lemma}\label{lem:f-adjoint-nef}
 Let $(X, B)$ be a $\mbQ$-factorial dlt pair, where $X$ is a compact analytic variety of dimension $\<3$ in Fujiki's class $\mcC$. If $K_X+B$ is pseudo-effective, then $K_X+B$ is nef if and only if $(K_X+B)\cdot C\>0$ for every compact curve $C\subset X$.
\end{lemma}

\begin{proof}
We will prove the dimension $3$ case and leave the surface case as an easy exercise for the reader. The only if part follows easily by passing to a resolution of singularities $f:X'\to X$ such that $X'$ is a K\"ahler manifold. Conversely, assume that $(K_X+B)\cdot C\>0$ for all curves $C\subset X$ and $K_X+B$ is not nef. Then by \cite[Theorem 2.35 and Remark 2.36]{DHP22}, there is a surface $S\subset X$ such that $(K_X+B)|_S$ is not pseudo-effective. Let $K_X+B\num \sum\lambda_iS_i+\beta$ be the Boucksom-Zariski decomposition of $K_X+B$ as in Definition \ref{def:mod-nef}. Then by a similar argument as in the proof of \cite[Lemma 4.1]{HP16} it follows that $S=S_i$ for some $\lambda _i>0$. By adjunction on the normalization $S^\nu\to S$ and the above decomposition of $K_X+B$ it follows that there is an effective $\mbQ$-divisor $\Delta\geq 0$ such that $K_{S^\nu}+\Delta$ is $\mbQ$-Cartier but not pseudo-effective. Let $\tilde S\to S^\nu$ be the minimal resolution of $S^\nu$. Then $K_{\tilde S}$ is also not pseudo-effective, and from surface classification (see \cite[Table 10, page 244]{BHPV04}) it follows that $\kappa(\tilde S)=-\infty$. Since $\tilde S$ is in Fujiki's class $\mcC$, by \cite[Theorem 1.3]{Fuj21}, $\tilde S$ is projective, and thus $S$ is Moishezon. Then we arrive at a contradiction by a similar argument as in the proof of \cite[Corollary 4.2]{HP16}. 
\end{proof}

\begin{lemma}\label{lem:fujiki-to-kahler-surface}
    Let $(S, B+\bbeta )$ be a gdlt pair, where $S$ is a compact analytic surface in Fujiki's class $\mcC$. Then $S$ is a K\"ahler surface with $\mbQ$-factorial rational singularities. 
\end{lemma}

\begin{proof}
Following a similar argument as in \cite[Proposition 2.43]{KM98} we may assume that $(S, B+\bbeta)$ is locally gklt. Thus by \cite[Lemma 2.20]{DHY23}, $S$ has $\mbQ$-factorial rational singularities.
% We begin by showing that $S$ has rational singularities and is $\Q$-factorial. Since the question is local, we may assume that $S$ is Stein and by usual arguments, we may assume that $(S,B+\beta ) $ is gklt. The claim follows from  \cite[Theorem 2.19 and Lemma 2.20]{DHY23}.  
Now let $\nu :S'\to S$ be a resolution of singularities of $S$ such that $S'$ is a K\"ahler manifold. Let $\omega '$ be a K\"ahler class on $S'$. By the Hodge index theorem, the intersection matrix of the exceptional curves of $\nu$ is negative definite. Thus there is a unique $\nu$-exceptional $\mbR$-divisor $E$ such that $\omega'+E\num_S 0$. Since $S$ has rational singularities and it belongs to Fujiki's class $\mcC$, by \cite[Lemma 3.3]{HP16} there is a class $\omega\in H^{1,1}_{\BC}(S)$ such that $\omega'+E=\nu^*\omega$. From the negativity lemma it follows that $E\>0$. Thus $\omega=\nu_*(\omega'+E)=\nu_*\omega'$ is a big class on $S$. 
% then is a unique fix $E$ an exceptional $\R$-divisor so that $\omega ' +E\equiv _S0$ (this is possible because the intersection matrix for the exceptional curves is negative definite). By the negativity lemma $E\geq 0$. By \cite[Lemma 6.1]{HP16} we have  $\omega '+E\equiv \nu ^* \omega $ for some class $\omega \in H^{1,1}_{\rm BC}(S)$.
%      Since $\omega =\nu _* \omega '$, then $\omega$ is big.
If $C\subset S $ is any (compact) curve and $C':=\nu ^{-1}_*C$, then $\omega \cdot C=(\omega '+E)\cdot C'>0$. Then by \cite[Theorem 2.36 and Remark 2.37]{DHP22}, $\omega$ is nef, and by \cite[Theorem 2.30]{DHP22}, $\omega $ is a  K\"ahler class.
\end{proof}

\begin{theorem}[Termination of Flips]\label{thm:termination}
    Let $(X, B)$ be a $\mbQ$-factorial dlt pair, where $X$ is a compact analytic $3$-fold in Fujiki's class $\mcC$. Then every sequence of $(K_X+B)$-flips terminates.
\end{theorem}

\begin{proof}
First we note that Fujino's proof of Special Termination (see \cite{Fuj07}) holds here as the log MMP for surfaces in Fujiki's class $\mcC$ is known again due to Fujino, see \cite{Fuj21}. So we may assume that $(X, B)$ is a klt pair. Next we claim that the proof presented in \cite[Theorem 3.3]{DO23} (which is based on the proof of \cite{Kaw92}) holds in our case without any change. This is because even though $X$ is assumed to be K\"ahler in \cite{DO23}, that property was never used in the proof, and all the necessary MMPs used in that proof are relative MMPs for projective morphisms which are known to exist due to \cite{Nak87} and more recently by \cite{DHP22} and  \cite{Fuj22}. 
\end{proof}

\section{Blowing down analytic spaces to points}\label{sec:point-contractions}
Recall the following result, generalizing \cite{Gra62}, that is claimed in \cite[Proposition 7.4]{HP16}.
\begin{claim}\cite[Proposition 7.4]{HP16}\label{c-HP16} 
Let $X$ be a normal compact complex space and $S$ a $\mbQ$-Cartier prime Weil divisor on $X$ with Cartier index $m$. Suppose that $S$ admits a morphism with connected fibres $f: S \to T$ such that $\mathcal O_S(-mS)$
is $f$-ample. Then there exists a bimeromorphic morphism $\varphi :X\to Y$  to
a normal compact complex space $Y$ such that $\varphi|_S = f$ and $\varphi|_{X\setminus S}$ is an
isomorphism onto $Y \setminus T$.
\end{claim}
The above claim contradicts \cite[Proposition 3]{Fuj74}, and the proof seems to confuse the divisor $S$ with its thickenings (a similar issue occurs in \cite{CHP16}). To be more precise the strategy of \cite[Proposition 7.4]{HP16} is based on \cite{AT84} which also appears to contain an error. One should instead use the corresponding result in \cite[Theorem 2]{Fuj74} which states:
\begin{theorem}\cite[Theorem 2]{Fuj74}\label{t-fuj74}
Let $X$ be a reduced complex space, $A$ an effective Cartier
divisor on $X$ and $f: A\to A'$ a proper surjective morphism of $A$ onto another
complex space $A'$.  Assume that the following conditions hold:
\begin{enumerate}
    \item $\mathcal O _A(-A)$ is $f$-ample, and
    \item $R^1f_*\OO _A(-kA)=0$ for all $k>0$. 
\end{enumerate}
Then there exists a blowing down $F:X\to X'$ such that $F_*\mathcal S_{X,A,f}=\OO _{X'}$, where $\mathcal S_{X,A,f}$ is the coherent sheaf defined by the following short exact sequence
\[ 0\to \mathcal S_{X,A,f}\to \OO _X \to \OO _A/f^{-1}\OO _{A'}\to 0.\] 
\end{theorem}
The proof of \cite[Proposition 7.4]{HP16} begins by checking that $R^1f_* \OO _S(-km'S)=0$ for $k>0$, and given $\mathcal I =\OO _X(-m'S)$ it is claimed that 
\[R^1f_* (\mathcal I^k/\mathcal I^{k+1})\cong R^1f_*\OO _S(-km'S)\] however the right hand side should presumably be replaced by $R^1f_*\OO _{kS}(-km'S)$ (note that the prime divisor $S$ is not assumed to be Cartier in \cite{HP16}, while \cite[Theorem 2]{Fuj74} requires Cartier divisors).
When $\dim A'=0$, one can however deduce the following result of Grauert \cite{Gra62}.

\begin{lemma}\label{l-cont}
Let $X$ be a normal analytic variety and $S\subset X$ an effective $\mbZ$-divisor such that $\Supp(S)$ is compact and has $k$ connected components.  Assume that $S$ is $\mbQ$-Cartier of Cartier index $m>0$ such that $\mcO_S(-mS|_S)$ is ample. Then there exists a proper bimeromorphic morphism $F:X\to Y$ to a normal analytic variety $Y$ such that $F(S)=\{p_1,\ldots, p_k\}\subset Y$ are points and $F|_{X\setminus S}:X\setminus S\to Y\setminus \{p_1,\ldots, p_k\}$ is an isomorphism. 
\end{lemma}
\begin{proof}
Since $\mcO_S(-mS|_S)$ is ample, from \cite[Proposition 1.2.16(i)]{Laz04a} it follows that  $\mcO_{mS}(-mS)$ is ample. Thus, by Serre's vanishing theorem, there is a positive integer $k_0>0$ such that $H^i(mS, \OO _{mS}(-kmS))=0$ for all $i>0$ and $k\geq k_0$.
     Then from the short exact sequences (for $j\in \mbZ^{\>0}$)
    \[0\to \OO _{mS}(-(k+j)mS)\to \OO _{(j+1)mS}(-kmS)\to \OO _{jmS}(-kmS) \to 0\]
    it follows that $H^i((j+1)mS, \OO _{(j+1)mS}(-kmS))=0$ for all $i>0$, $j> 0$ and $k\geq k_0$.
    In particular, $H^i(k_0mS, \OO _{k_0mS}(-jk_0mS))=0$ for all $i,j>0$.
    Replacing $m$ by $k_0m$, we may assume that $H^i(mS, \OO _{mS}(-jmS))=0$ for all $i,j>0$.
    %$H^0(\OO _{k_1mS} )\to H^0(\OO _{k_2mS} )$ is surjective for all $k_1>k_2\geq k_0$. Replacing $m$ by $mk_0$ we have  that $H^0(\OO _{k_1mS} )\to H^0(\OO _{k_2mS} )$ is surjective for all $k_1>k_2\geq 1$.
    
    Suppose $A:=mS$. Then $A$ is Cartier and $\mcO_A(-A|_A)$ is ample since $\mcO_S(-mS|_S)$ ample (see \cite[Proposition 1.2.16(i)]{Laz04a}). Note that $A$ is a projective scheme (of finite type) over $\mbC$ with $k$   connected components. Let $f:A\to \Spec \mbC$ be the structure morphism. Then we have $R^1f_*\OO _A(-kA)=H^1(mS, \OO _{mS}(-kmS))=0$ for all $k>0$, and thus by Theorem \ref{t-fuj74} there is a blowing down $F:X\to X'$ such that $F(A)=\{p_1,\ldots, p_k\}\subset X'$ and $F$ satisfies the other required properties. Since $\Supp(S)=\Supp (A)$, we are done.
    
    % and $f:A\to A':={\rm Spec }H^0(\OO _{mS})$, then $A$ is Cartier, $-A|_A$ is ample and $R^1f_*\OO _A(-kA)=H^i(\OO _{mS}(-kmS))=0$ for all $k>0$. By Theorem  \ref{t-fuj74}  there exists a blowing down $F:X\to X'$, where $F|_A=f$.    
    % Finally, note that $H^0(\OO _{mS})$ is Artinian and in fact a finite dimensional vector space containing $\C\cong H^0(\OO _{S})$ and so the composition $S\to A\to A'$ factors through the map $S\to {\rm Spec } (\C)$ and the inclusion of $0$-dimensional spaces ${\rm Spec } (\C)\to A'$.
   
\end{proof}

As a corollary we will show below that if $\alpha$ is a nef $(1,1)$ class and $S:=\Null(\alpha)$ is an irreducible subvariety of codimension 1 such that $\alpha|_{S^\nu}\num 0$, where $S^\nu\to S$ is the normalization, then we can contract $S$.

%As a corollary to Lemma \ref{l-cont} we prove the existence of 3-fold divisorial contractions to a point.
\begin{corollary}\label{c-1}
Let $X$ be a normal $\Q$-factorial compact K\"ahler variety and 
$\alpha\in H^{1,1}_{\rm BC}(X)$ a nef and big class such that the following hold:
\begin{enumerate}
\item $\alpha ^\perp$ defines an extremal ray $R$ in $\overline{\rm NA}(X)$, 
\item %$R$ contains a curve and all 
the curves in $R$ cover a prime Weil divisor $S\subset X$, and 
\item $\alpha |_{S^\nu}\equiv 0$, where $S^\nu\to S$ is the normalization morphism. 
\end{enumerate}
Then there exists a proper bimeromorphic contraction  $f:X\to Y$ which contracts $S$ to a point and is the identity morphism on the complement of $S$.

%\hl{Note that if $\dim X=3$ and $\alpha =K_X+B+\beta +\omega$ where $(X,B+\beta)$ is a generalized lc pair and $\omega$ is K\"ahler, then $S$ is Moishezon}\footnote{Om: I believe this is a circular statement. I explained it in the proof. CH: You can not trust anything these guys say; even if it's correct it can be stated in a misleading way. They do not prove that if the nef reduction map is trivial, then $S$ is covered by $\alpha$-trivial curves; there could be no such curves, and then $\alpha $ is numerically trivial (in the algebraic sense, which is an empty statement as there are no curves!). At any rate, it is always the case that $-S|_S$ is ample and hence $S$ is projective! See the corrected proof below.}.
\end{corollary}

\begin{proof}
Since $\alpha$ is big, we may write $\alpha=\gamma +\omega$, where $\omega $ is a K\"ahler class and $\gamma$ is a big class. Let $\gamma=D+\beta$ be the Boucksom-Zariski decomposition of $\gamma$ as defined above, where $D$ is an $\mbR$-divisor and $\beta$ is a modified nef class (see Lemma \ref{lem:positive-part}). Then $\beta |_{S^\nu}$ is pseudo-effective.  If $S$ is not contained in the support of $D$, then $\alpha |_{S^\nu} \equiv (D+\beta)|_{S^\nu}+\omega|_{S^\nu}$ is big, which is impossible as $ \alpha |_{S^\nu} \equiv 0$.
Thus $S$ is  contained in the support of $D$. We have $0= \alpha \cdot R = (D+\beta+\omega)\cdot R$ and so $T\cdot R<0$, where $T=\beta$ or $T$ is a component  of $D$. Since $\beta$ is modified nef, $\beta|_{S^\nu}$ is pseudo-effective and so $T$ is a component  of $D$.  
Note that $\NA(X)=\NA(X)_{T\>0}+\NA(X)_{T<0}$, so from a standard argument (e.g. see the proof of \cite[Claim 3.25]{DHY23}) it follows that $\alpha -\delta T$ is positive on $\overline{\rm NA}(X)\setminus\{0\}$ for some $\delta>0$, and hence K\"ahler. But then $(\alpha -\delta T)|_{S^\nu}\equiv -\delta T|_{S^\nu}$ K\"ahler. 
Thus $T=S$, as otherwise $-T|_{S^\nu}$ is an anti-effective class on $S^\nu$ and thus cannot be K\"ahler. In particular,  $-S|_S$ is K\"ahler, and hence $\mcO_S(-mS|_S)$ is an ample line bundle, where $m$ is the Cartier index of $S$ on $X$. Then by Lemma \ref{l-cont} there is a bimeromorphic contraction $f:X\to Y$ such that $f(S)=\{p_0\}$ is a point and $f|_{X\setminus S}: X\setminus S\to Y\setminus \{p_0\}$ is an isomorphism.
%Suppose now that $\dim X=3$ and $\alpha =K_X+B+\beta +\omega$ where $(X,B+\beta)$ is a generalized lc pair, and $\omega$ is K\"ahler.
% Let $(1-s)={\rm mult}_S(B)$ and \[K_{ S^\nu }+\Delta _{S^\nu}+\beta _{S^\nu}=(K_X+B+sS+\beta )|_{S^\nu}\equiv (-\omega  +sS)|_{S^\nu}.\] By adjunction $\Delta _{S^\nu}\geq 0$ and $\beta _{S^\nu}$ is (generalized) nef, thus $K_{S^\nu}$ is negative on curves in $R$ and hence it is not pseudo-effective. But then $\kappa (K_{S'})<0$ for any resolution of $S$. It follows that $S'$ is projective and so $S^\nu$ is Moishezon.
\end{proof}

%In the following we establish the existence of flipping contraction in arbitrary dimensions under certain hypothesis, see Proposition \ref{p-1}. 

A second consequence of Lemma \ref{l-cont} is the existence of flipping contractions for terminal 3-folds (or more generally log canonical 3-folds with isolated rational singularities).

\begin{proposition}\label{p-1}
Let $X$ be a normal compact K\"ahler variety of arbitrary dimension with isolated rational singularities and $\alpha$ a nef and big $(1,1)$ class on $X$. If  $\Null(\alpha)$ is a pure $1$-dimensional analytic subset of $X$, then there is a proper bimeromorphic morphism $f:X\to Y$ such that $\Ex(f)=\Null(\alpha)$ and $f(\Null(\alpha))$ is a finite set of points.
%Suppose that $\alpha$ is nef and big and  $X$ has only isolated singularities. If $\dim {\rm Null}(\alpha)\leq 1$ then there is a contraction morphism $X\to Y$ which contracts $\Null(\alpha)$ to a point and is an isomorphism on the complement of ${\rm Null}(\alpha)$.
\end{proposition}
\begin{proof}
   %We will follow the arguments in the proof of \cite[Theorem 7.12]{HP16}.
   Let $\nu :X'\to X$ be a log resolution of $X$ and $\Null(\alpha)$, which is an isomorphism on the complement of $X_\textsubscript{sing}\cup {\rm Null}(\alpha)$. Then $\nu ^*\alpha $ is also nef and big and hence, by \cite[Theorem 1.1]{CT15}  and Lemma \ref{lem:null-locus} we have
   \begin{equation}\label{eqn:null-non-kahler}
       E_{nK}(\nu^*\alpha )=\Null(\nu ^*\alpha)=\Ex(\nu)\cup \nu^{-1}(\Null(\alpha)).
   \end{equation}
 Since $E_{nK}(\alpha )\subset \nu (E_{nK}(\nu ^*\alpha ))$ by Lemma \ref{lem:non-kahler-locus}, from eqn. \eqref{eqn:null-non-kahler} above it follows that $E_{nK}(\alpha)$ doesn't contain any divisor of $X$. By Demailly's regularization theorem (see \cite[Theorem 3.17(ii)]{Bou04}), there exists a K\"ahler current $T$ in $\nu ^*\alpha$ with analytic singularities such that
$E_+(T ) = E_{nK} (\nu ^* \alpha )$. Let $\mu :X''\to X'$ be a resolution of the singularities of $T$ so that $\mu ^*T=\phi +F$, where $F\>0$ is an effective $\R$-divisor and $\phi$ is a smooth form (see \cite[Definition 2.5.1]{Bou04}); note that $F$ is $\nu\circ\mu$-exceptional.
% \footnote{Om: This doesn't seem to follow without the equality $E_{nK}=\nu(E_{nK}(\nu^*\alpha)$, and without $D$ being exceptional rest of our argument will not work!}
Let $\omega'$ be a K\"ahler form on $X'$ such that $T\geq \ve\omega'$ for some $\ve>0$. Then $\phi\geq \ve\mu^*\omega'$, and thus by \cite[Lemma 2.9]{Bou02ar} we see that there is an effective $\mu$-exceptional $\mbR$-divisor $E\>0$ such that $\phi -\delta E$ is cohomologous to a K\"ahler form for $0<\delta \ll 1$. Then we can write $(\nu \circ \mu) ^*\alpha \equiv \omega '' +G$, where $\omega ''$ is a K\"ahler class and $G\geq 0$ is effective $\mbQ$-divisor such that $\Supp(G)=\Ex(\nu\circ\mu)$.
Let $C^{n}\to C$ be the normalization of $C$, a connected component of $\mathcal N:={\rm Null}(\alpha)$. Then $\alpha |_{C^n}\equiv 0$ (since $C$ is a curve) and so if $D:=(\nu \circ \mu) ^{-1}(C)$, then $D$ is a reduced effective divisor contained in the support of $G$ and
$-G|_D\equiv \omega ''|_D$ is ample. Now from our construction above it follows that $(\nu\circ\mu)(\Ex(\nu\circ\mu))=X_{\sing}\cup\Null(\alpha)$. Thus the components of $G$ which do not map into $\Null(\alpha)$ are contracted to (the singular) points of $X$, and hence $-G|_{{\rm Supp}(G)}$ is ample. So by Lemma \ref{l-cont} there is a contraction $X''\to Y$ which contracts ${\rm Supp}(G)$ to a finite set of points and is an isomorphism on the complement of ${\rm Supp}(G)$. Then by the rigidity lemma (see \cite[Lemma 4.1.13]{BS95}), we obtain the required morphism $f:X\to Y$.
%Since $C$ is a connected component of ${\rm Null}(\alpha)$, recalling our assumptions on the singular locus of $X$, it follows that the divisor $D$ is a connected component of $G$. We then write $kG'=\sum m_iD_i+R$ where $kG'$ is integral (for some $k\in \mathbb N$) $R$ is disjoint from $D$ and the support of $D$ and $\sum m_iD_i$ coincide. But then $-(\sum m_iD_i)|_{\sum m_iD_i}$ is ample and so there is a contraction $X''\to Y$ which contracts $\sum D_i$ to a point and is an isomorphism on the complement of $\sum D_i$. By the rigidity lemma (see eg. \cite[Lemma 4.1.13]{BS95}), \hl{there is a contraction $X\to Y$ which contracts $C$ to a point and is an isomorphism on the complement}\footnote{Om: There is a small problem here, first we need to contract the exceptional divisors of $X''\to X$ which lands in $X_{\sing}$. Fortunately, by Lemma \ref{lem:null-locus} we know that $(\nu\circ \mu)(\Ex(\nu\circ \mu))=\Null(\alpha)\cup X_{\sing}$, so there aren't any other type of exceptional divisors which need to be contracted. }.
\end{proof}

\begin{corollary}
    Let $(X,B)$ be a compact K\"ahler $3$-fold terminal pair. Let $\omega$ be a K\"ahler class on $X$ and $\alpha=K_X+B+\omega$ a nef but not K\"ahler such that $\Null(\alpha)$ is a pure $1$-dimensional analytic subset of $X$ and $\overline{\rm NA}(X)\cap \alpha ^\perp =R$ is an extremal ray. Then the flipping contraction $f:X\to Z$ and the corresponding flip $f^+:X^+\to Z$ both exist.
\end{corollary}
\begin{proof}
    As $(X,B)$ is a terminal $3$-fold, $X$ has only isolated rational singularities.
    Thus Proposition \ref{p-1} applies and hence the flipping contraction $f:X\to Z$ exists.
    By \cite[Theorem 1.3]{DHP22} or \cite[Theorem 1.8]{Fuj22}, we can construct the log canonical model of $(X,B)$ locally over $Z$ and then glue these together to obtain the flip $f^+:X^+\to Z$.
\end{proof}
In order to complete the 3-fold MMP for klt pairs, it is necessary to construct divisorial contractions where $n(\alpha )=1$, i.e. divisorial contractions to a curve,  and flipping contractions in the non-terminal case.
Divisorial contractions to a curve were constructed in \cite{DH20} and flipping contractions are constructed in \cite{CHP16}. Note however that the proof of \cite{CHP16} is somewhat technical as it involves the use of ample sheaves (which are not necessarily vector bundles). In what follows we will give a unified approach to the construction of flipping contractions and divisorial contractions to a curve. This approach relies on first proving that the pl-extremal contractions exist. Our approach seems to be well suited for higher dimensional arguments and settles several important higher dimensional cases. 

%%%%%%%%%%%%%%%%%%%%%%%%%%%%%%%%%%%%%%%%%
\section{The Main Contraction Theorems}\label{sec:main-contractions}
The goal of this section is to prove the contraction Theorems \ref{t-pl} and \ref{thm:contractions}.
\subsection{ Existence of pl-contractions and flips for generalized pairs} In this subsection we will prove Theorem \ref{t-pl}. First we need some preparatory results. The following theorem is a version of the relative Kawamata-Viehweg vanishing theorem which holds for generalized pairs. 
\begin{lemma}\label{l-wdvan}
Let $(X,B+\bbeta)$ be a  gklt pair, $g:X\to T$ a projective morphism of analytic varieties, and $D$ a $\Q$-Cartier $\mbZ$-divisor such that $D-(K_X+B+\bbeta_X)$ is nef and big over $T$. Then $R^ig_* \OO _X(D)=0$ for all $i>0$.
\end{lemma}
\begin{proof}
Since the question is local on the base, we may assume that $T$ is a relatively compact Stein space. Let $\nu :(X',B'+\bbeta_{X'})\to (X,B+\bbeta_X) $ be a log resolution of the gklt pair $(X,B+\bbeta)$ such that $K_{X'}+B'+\bbeta_{X'}=\nu^*(K_X+B+\bbeta_X)$. 
We have $\lceil \nu ^*D\rceil=\nu ^*D+\{-\nu^*D\}$, and we let  $E:=\lfloor B'+\{-\nu ^*D\}\rfloor $, where $\{\cdot\}$ denotes the fractional part. 

We make the following claim.
\begin{claim}\label{clm:divisor-comparison}
  $\lceil \nu ^*D\rceil -E\geq \lfloor \nu ^*D\rfloor $.  
\end{claim}
  %To see this, note that by assumption $-\lfloor B'\rfloor \geq 0$. Let $P$ be a prime divisor on $X'$. If ${\rm mult}_P(\nu ^*D)\in \Z$, then ${\rm mult}_P\{-\nu ^*D\}=0$ and so ${\rm mult}_P(-E)\geq 0$ and the claim is clear. If ${\rm mult}_P(\nu ^*D)\not \in \Z$, then ${\rm mult}_P(\lceil \nu ^*D\rceil -\lfloor \nu ^*D\rfloor =1$ and ${\rm mult}_P(-E)\geq -1$ and the claim is clear). 
\begin{proof}[Proof of Claim \ref{clm:divisor-comparison}]
Let $P$ be any prime Weil divisor on $X'$. First assume that $P$ is not contained in the support of $\{-\nu ^*D\}$. Then ${\mult}_P(-E)={\mult}_P(-\lfloor B'\rfloor )\geq 0$ and ${\mult}_P(\lceil \nu ^*D\rceil) =\mult_P(\nu^*D)={\mult}_P(\lfloor \nu ^*D\rfloor )$, and the claim follows.  Suppose now that $P$ is contained in the support of $\{-\nu ^*D\}$. Then $\mult_P(\{\nu^*D\})=(1-\mult_P(\{-\nu^*D\}))>0$, so that ${\rm mult}_P(\lceil \nu ^*D\rceil) = {\rm mult}_P(\lfloor \nu ^*D\rfloor )+1$. Since  ${\rm mult}_P(-E)={\rm mult}_P(-\lfloor B'+\{-\nu ^*D\}\rfloor )\geq -1$, the required inequality follows.  
\end{proof}

Now recall that $(X, B+\bbeta)$ is a gklt pair and $D$ is a $\mbZ$-divisor. Therefore $E=\lrd B'+\{-\nu^*D\}\rrd$  is $\nu$-exceptional, $\nu _* (\lceil \nu ^*D\rceil-E)=D$ and $\nu_*\lrd \nu^*D\rrd =D$. Since $\nu$ is bimeromorphic, the natural morphism $\nu_*\mcO_{X'}(\lru\nu^*D\rru-E)\to \mcO_X(\nu_*(\lru\nu^*D\rru-E))=\mcO_X(D)$ is injective, and $\nu_*\mcO_{X'}(\lrd\nu^*D\rrd)=\mcO_X(D)$ (see \cite[Lemma 3.2]{Das21}).  By Claim \ref{clm:divisor-comparison} it follows that
\[\OO _{X}(D)\supset \nu _* \OO _{X'}(\lceil \nu ^*D\rceil -E)\supset \nu _* \OO _{X'}(\lfloor \nu ^*D\rfloor)=\OO _X(D).\] 
Therefore we have
\begin{equation}\label{eqn:kv-isomorphism}
    \nu_*\mcO_{X'}(\lru\nu^*D\rru-E)=\mcO_X(D).
\end{equation}
 Now observe that 
 \[\lceil \nu ^*D\rceil -E= K_{X'}+B^*+\bbeta _{X'}+\nu ^*(D-(K_X+B+\bbeta_X)),\]
where $B^*:=B'+\{-\nu ^*D\}-E=\{B'+\{-\nu ^*D\}\}$. But then $(X',B^*)$ is klt and the $\R$-divisor \[\lceil \nu ^*D\rceil -E-(K_{X'}+B^*)\equiv \bbeta _{X'}+\nu ^*(D-(K_X+B+\bbeta _X))\] is  nef and big over $T$, and so, by the relative Kawamata-Viehweg vanishing theorem \cite[Theorem 3.7]{Nak87} (see also \cite[Theorem 5.2]{Fuj22}) for all $i>0$ we have 
\[R^i\nu _* \OO _{X'}(\lceil \nu ^*D\rceil -E)=0,\ \qquad R^i(g\circ \nu )_* \OO _{X'}(\lceil \nu ^*D\rceil -E)=0.\]
By a standard spectral sequence argument and eqn. \eqref{eqn:kv-isomorphism} it follows that 
\[R^ig_*\OO _X(D)=R^ig_* (\nu _* \OO _{X'}(\lceil \nu ^*D\rceil -E))=0\qquad \mbox{ for all } \ i>0.\]

%We now address the general case. We begin by constructing a $\Q$-factorialization. Let $\nu :X'\to X$ be a projective log resolution and write $K_{X'}+B'+\bbeta _{X'}=\nu ^*(K_X+B+\bbeta _X)$, then  
\end{proof}

\begin{lemma}\label{lem:normality-of-plt-center}
    Let $(X, S+B+\bbeta)$ be a gplt pair such that $\lrd S+B\rrd=S$ is irreducible. Then $S$ is normal and $(S, B_S+\bbeta_S)$ is gklt, where $K_S+B+\bbeta_S=(K_X+S+B+\bbeta)|_S$ is defined by adjunction as in Subsection \ref{subsec:adjunction}.
\end{lemma}

\begin{proof}
%\footnote{Om: I noticed after writing this proof that you proved a relative vanishing Lemma \ref{l-wdvan}, which will reduced the length of the following proof by removing the local computation. I will change it later.}
%    Since normality is a local property, it is enough to work locally around each point of $x\in S$. Thus replacing $X$ by a relatively compact Stein open neighborhood of a point $x\in S$ we may assume that $X$ is relative compact Stein space and there is a Stein compact subset $x\in W\subset X$ (see \cite[Renark 2.18]{DHY23}).
Since the question is local on $X$, we may assume that $X$ is a relatively compact Stein space. Let $f:X'\to X$ be a log resolution of the generalized pair $(X, B+\bbeta)$ and 
\[ 
K_{X'}+S'+B'+\bbeta_{X'}=f^*(K_X+S+B+\bbeta_X)+E'
\]
where $S'=f^{-1}_*S, B'\>0, E'\>0, f_*B'=B, f_*E'=0$ and $S'+B'$ and $E'$ do not share any common component.

Then $-S'+\lru E'\rru \num_f K_{X'}+B'+\{-E'\}+\bbeta_{X'}$ and $(X', B'+\{-E'\})$ is a klt pair. Since $[\bbeta_{X'}]\in H^{1,1}_{\BC}(X')$ is a nef over $X$,
% \footnote{Om: Is it ok to define nef class when $X'$ is not compact (since $X$ is not compact here)? The analytic definition does seem to have any issue other than the fact that it will depend on the choice of a K\"ahler form. I think we are fine here as the fibers of $f$ are compact and we only need relative compactness, but we will need to define  nefness on non compact thing first which may not be worth it as }
from Lemma \ref{l-wdvan} it follows that $R^if_*\mcO_{X'}(-S'+\lru E'\rru)=0$ for all $i>0$. Now consider the following exact sequence 
\begin{equation*}
\xymatrixcolsep{3pc}\xymatrix{
0\ar[r] & \mcO_{X'}(\lru E'\rru -S')\ar[r] & \mcO_{X'}(\lru E'\rru)\ar[r] & \mcO_{S'}(\lru E'\rru|_{S'})\ar[r] & 0.
}
\end{equation*}
Since $R^1f_*\mcO_{X'}(-S'+\lru E'\rru)=0$, from the long exact sequence of cohomology we get the surjection 
\[ 
f_*\mcO_{X'}(\lru E'\rru)\surjective f_*\mcO_{S'}(\lru E'\rru|_{S'}).
\]
Since $\lru E'\rru|_{S'}$ is an effective divisor, this map factors as
\[
\mcO_{X}=f_*\mcO_{X'}(\lru E'\rru)\to \mcO_{S}\injective \nu_*\mcO_{S^\nu}\injective f_*\mcO_{S'}(\lru E'\rru|_{S'}),
\]
where $\nu:S^\nu\to S$ is the normalization morphism. Thus we have $\mcO_{S}\cong\nu_*\mcO_{S^\nu}$, and hence $S$ is normal.\\

For the second part by contradiction assume that there is a prime Weil divisor $F$ over $S$ such that $a(F, S, B_S+\bbeta_S)\<-1$. Passing to a sufficiently high resolution $f:X'\to X$ we may assume that we are in the setup of Subsection \ref{subsec:adjunction} and that there is a $f$-exceptional divisor $E$ on $X'$ such that $E\cap S'=F$ and $f(E)=f|_{S'}(F)$. It then follows that $a(E, X, S+B+\bbeta_X)=a(F, S, B_S+\bbeta_S)\<-1$, a contradiction, since $(X, S+B+\bbeta)$ is gplt.   

\end{proof}~\\

\begin{lemma}\label{lem:CM-sheaves}
    Let $(X, B+\beta )$ be a generalized dlt pair, and $D$ a $\mbQ$-Cartier  $\mbZ$-divisor. Then $\mcO_X(D)$ is Cohen-Macaulay.
\end{lemma}

\begin{proof} Since the question is local on $X$, by Lemma \ref{l-dlt} we may assume that $X$ is a relatively compact Stein space and $(X,B)$ is klt.
    Then the proof follows form \cite[Corollary 5.25]{KM98}. We explain the necessary changes needed for analytic varieties. Let $m>0$ be the Cartier index of $D$, then $\mcO_X(-mD)$ is a line bundle on $X$. Since $X$ is Stein, $\mcO_X(-mD)$ is generated by its global sections. Using $\mcO_X(-mD)$ in place of $\mcO_X(mL-mD)$ in the proof of \cite[Corollary 5.25]{KM98}, and using \cite[Theorem 2.21]{DHP22} for the necessary Bertini-type theorems, the arguments of \cite[Corollary 5.25]{KM98} work here.  
    
    % Now passing to a resolution of singularities of $X$ and using \cite[Theorem 2.21]{DHP22}, we see that there is an effective divisor $E\>0$ on $X$ such that $E\sim -mD$ and $(X, B+(1-\frac 1m)E)$ is klt. Let $p:X'\to X$ be the index $m$ covering of $X$ induced by the defining section $E$. If $R$ is the ramification divisor of $p$, then $R=(1-\frac 1m)p^*E$. Thus we have $p^*(K_X+B+(1-\frac 1m)E)=K_{X'}+p^*B+(1-\frac 1m)p^*E-R=K_{X'}+p^*B$; in particular, $(X', p^*B)$ is klt. This implies that $X'$ has rational singularities, and hence $\mcO_{X'}$ is CM. Since by construction, $\mcO_X(D)$ is a direct summand of $p_*\mcO_{X'}$, it is CM by \cite[Proposition 5.4]{KM98}. 
\end{proof}

% \begin{definition}\label{def:relative-canonical}
% Let $X$ be normal variety such that $K_X$ is $\mbQ$-Cartier, i.e. there is a positive integer $m>0$ such that $(\omega^{\otimes m})^{**}$ is a line bundle. Let $f:X'\to X$ be a resolution of singularities of $X$. Then there is an $f$-exceptional Weil divisor $E$ on $X'$ such that $\omega_{X'}^{\otimes m}\cong f^*((\omega_X^{\otimes m})^{**})(E)$. We define the \textit{relative canonical divisor} $K_{X/X'}$ as the following $\mbQ$-divisor on $X'$
% \[
% K_{X/X'}:=\frac{1}{m}E.
% \]
% \end{definition}

Proposition \ref{p-hw} below is an analytic version of a key result from \cite{HW19} (see also \cite{BK23}), and an important step towards the proof of Theorem \ref{t-pl}. First we  fix some notations.

\begin{definition}\label{def:Q-line-bundle}
Let $X$ be a normal analytic variety.
    \begin{enumerate}
    %     \item Let $S\subset X$ be a prime Weil divisor and $S$ is also normal. Then by adjunction, there is an effective $\mbQ$-divisor $\Delta_{\rm diff}\>0$ on $S$, called the \emph{different}, such that the following holds:
    % \[ 
    % (K_X+S)|_S\sim_{\mbQ}K_S+\Delta_{\rm diff}.
    % \]
    \item A non-empty open subset $U\subset X$ called \emph{big}, if $X\setminus U$ is a countable union of locally closed (analytic) subsets of $X$ of codimension at least $2$.

    \item A sheaf $\mcL$ on $X$  is called a $\Q$-{\emph {line bundle}}, if $\mathcal L$ is a reflexive sheaf of rank $1$ and $(\mathcal L^{\otimes m})^{**}$ is locally free for some $m>0$. 
Note that if $D$ is a $\Q$-Cartier Weil divisor, then $\mathcal L=\OO _X(D)$ is a $\Q$-line bundle. On the other hand, if $X$ is Stein, then for any $\Q$-line bundle $\mathcal L$ there exists a $\Q$-Cartier $\mbZ$-divisor $D$ such that $\mathcal L=\OO _X(D)$. Moreover, 
If $\mathcal L$ is a $\Q$-line bundle and $G$ is a $\mbZ$-divisor, then we define  $\mathcal L(G):=(\mathcal L \otimes \OO _X(G))^{**}$ and $\mathcal L(K_X+G):=(\mathcal L \otimes \omega _X\otimes \OO _X(G))^{**}$.
    \end{enumerate}  
\end{definition}
 Recall from Subsection \ref{subsec:adjunction} and Lemma \ref{lem:normality-of-plt-center} that if $(X,S+B+\bbeta )$ is a gplt pair, then $S$ is normal and $(S, B_S+\bbeta_S)$ is klt, where $(K_X+S+B+\bbeta)|_S=K_S+B_S+\bbeta_S$.
\begin{proposition}\label{p-hw} 
    Suppose that $(X,S+B+\bbeta )$ is a relatively compact gplt pair, and $S$ a $\mbQ$-Cartier prime Weil divisor.
    Let $\mcL$ be an rank one reflexive sheaf on $X$ such that $\mathcal L(K_X)$ is a $\mbQ$-line bundle  
    on $X$. Then there exists an effective $\Q$-divisor $0\<\Delta _S \leq B_S$, a positive integer $m>0$, 
and a  reflexive rank 1 sheaf $\mathcal E$ on $S$  such that the following sequence
\begin{equation}\label{eqn:reflexive-exact}
0\to \mathcal L(K_X)\to \mathcal L (K_X+S)\to \mathcal E \to 0
\end{equation}
is exact, and there is a smooth big open subset $S^0\subset S$ such that $\mathcal E|_{S^0}$ and
$(\mathcal L^{\otimes m})|_{S^0}$ are locally free, $m\Delta _S$ is a $\mbZ$-divisor,  and 
\begin{equation}\label{eqn:isom-on-S}
(\mathcal E|_{S^0}) ^{\otimes m}\cong
(\mathcal L^{\otimes m})|_{S^0}\otimes \mathcal O _{S^0}(m(K_S+\Delta _S)).
\end{equation}
\end{proposition}
\begin{proof} This proof follows along the lines of \cite[Proposition 3.1]{HW19}. We include the details as various arguments require more care than their algebraic counterparts.

Let $X=\cup X_i$ be a (finite) cover by relatively compact Stein open subsets $X_i$. Then $\mcL(K_X)|_{X_i}\cong \mcO_{X_i}(G_i)$ for some  $\mbQ$-Cartier $\mbZ$-divisor $G_i$ whose support does not contain $S_i:=S\cap X_i$.  Let $f_i:X'_i\to X_i$ be a log resolution of $(X_i, S_i)$.
We define $\Delta_{S_i}:=(f_i|_{S'_i})_*(\{-f^*_iG_i\}|_{S'_i})$, where $S'_i:=(f_i)^{-1}_{*}S_i$. We first claim that $\Delta_{S_i}$ is independent of the log resolution $f_i$. To this end let $X''_i\to X_i$ be another log resolution of $(X_i, S_i)$ %, and then passing to a hat diagram we may assume that $X''_i\to X_i$ factors 
that factors through $f_i$, and $g_i:X''_i\to X'_i$ the induced morphism. It suffices to show that   \begin{equation}\label{eqn:defining-the-boundary}
    ((f_i\circ g_i)|_{S''_i})_*(\{-(f_i\circ g_i)^*G_i\}|_{S''_i})=(f_i|_{S'_i})_*(\{-f^*_iG_i\}|_{S'_i})
\end{equation} where $S''_i:=(g_i)^{-1}_{*} S'_i$.
Let $P\subset S_i$ be a prime divisor on $S_i$, and $P'$ and $P''$  its strict transforms on $S'_i$ and $S''_i$, respectively. Note that, since $G_i$ is a $\mbZ$-divisor, the support of $\{-f^*_iG_i\}$ and $\{-(f_i\circ g_i)^*G_i\}$ are contained in the exceptional loci of $f_i$ and $f_i\circ g_i$, respectively. 
Since $X'_i$ is smooth, $\lfloor f_i^*G_i\rfloor $ is Cartier and so is $g_i^*(\lfloor f_i^*G_i\rfloor )$. Thus $\{-(f_i\circ g_i)^*G_i\}=g_i^*\{-f_i^*G_i\}$.
But then \[ \{-(f_i\circ g_i)^*G_i\}|_{S''_i}= (g_i^*\{-f_i^*G_i\})|_{S''_i}= (g_i|_{S''_i})^*(\{-f_i^*G_i\}|_{S'_i}).\]
Pushing forward by $g_i|_{S''_i}$ we have  $(g_i|_{S''_i})_*(\{-(f_i\circ g_i)^*G_i|_{S''_i}\})=\{-f_i^*G_i\}|_{S'_i}$. Pushing forward by $f_i|_{S'_i}$ we obtain \eqref{eqn:defining-the-boundary}.

%Now since the the divisor $\Ex(f_i)+S'$ has SNC support, there is at most one component of $E'$ of $\{-f^*_iG_i\}$ such that $E'|_{S'_i}=P'$. Then again, since $\Ex(f_i\circ g_i)+S''$ has SNC support, $E'':=g^{-1}_{i,*}E'$ is the unique component of $\{-(f_i\circ g_i)^*G_i\}$ such that $E''|_{S''_i}=P''$. Moreover, since $g_i|_{S''_i}:S''_i\to S'_i$ is an isomorphism near the generic points of $P''$ and $P'$, it follows that the coefficients of the LHS and RHS of \eqref{eqn:defining-the-boundary} at $P$ are equal. This proves our claim that $\Delta_{S_i}$ is independent of the log resolution $f_i$. 
Next, observe that if $G'_i$ is another $\mbZ$-divisor on $X_i$ such that 
$G_i\sim G'_i$, then $f^*_iG_i\sim f^*_iG'_i$, and thus $\{-f^*G_i\}=\{-f^*_iG'_i\}$. In particular,  $\Delta_{S_i}$ depends only on $\mcL(K_X)|_{X_i}$. Next we will show that the $\Delta_{S_i}$ can be glued together to a unique divisor $\Delta_S$ on $X$. To this end, let $X_{ij}:=X_i\cap X_j$ and $S_{ij}:=S_i\cap S_j$. Then $G_i|_{X_{ij}}\sim G_j|_{X_{ij}}$, and thus $\Delta_{S_i}|_{S_{ij}}=\Delta_{S_j}|_{S_{ij}}$, this follows from the fact that $\Delta_{S_i}$ does not depend on the choice of log resolution, as we proved above. Thus $\{\Delta_{S_i}\}_i$ glue together to give a unique $\mbQ$-divisor $\Delta_S$ on $S$. If $X=\cup X'_i$ is another (finite) cover of $X$ by relatively compact Stein open subsets of $X$ and if $\Delta'_S$ is another $\mbQ$-divisor on $S$ defined as above using the cover $\{X'_i\}$, then from our arguments above it follows that, for any $i$, $\Delta'_S|_{S\cap X'_i\cap X_j}=\Delta_S|_{S\cap X'_i\cap X_j}$ for all $j$. In particular, $\Delta'_S=\Delta_S$, i.e. $\Delta_S$ is independent of the choice of the cover of $X$. \\

%%%%%%%%%%%%%%%%%%%%%%%%%%%%%%%%%%%%%%%%%%%%%%%%%%%%%%%%%%%%%%%%%%%%%%%%%%%%%%%%%%%%%%%%%%%%%%%%%%%%%%%%%%%%%%%%%%%%%%%

% We will first define $\Delta_S$; let $\mu:X'\to X$ be a log resolution of $(X, S)$ given by a sequence of blow ups of smooth centers of codimensions at least $2$, and $S':=\mu^{-1}_*S$. Let $X=\cup U_i$ be a cover by Stein open subsets. We may write $\mathcal L (K_X)|_{U_i}=\OO _{U_i}(G_i)$, where $G_i$ is a Weil divisor and $\Q$-Cartier. We define $U'_i:=\mu ^{-1}(U_i)$, $\mu _i:=\mu |_{U'_i}$, $S_i:=S\cap U_i$, $S'_i:=S'\cap U'_i$, $\nu _i:=\mu _i|_{S'_i}$ and \hl{$\Delta_{S_i}:=\nu _{i,*} (\{-\mu_i^*G_i\}|_{S'_i})$}.
%     Note that the support of $\{-\mu_i^*G_i\}$ is contained in $\Ex(\mu_i)$, and hence it does not contain $S'_i$. Also notice that if $\bar G_i\sim G_i$ on $U_i$, then $\{-\mu_i^*G_i\}=\{-\mu_i^*\bar G_i\}$. Therefore, the definition of $\Delta_{S_i}$ is independent of our choice of $G_i$. It then follows that, on $U_{i,j}:=U_i\cap U_j$, we have $G_i|_{U_{i,j}}\sim G_j|_{U_{i,j}}$ and hence $\Delta_{S_i}|_{S_{i,j}}=\Delta_{S_j}|_{S_{i,j}}$, where $S_{i,j}:=S_i\cap S_j$.
%     Therefore, the $\Delta_{S_i}$ glue together to define a unique $\mbQ$-divisor $\Delta_S$ on $S$ such that $\Delta_S |_{S_i}=\Delta_{S_i}$.

%     \hl{We note that $\Delta_S$ depends on the choice of the log resolution $\mu:X'\to X$.}

Next we will prove rest of the claims under the additional assumption that $X$ is $\Q$-factorial and then deduce the general claim at the end. Assume therefore, for the time being that $X$ is $\mbQ$-factorial.

Observe that the exactness of the sequence \eqref{eqn:reflexive-exact} can be checked locally on $X$. Moreover, notice that in order to produce a smooth big open subset $S^0\subset S$ and an integer $m>0$ satisfying the isomorphism in \eqref{eqn:isom-on-S}, it is enough to work on a finite open cover of $X$, and since $X$ is relatively compact, it can be covered by finitely many relatively compact Stein open subsets. Thus from now on we assume that $X$ itself is a relatively compact Stein space. We note here that $X$ is no longer $\mbQ$-factorial (as $\mbQ$-factoriality is not an analytic local property), however, $K_X, S$ and $B$ are all $\mbQ$-Cartier and $\mcL$ is represented by an effective $\mbQ$-Cartier $\mbZ$-divisor, say $D$, i.e. $\mcL\cong \mcO_X(D)$; we may further assume here that $D$ does not contain $S$ in its support. In particular, we have $\mathcal L(K_X)=\OO _X(K_X+D)$ and $\mathcal L(K_X+S)\cong \OO _X(K_X+S +D)$, where $K_X+D$ and $K_X+S+D$ are both $\mbQ$-Cartier Weil divisors.
Note that since $(X,S+B+\bbeta )$ is gplt, by Lemma \ref{lem:normality-of-plt-center}, $S$ is normal and $(X, S)$ is plt.
Let $\Theta$ be an effective $\mbQ$-divisor on $S$ such that $(K_X+S)|_S=K_S+\Theta$. Then $\Theta\leq B_S$, and thus
by cutting down $X$ by general hyperplanes it follows from \cite[Lemma 3.3]{HW19} that $\Delta_S\leq \Theta\leq B_S$.

Now consider the following exact sequence
\begin{equation}\label{eqn:seq-on-X}
    0\to \OO _X(K_{X} +D)\to \OO _X(K_{X} +S+D)\to \mathcal E\to 0.
\end{equation}
By Lemma \ref{lem:CM-sheaves}, $\OO _X(K_{X} +D)$ and $\OO _X(K_{X} +S+D)$ are Cohen-Macaulay, and hence so is $\mathcal E$ by \cite[Corollary 2.62]{Kol13}. Thus by \cite[Proposition 1.3]{Har80}, $\mcE$ is a reflexive sheaf on $S$ of rank $1$. 

Let $\mu:X'\to X$ be a log resolution of $(X, S+D)$ and $S':=\mu^{-1}_*S$. Let $B_{X'}$ be a $\mbQ$-divisor on $X'$ defined by the equation.
    \[K_{X'}+S'+B _{X'}=\mu ^*(K_X+S).\]
Note that since $\mu: X'\to X$ is projective, $K_{X'}$ is represented by a Weil divisor.
    Moreover, since $(X,S)$ is plt,  $\lfloor B _{X'} \rfloor \leq 0$, and it follows easily that $\lceil \mu ^*D\rceil \geq \lfloor B _{X'} +\mu ^*D \rfloor$.
    Thus 
    \[K_{X'}+S'+\lceil \mu ^*D\rceil\geq \lfloor K_{X'}+S'+B_{X'} +\mu ^*D \rfloor =\lfloor \mu ^*(K_X+S+D) \rfloor , \qquad {\rm and}\]

    \[K_{X'}+\lceil \mu ^*D\rceil\geq \lfloor K_{X'}+B _{X'} +\mu ^*D \rfloor \>\lfloor \mu ^*(K_X+D) \rfloor .\]
%{\color{red} The second inequality follows from a similar computation as above. For this one $K_X$ $\mbQ$-Cartier is needed. }
        It follows (see for example \cite[Lemma 3.2]{Das21}) that 
    \[\mu _* \OO _{X'}(K_{X'}+S'+\lceil \mu ^*D\rceil)=\OO _X(K_{X}+S +D), \qquad {\rm and}\]
    \[\mu _* \OO _{X'}(K_{X'}+\lceil \mu ^*D\rceil)=\OO _X(K_{X} +D).\]
Let $G':=K_{S'}+\lru\mu^*D\rru|_{S'}$. Then we have the following exact sequence
    \[0\to \OO _{X'}(K_{X'}+\lceil \mu ^* D\rceil )\to \OO _{X'}(K_{X'}+S'+\lceil \mu ^* D\rceil )\to  \mcO_{S'}(G')\to 0.\] 
Since $K_{X'}+\lceil \mu ^* D\rceil\equiv _X K_{X'}+\{ -\mu ^*D \} $, where $(X',\{ -\mu ^*D \})$ is klt (as $\mu^*D$ has SNC support), and $\mu$ is bimeromorphic,
    by the relative Kawamata-Viehweg vanishing theorem (see Lemma \ref{l-wdvan}), we have $R^1\mu _* \OO _{X'}(K_{X'}+\lceil \mu ^* D\rceil )=0$. Thus applying $\mu _*$, we obtain an exact sequence
    \begin{equation}\label{eqn:pushforward-seq-on-X}
      0\to \OO _X(K_{X} +D)\to \OO _X(K_{X} +S+D)\to (\mu|_{S'})_*\mcO_{S'}(G') \to 0.  
    \end{equation}
Now let $G:=(\mu|_{S'})_*G'=(\mu|_{S'})_*(K_{S'}+\{-\mu^*D\}|_{S'}+\mu^*D|_{S'})=K_S+\Delta_S+D|_{S}$. We will now proceed to show the isomorphism \eqref{eqn:isom-on-S}. To that end first observe that, since $S$ is normal, it is smooth in codimension 1, and since $\mathcal E$ is reflexive, it is locally free in codimension 1. Hence, there is a big open subset $S^0\subset S$ such that $S^0$ is smooth, $\mu|_{S'}:S'\to S$ is an isomorphism over $S^0$, and $\mathcal E|_{S^0}$ is locally free.
 Recall that  $\mu:X'\to X$ is given by a sequence of blow ups along smooth centers. Let $X^0\subset X$ be the open subset of $X$ obtained by removing the images of the exceptional divisors of $\mu$ whose centers have codimension $\>3$ in $X$. Thus,  over $X^0$, $\mu$ is given by a sequence of blow ups of smooth centers of codimension $2$. Furthermore, removing the closed analytic subset $S\setminus S^0$ of $X$ (which has comdim $\geq 3$ in $X)$ from $X^0$ we may assume that $S^0=S\cap  X^0$. By the classification of plt surface singularities it follows that we may also assume that there is an integer $m>0$ such that $mK_X$, $mS$ and $mD$ are Cartier on $X^0$.
Now from \eqref{eqn:seq-on-X} and \eqref{eqn:pushforward-seq-on-X} it follows that $\mcE\cong (\mu|_{S'})_*\mcO_{S'}(G')$. Moreover, since $\mu|_{S'}$ is an isomorphism over $S^0$ by our construction, the natural morphism $(\mu|_{S'})_*\mcO_{S'}(G')\to  \mcO_S(G)$ is an isomorphism over on $S^0$. Therefore we have $\mcE|_{S^0}\cong \mcO_{S}(G)|_{S^0}$; in particular, $\mcE\cong \mcO_S(G)$, as both sheaves are reflexive and $S^0$ is a big open subset of $S$. Restricting the sequence \eqref{eqn:pushforward-seq-on-X} on $X^0$ we see that $mG=m(K_S+\Delta_S+D|_S)$ is a Cartier divisor on $S^0$; in particular, $m\Delta_S$ is a $\mbZ$-divisor on $S$, as $mD|_S$ is Cartier. Thus
\[
(\mcE|_{S^0})^{\otimes m}\cong (\mcO_{S}(G)|_{S^0})^{\otimes m}\cong (\mcL^{\otimes m})|_{S^0}\otimes \mcO_{S}(m(K_S+\Delta_S))|_{S^0}. 
\]

% It suffices now to check that the natural inclusion \[(\mathcal E|_{S^0}) ^{\otimes m}\to (\mathcal L  \otimes \OO _X(K _X+S))^{\otimes m}|_{S^0}=(\mathcal L ^{\otimes m})|_{S^0}\otimes \OO _{S_0}(m(K_S+\Delta _{\rm diff}))\] induces the isomorphism $(\mathcal E|_{S^0}) ^{\otimes m}\cong \OO _{S^0}(m(K_S+\Delta _S))\otimes \mathcal O _X(mD)|_{S^0}$ on a neighborhood $U_s\subset X^0$ of each point $s\in S^0$.
% Recall that $\mcE=\mu_*\mcE'$, where $\mcE'=\mcO_{S'}(K_{S'}+\lru\mu^*D\rru|_{S'})$. Let $G':=K_{S'}+\lru\mu^*D\rru|_{S'}$; then $G:=(\mu|_{S'})_*G'=(\mu|_{S'})_*(K_{S'}+\{-\mu^*D\}|_{S'}+\mu^*D|_{S'}))\sim_{\mbQ} K_S+\Delta_S+D|_{S}$. Since $\mu|_{S'}$ is an isomorphism over $S^0$ by our construction, the natural morphism $\mcE\to \mcO_S(G)$ is an isomorphism over on $S^0$. Moreover, since $\mcE$ and $\mcO_S(G)$ are both reflexive sheaves and $S^0$ is a big open subset of $S$, it follows that $\mcE\to \mcO_S(G)$ is an isomorphism on $S$. 

% We may assume that $U_s$ is Stein. Cutting by general hyperplanes passing through $s$ we may assume that $\dim X=2$. The claim now follows by a standard computation for surface plt singularities as in \cite[Lemma 3.3]{HW19}.
     
 Finally, we address the general (non-$\Q$-factorial) case. As argued in the $\mbQ$-factorial case, we may assume that $X$ is a relatively compact Stein space. Shrinking $X$ further if necessary, we may assume by   
 \cite[Theorem 2.19]{DHY23} that there is exists a small bimeromorphic morphism $\mu:X'\to X$ such that $X'$ is $\Q$-factorial.
 We let $K_{X'}+S'+B'+\bbeta _{X'}=\mu ^*(K_X+S+B+\bbeta _X)$ and $\mathcal L'(K_{X'}):=(\mu ^* \mathcal L(K_{X}))^{**}$ the reflexive hull of $\mu ^* \mathcal L(K_{X})$. Then $\mathcal L'(K_{X'})$ is a $\Q$-line bundle and 
 $\mathcal L'(K_{X'})\equiv _X0$. By what we have seen above, there is a short exact sequence
 \[0\to \mathcal L'(K_{X'})\to \mathcal L'(K_{X'}+S')\to \mathcal E '\to 0\]
 and a smooth big open subset $S'^0\subset S'$ such that $\mathcal E'|_{S'^0}$ and
$(\mathcal L'^{\otimes m})|_{S'^0}$ are locally free, $m\Delta' _S$ is a $\mbZ$-divisor, $0\leq \Delta' _S\leq\Theta'\leq B'_{S'}$ where $(K_{X'}+S')|_{S'}=K_{S'}+ \Theta'$, $\Theta'\leq B'_{S'}$, and 
\[(\mathcal E'|_{S'^0}) ^{\otimes m}\cong (\mathcal L'^{\otimes m})|_{S'^0}\otimes \mathcal O _{S'^0}(m(K_{S'}+\Delta _{S'})).\]
Since $\mu$ is small, $\mu _*\mathcal L'(K_{X'})=\mathcal L(K_{X})$ and $\mu _*\mathcal L'(K_{X'}+S')=\mathcal L(K_{X}+S)$. Note that, since $X$ is Stein, by a similar argument as in the proof of \cite[Theorem 2.19]{DHY23} it follows that there is a $\mbQ$-divisor $\Gamma'$ on $X'$ such that $\bbeta_{X'}\num_{X'} \Gamma'$, and $(X', S'+B'+\Gamma')$ and $(X, S+B+\Gamma)$ are both plt, where $\Gamma:=\mu_*\Gamma'$. Moreover, we have $K_{X'}+S'+B'+\Gamma'\num \mu^*(K_X+S+B+\Gamma)$. Finally, by a similar argument as in \cite[Proposition 2.43]{KM98}, we see that there is an effective $\mbQ$-divisor $G$ on $X$ such that $(X, G)$ is klt. Let $K_{X'}+G'=\mu^*(K_X+G)$. Then $\mcL'(K_{X'})\num_{X'} 0\num_{X'} \mu^*(K_X+G)\num_{X'}K_{X'}+G'$, where $(X', G')$ is klt. Thus by the relative Kawamata-Viehwewg vanishing theorem $R^1\mu _*\mathcal L'(K_{X'})=0$, and so there is a short exact sequence
 \[0\to \mathcal L(K_{X})\to \mathcal L(K_{X}+S)\to \mu _*\mathcal E '\to 0.\]
 Arguing as above we see that $\mathcal L(K_{X})$ and $\mathcal L(K_{X}+S)$ are both Cohen-Macaulay, and hence so is $\mathcal E:=\mu _*\mathcal E '=\mathcal L(K_{X}+S)|_S$. It follows, as explained above that $\mathcal E$ is reflexive of rank 1. Since $\mu ({\rm Ex}(\mu))$ has codimension $\geq 3$, $S^0:=S\setminus \left(\mu\left(\Ex(\mu)\cup \left(S'\setminus {S'}^0\right)\right)\right)$ is a smooth big open subset of $S$ such that $\mu$ is an isomorphism on a neighborhood of $S^0$. Thus $\mathcal E|_{S^0}$ and
$(\mathcal L^{\otimes m})|_{S^0}$ are locally free, $m\Delta _S:=(\mu|_{S'})_*(m\Delta' _S)$ is a $\mbZ$-divisor, where $0\leq \Delta _S=(\mu|_S)_*\Delta'_S\leq (\mu|_S)_*B'_{S'}=B_S$, $K_{S}+B_S+\bbeta_S=(K_{X}+S+B+\bbeta)|_{S}$, and \[(\mathcal E|_{S^0}) ^{\otimes m}\cong (\mathcal L^{\otimes m})|_{S^0}\otimes \mathcal O _{S^0}(m(K_{S}+\Delta _{S})).\]

\end{proof}

%%%%%%%%%%%%%%%%%%%%%%%%%%%%%%%%%%%%%%%%%%%%%%%%%%%%%%%%%%%%%%%%%%%%%%%%%%%%%%%%%%%%%%%%
%%%%%%%%%%%%%%%%%%%%%%%%%%%%%%%%%%%%%%%%%%%%%%%%%%%%%%%%%%%%%%%%%%%%%%%%%%%%%%%%%%%%%%%%

The following lemma (which is well known to experts) will be useful in our proof of the general contraction Theorem \ref{t-ext}. 
\begin{lemma}\label{lem:partial-q-factorization}
    Let $(X, B+\bbeta)$ be a relatively compact gklt pair and $D_1, D_2,\ldots, D_n$ a collection of  $\mbQ$-divisors. Then there is a projective small bimeromorphic morphism $f:X'\to X$ such that $f^*D_i:=f^{-1}_*D_i$  are  $\mbQ$-Cartier divisors on $X'$ for all $1\<i\<n$.
\end{lemma}

\begin{proof}
Note that by induction it is enough to prove the result for $n=1$. Now replacing $X$ by a small relatively compact Stein neighborhood we may assume by \cite[Theorem 2.16]{DHY23} that there is a projective small $\mbQ$-factorization $f:X'\to X$ and an effective $\mbR$-divisor $\Delta'$ such that $\bbeta_{X'}\num_X \Delta'$, $(X', B'+\Delta')$ and $(X, B+\Delta)$ are both klt, and $K_{X'}+B'+\Delta'=f^*(K_X+B+\Delta)$ , where $\Delta:=f_*\Delta'$. Let $D:=nD_1$ so that $D$ is a $\mbZ$-divisor, and $D':=f^{-1}_*D$; then $D'$ is $\mbQ$-Cartier. Choose $0<\eps\ll 1$ so that $(X, B+\Delta+\eps D)$ is klt. 
Shrinking $X$ further we may assume that $K_X+B+\Delta\sim_{\mbR} 0$. Then $K_{X'}+B'+\Delta'+\eps D'\sim_{\mbR} \eps D'$. Next, we run a $D'$-MMP over $X$ as in \cite[Theorem 1.4]{DHP22} and then replacing $X'$ by the output of this MMP we may assume that $D'$ is nef over $X$. 
By the relative base-point free theorem (see \cite[Theorem 8.1]{Fuj22}) we have that $D'$ is semi-ample over $X$. In particular, $\oplus_{m\geq 0}f_*\mcO_{X'}(mD')$ is a finitely generated $\mcO_X$-algebra. But since $f$ is  small, $f_*\mcO_{X'}(mD')=\mcO_X(mD)$. Thus $\oplus_{m\geq 0}\mcO_{X}(mD)$ is a finitely generated $\mcO_X$-algebra.
Now, since $X$ is a relatively compact space, there is an analytic space $W$ such that $X\subset W$ is an open subset and $\overline X\subset W$. Let $\overline X\subset U\subset W$ be a small neighborhood of $\overline X$ such that $U=\cup U_i$ is a finite union of relatively compact Stein open subsets and $\oplus_{m\>0}\mcO_{U_i}(mD|_{U_i})$ are finitely generated $\mcO_{U_i}$-algebras. Then $Y:={\rm Projan} \oplus_{m\geq 0}\mcO_{U}(mD|_{U})$ is a normal analytic variety such that the projection $\pi:Y\to U$ is a projective small bimeromorphic morphism and $\pi^{-1}_*D$ is $\mbQ$-Cartier (see \cite[Lemma 6.2]{KM98}). Let $X'':=\pi^{-1}(X)$ and $g:=\pi|_{X''}$; then $g:X''\to X$ is a projective small bimeromorphic morphism of normal analytic varieties such that $g^{-1}_*D$ is $\mbQ$-Cartier. This completes our proof.

\end{proof}

The following result is a generalization of Theorem \ref{t-pl} that works for generalized pairs.

\begin{theorem}\label{t-ext}
Let $(X,S+B+\bbeta )$ be a generalized plt pair, where $X$ is a compact analytic variety such that $S$ is $\mathbb Q$-Cartier. Suppose that 
\begin{enumerate}
    \item $\lrd S+B\rrd=S$ is irreducible, 
    \item there is a proper morphism $\pi:S\to T$ such that $\pi _*\OO_S=\OO_T$,
    \item $-(K_X+S+B+\bbeta_X)|_S$ is relatively K\"ahler (over $T$), and 
    \item $-S|_S$
is $\pi$-ample.
\end{enumerate}  Then there exists a bimeromorphic morphism $p:X\to Z$ such that $p|_S=\pi$ and $p|_{X\setminus S}$ is an isomorphism.
\end{theorem}
\begin{proof}
By adjunction (as in Lemma \ref{lem:normality-of-plt-center}) we know that $S$ is normal, $(S,B_S+\bbeta _S)$ is gklt and  $-(K_S+B_S+\bbeta _S) $ is K\"ahler over $T$, where $(K_X+S+B+\bbeta)|_S=K_S+B_S+\bbeta _S$.
%Since $K_X+B$ and $S$ are $\mathbb R$-Cartier, $(X,S+B)$ is plt. 
%We also write $(K_X+S)|_S=K_S+\Delta _S$ where $\Delta _S\leq B_S$. Let $\nu :(X',S'+B'+\beta ')\to (X,S+B+\beta) $ be a log resolution of the generalized pair $(X,B+\beta)$. We write $E=f^*\beta -\beta'$, then $E$ is effective, $\nu$-exceptional and $K_{X'}+S'+B'-E=\nu ^*(K_X+S+B)$ so that $\Delta _S=\mu _* (B'-E)|_{S'}$ and $B_S=\mu_* B'|_S$ where $\mu =\nu |_{S'}:S'\to S$. For any $k\geq 0$  consider the divisor $D_{k}=-(k+1)S-K_X-S-B$, $\Delta' _k =\lceil \nu ^*D_k\rceil -\nu ^*D_k$ and the short exact sequence  \[0\to \OO_{X'}(K_{X'}+ \lceil \nu ^*D_k\rceil -\lfloor \Delta '_k\rfloor )\to \OO_{X'}(K_{X'}+S'+ \lceil \nu ^*D_k\rceil -\lfloor \Delta '_k\rfloor )\to \mathcal O _{S'}(G'_k)\to 0.\]

Choosing $\mcL=\mcO_X(-K_X-(k+1)S)$ for $k\>1$ in eqn. \eqref{eqn:reflexive-exact} we get the following short exact sequences
 \begin{equation}\label{eqn:reflexive-exact-k}
 0\to \OO_X(-(k+1)S)\to \OO_X(-kS)\to \mathcal E_k\to 0.
 \end{equation}
% Since $K_X$ and $K_X+B$ are both $\mbQ$-Cartier, from \cite[Lemma 2.13]{DHY23} it follows that $X$ has klt singularities, and thus by Lemma \ref{lem:CM-sheaves} the first two terms of the above sequence are Cohen-Macaulay. 
Since $S$ is $\Q$-Cartier, $\OO _X(S)$ is a $\Q$-line bundle, and thus by Proposition \ref{p-hw}, $\mcE_k$ is a reflexive rank $1$ sheaf on $S$. %$\mcE_k\cong \OO _S(G_k)$ where $G_k$ is a Weil divisor such that $G_k \num -kS|_S-\Delta_k$, where $\Delta_k\>0$ is an effective $\mbQ$-divisor on $S$ such that $0\leq \Delta _k\leq {\rm Diff}_S(S+B)$.
Let $S_{k}$ be the thickening of $S$ defined by $\OO_{S_k}\cong \OO_X/\OO_X(-kS)$. Then we have the following short exact sequence of sheaves of rings on $S$
\[0\to \mathcal E_k \to  \OO_{S_{k+1}}\to \OO_{S_k}\to 0 .\]
We can then consider the corresponding long exact sequence of sheaves of rings on $T$ obtained by pushforward (note that identifying the topological spaces of the thickenings $|S_k|=|S|$, we may regard $\pi:|S_k|\to |T|$ as a morphism of topological spaces).
We claim that: 
\begin{claim}\label{c-van} 
$R^1\pi _* \mathcal E_k =0$, and so $\pi_* \OO_{S_{k+1}}\to \pi_* \OO_{S_k}$ is surjective for all $k\>1$.
\end{claim}

Grant this for the time being. Since $\pi_* \OO_{S_{k+1}}\to \pi_* \OO_{S_k}$ is surjective for $k\geq 1$,  we have inclusions $T_k\hookrightarrow T_{k+1}$, where $T_k:={\rm Specan} (\pi _* \OO_{S_{k}})$. Since $\pi_* \OO_{S_1}=\pi_*\mcO_S=\mcO_T$, the $T_k$ define thickenings %\footnote{Om: This thing worries me a bit! The point is, since $\pi: S_k\to T$ is only seen a map of topological spaces, $\pi_*\mcO_{S_k}$ not a sheaf of $\mcO_T$-algebras, but only sheaf of rings. So there is no induced morphism $T_k\to T$, right? CH: There is not supposed to be a $\mcO_T$-algebras structure! $T_k$ is not a subspace of $T$, it's the other way around. There is a $\mcO _{S_{k+1}}$ module structure on $\mcO _{S_{k}}$, which is given by the homomorphism $\pi_* \OO_{S_{k+1}}\to \pi_* \OO_{S_k}$.} 
of $T$  with underlying topological space $|T|$. The natural isomorphisms $H^0(\OO _{T_k})\cong H^0(\pi _* \OO_{S_{k}})\cong H^0(\OO _{S_k})$ induce morphisms of complex spaces $\pi _k :S_k\to T_k$ (e.g. see \cite[Exercise II.2.4]{Har77}). The induced composite morphisms $S_k\to T_k\injective T_{k+1}$ factors as $S_k\to S_{k+1}\to T_{k+1}$. We would like to emphasize here that $\pi_k:S_k\to T_k$ are morphisms of non-reduced complex spaces for all $k\geq 1$, not just maps of topological spaces.
Fix $\ell\>1$ such that $A=\ell S=S_\ell$ is Cartier and let $\pi _A:A\to T_\ell$ be the induced morphism (here $\pi_A=\pi _\ell$). 
%Since $\pi _* \OO_{A}\to \pi _* \OO_S$ is surjective, then we have an inclusion $T\to T_l$ such that the induced morphism $S\to T\to T_l$ factors as $S\to A\to T_l$. 

Consider the following commutative diagram 
\begin{equation}
    \xymatrixcolsep{3pc}\xymatrixrowsep{3pc}\xymatrix{
S=A_{\red}\ar[r]^{\vphi}\ar[d]_{\pi} & A\ar[d]^{\pi_A}\\
T=(T_\ell)_{\red}\ar[r]^{\psi} & T_\ell.
    }
\end{equation}
We claim that $\mcO_A(-A)$ is $\pi_A$-ample. Indeed, since $\vphi^*\mcO_A(-A)=\mcO_S(-\ell S|_S)$ is $\pi$-ample, and $\psi$ is a proper finite morphism, $\vphi^*\mcO_A(-A)$ is $(\pi_A\circ\vphi)$-ample. Since $\vphi:A_{\red}\to A$ is the reduction of nilpotent elements of $\mcO_A$, from an argument similar to the proof of \cite[Proposition 1.2.16(i)]{Laz04a} it follows that $\mcO_A(-A)$ is $\pi_A$-ample. 

Arguing as in the proof of Lemma \ref{l-cont}, we will show the following. 
\begin{claim}\label{c-van+}
    Replacing $A$ by a multiple, we may assume that $R^i\pi  _{A,*}\OO _{A}(-kA)=0$ for all $i,k\>1$.
\end{claim} 

By Theorem \ref{t-fuj74}, it then follows that there exists a proper bimeromorphic morphism $p:X\to Z$ such that $p|_{X\setminus A}$ is an isomorphism and $p|_{A}=\pi _A$.
Finally, $p|_S=(p|_A)|_S=\pi _A|_S=\pi$.

\begin{proof}[Proof of Claim \ref{c-van}]To check the claim, note that the vanishing $R^1\pi _* \mathcal E_k =0$ can be checked locally on the base $T$, and hence we may assume that $T$ is a relatively compact Stein space. Since $\pi$ is a projective morphism,  it follows  that $\mathcal E_k=\OO _S(G_k)$ for some Weil divisor $G_k$ on $S$. 
By Proposition \ref{p-hw}, $G_k \sim _\Q -kS|_S-\Delta_k$, where $\Delta_k\>0$ is an effective $\mbQ$-divisor on $S$ such that $0\leq \Delta _k\leq B_S$. (Note that with the notation of Proposition \ref{p-hw}, $\Delta _k=B_S-\Delta _S$.)

If every component of $B_S$ is $\mbQ$-Cartier, then $\Delta_k$ is $\mbQ$-Cartier and so is $G_k$. Moreover, in this case we can also write
\[ G_k\sim _\Q K_{ S}+B_S-\Delta _k+\bbeta _S-(K_S+B_S+\bbeta _S)-kS|_S\]
so that $(S, (B_S-\Delta _k)+\bbeta _S)$ is gklt (by Lemma \ref{lem:normality-of-plt-center}) and $-(K_S+B_S+\bbeta _S)-kS|_S$ is relatively K\"ahler over $T$, and hence the claim now follows from Lemma \ref{l-wdvan}.

Otherwise,  since $(S, B_S+\bbeta_S)$ is a relatively compact generalized klt pair, by Lemma \ref{lem:partial-q-factorization} there is  a projective small bimeromorphic morphsim $q:\tilde S\to S$ such that every component of $q^{-1}_*B_S$ is $\mbQ$-Cartier. Let $\tilde G_k:=q^{-1}_*G_k$, $\tilde B_S:=q^{-1}_*B_S $, $ \tilde \Delta _k:=q^{-1}_*\Delta _k$ so that $K_{\tilde S}+\tilde B_S-\tilde \Delta _k+\bbeta _{\tilde S} =q^*(K_S+B_S-\Delta _k+\bbeta _S)$
and $\tilde{\mathcal E} _k :=\OO _{\tilde S}(\tilde G_k)$. Then
\[\tilde G_k\sim _\Q K_{\tilde S}+\tilde B_S-\tilde \Delta _k+\bbeta _{\tilde S}-q^*(K_S+B_S+\bbeta _S+kS|_S),\] where 
$(\tilde  S, (\tilde B_S-\tilde \Delta _k)+\bbeta _{\tilde S})$ is  gklt and $-q^*(K_S+B_S+\bbeta _S+kS|_S)$ is relatively nef and big over $T$ (hence also over $S$).
By Lemma \ref{l-wdvan}, $R^iq_*\OO _{\tilde S}(\tilde G_k)=0$ and $R^i(\pi \circ q )_*\OO _{\tilde S}(\tilde G_k)=0$ for all $i>0$.
Since $q$ is small, then $q _*\OO _{\tilde S}(\tilde G_k)=\OO _S(G_k)$, and so by an easy spectral sequence argument we have \[R^i\pi_*\OO _S(G_k) =R^i\pi_*( q _*\OO _{\tilde S}(\tilde G_k))=0\qquad \mbox{ for all } \ i>0.\] 
%Since $\pi$ is projective, we then have that $\mathcal E _k=\OO_S(G_k)$ where $G_k$ is a $\Q$-Cartier Weil divisor $\Q$-linearly equivalent to $-kS|_S-\Delta _k$. From the proof of \cite[Prop 3.1]{HW19}, it follows that $\OO _S(G_k)=\nu _*\OO _{S'}(G'_k)$ where \[G'_k\sim K_{S'}+\lceil \nu ^*(-(k+1)S-K_S-\Delta)\rceil\sim _\Q \] where  If $\beta =0$, then the required vanishing follows immediately from \cite[Theorem 5.2]{Fuj22}. In the general case, let $\nu:(S',B_{S'}+\beta _{S'})\to (S,B_S+\beta)$ be a log resolution of the generalized klt pair $(S,B_S+\beta)$.
\end{proof}
\begin{proof}[Proof of Claim \ref{c-van+}] 
Since the statement is local over $T_\ell$, replacing $T_\ell$ by a Stein open subset $U$ and $X$ by a neighborhood of $\pi _A^{-1}(U)$, we may assume that $T_\ell$ is Stein and the statement is equivalent to $H^i(A,\OO _A(-kA))=0$ for all $k>0$. Since $\OO _A(-A)$ is ample, by Serre vanishing, there is an integer $k_0>0$ such that $H^i(A,\OO _A(-kA))=0$ for all $k\geq k_0$ and $i>0$.
For $j\>1$, consider the short exact sequence
  \[0\to \OO _{A}(-(k+j)A)\to \OO _{(j+1)A}(-kA)\to \OO _{jA}(-kA) \to 0.\]
  Proceeding by induction we get that $H^i(\OO _{(j+1)A}(-kA))=0$ for all $k\geq k_0$, $j\>1$, and $i>0$. Replacing $A$ by $k_0A$, the claim now follows.
 \end{proof}
\end{proof}

\begin{remark}\label{r-bpf}
   Assume that we are in the settings of  Theorem \ref{t-ext}, and additionally assume that $X$ is a K\"ahler space. Then the morphism $p$ is projective (since $-S$ is relatively ample over $Z$), and $Z$ is in Fujiki's class $\mathcal C$. Working locally over $Z$, we may pick a K\"ahler form $\omega$ such that $\omega \equiv_Z-(K_X+S+B+\bbeta _X)$.
    Let $\nu :X'\to X$ be a log resolution of $(X,S+B+\bbeta)$ and write $K_{X'}+S'+B'+\bbeta _{X'}=\nu ^*(K_X+S+B+\bbeta _X)$, where $[\bbeta _{X'}]\in H^{1,1}_{\BC}(X')$ is nef. Then $\bbeta _{X'}+\nu ^* \omega $ is numerically equivalent (over $Z$) to a nef and big  $\mbR$-divisor $G'\geq 0$ such that $K_X+S+B+G=\nu _*(K_{X'}+S'+B'+G')$ is plt. By the Base-point free theorem (see \cite[Theorem 8.1]{Fuj22}) $K_Z+S_Z+B_Z+G_Z=\pi _*(K_X+S+B+G)$ is plt and in particular has rational singularities. Then by \cite[Lemma 3.4]{HP16} and \cite[Lemma 8.7]{DHP22}, the image of $p ^*: H^{1,1}_{\BC}(Z)\to H^{1,1}_{\BC}(X)$ is given by 
    \[\Im(p^*)=\{\alpha \in H^{1,1}_{\BC}(X)\ \;:\; \alpha \cdot C=0,\ \mbox{ for all } C\subset X \mbox{ curves s.t. } p(C)=\pt\}.\]
\end{remark}
The next result shows that flips for generalized pairs exist in all dimensions.
\begin{theorem}\label{t-flip}
Let $(X,B+\bbeta)$ be a $\mbQ$-factorial gdlt pair, where $X$ is a compact analytic variety belonging to  Fujiki's class $\mcC$, and $f:X\to Z$ a flipping contraction. Then the flip of $f$ exists. 
\end{theorem}
Recall that a flipping contraction $f:X\to Z$ is a small bimeromorphic morphism such that $\rho (X/Z)=1$ and $-(K_X+B+\bbeta _X)$ is K\"ahler over $Z$. The corresponding  flip (if it exists) is a small bimeromorphic morphism $f^+:X^+\to Z$ such that $\rho (X^+/Z)=1$ and $K_{X^+}+B^++\bbeta _{X^+}$ is is K\"ahler over $Z$, where $B^+$ is the strict transform of $B$.
\begin{proof} 
Replacing $B$ by $(1-\epsilon)B$ for some $0<\epsilon<1$, we may assume that $(X,B+\bbeta )$ is gklt. Let $\nu :X'\to X$ be a log resolution of $(X,B+\bbeta )$ such that $f\circ\nu$ is a projective morphism and $K_{X'}+B'+\bbeta_{X'}=\nu^*(K_X+B+\bbeta_X)$.
Let $\bbeta _{X'}+E=\nu ^*\bbeta _X$, then $E\geq 0$ is an effective $\mbR$-divisor by the negativity lemma.
Since $\rho (X/Z)=1$, it follows that if $K_X+B\not \equiv _Z0$, then $X\to Z$ is projective and $\bbeta _X\equiv _Z \lambda (K_X+B)$
for some $\lambda \in \mathbb R$. For any point $z\in Z$ we fix a neighborhood $z\in W\subset Z$ such that $W$ is relatively compact, Stein and satisfies Property ({\bf P}) (see \cite[Definition 2.17 and Remark 2.18]{DHY23}). Let $X_W=f^{-1}(W)$, then we may assume that $K_{X_W}$
is a $\Q$-Cartier divisor and we let $D:=\lambda (K_{X_W}+B_W)$ where $B_W=B|_{X_W}$. We also let $X'_W=\nu ^{-1}(X_W)$, $B'_W=B'|_{X'_W}$, $E_W=E|_{X'_W}$ and $D':=-\nu _W ^*D-E_W$. Then $D'\equiv_W \bbeta _{X'_W}$ is nef and big over $W$. Since $D'\equiv_W \bbeta _{X'_W}$ is nef and big over $W$, we may assume that $D'\equiv_W \Delta' _W\geq 0$ so that $(X'_W,B'_W+\Delta '_W)$ is sub-klt, and hence $(X_W,B_W+\Delta _W)$ is klt, where $B_W+\Delta_W:=\nu _{W,*}(B'_W+\Delta '_W)$. But then the relative log canonical model $(X^+_W,B^+_W+\Delta^+ _W)$ of $(X_W,B_W+\Delta _W)$ exists by \cite[Theorem 1.3]{DHP22} and \cite[Theorem 1.8]{Fuj22}, and it is the relative log canonical model of $(X_W,B_W+\bbeta _W)$. Since these relative log canonical models are unique by \cite[Lemma 2.12(3)]{DHY23}, they glue together, and we obtain a relative log canonical model $(X^+,B^++\bbeta )$ of $(X,B+\bbeta )$ over $Z$.

Suppose now that $K_X+B \equiv _Z0$. Then $K_Z+B_Z:=f_*(K_X+B)$ is klt and $Z$ has rational singularities.  Fix $z\in Z$, we will construct the flip locally over a neighborhood of $z$. We may assume that a resolution $\nu:X'\to X$ is projective over $Z$ and $X'_z$  is a divisor with simple normal crossing support. Let $f'=f\circ \nu$. Since $Z$ has rational singularities, from the short exact sequence  (\cite[\S 3, eqn. (2)]{HP16})
\[0\to \R \to \OO _{X'}\to \mathcal H _{X'}\to 0\]
it follows that $R^1f'_*\mathcal H _{X'}\to R^2f'_*\R$ is an isomorphism. Note that $ R^2f'_*\R\cong R^2f'_*\mbZ\otimes_{\mbZ}\mbR$ by the universal coefficient theorem, as $f'$ is proper. Now consider the class $[\bbeta  _{X'}]\in H^{1,1}_{\rm BC}(X')\cong H^1(
\mathcal H _{X'})$. By \cite[Lemma 12.1.1]{KM92}, there is a neighborhood $z\in W\subset Z$ and a $\R$-divisor $D'_W$ on $X'_W$ such that  the images of $\bbeta'_W:=\bbeta_{X'}|_{X'_W}$ and $D'_W$ in $R^2f'_*\R$ coincide. But then $\bbeta '_W\equiv _W D'_W$. Shrinking $W$, we may assume that it is relatively compact, Stein and satisfies Property ({\bf P}). We now argue as above and construct the flip $X\dasharrow X^+$ locally over $Z$ and then glue these together.\end{proof}
\begin{remark} Even though it is not clear that the flipping and flipped contractions are projective over $Z$, the proof shows that they are locally projective over $Z$. %In fact, we expect them to be projective over $Z$. Indeed, assuming the generalized mmp, we pick a projective resolution $X'\to Z$ and run the relative mmp with scaling for a very general K\"ahler form on $X'$ following \cite{Kol21}. This should end with $X'$ and show that $X'$ is projective over $Z$.  
\end{remark}
We will also need the following technical alternate version of Theorem \ref{t-flip}.
\begin{theorem}\label{t-flip1}
Let $f:X\to Z$ be a small proper bimeromorphic morphism of normal relatively compact analytic varieties in Fujiki's class $\mcC$. Let $(X,B+\bbeta)$ be a gdlt pair,  $\lfloor B\rfloor$ is $\Q$-Cartier, and $K_X+B+\bbeta _X\equiv _Z -G$, where $G$ is an $\R$-Cartier divisor. Then there exists $f^+:X^+\to Z$ a small bimeromorphic morphism such that $K_{X^+}+B^++\bbeta _{X^+}$ is K\"ahler over $Z$. We say that $f'$ is the flip of $f$. 
\end{theorem}
\begin{proof} 
Replacing $B$ by $B-\epsilon \lfloor B\rfloor$ for some $0<\epsilon<1$, we may assume that $(X,B+\bbeta )$ is gklt.
By standard arguments, it is easy to see that it suffices to construct the flip $f'$ locally over $Z$ and hence we may assume that $Z$ is relatively compact and Stein.

Let $\nu :X'\to X$ be a log resolution of $(X,B+\bbeta )$ such that $f\circ\nu$ is a projective morphism and $K_{X'}+B'+\bbeta_{X'}=\nu^*(K_X+B+\bbeta_X)$.
Since $\bbeta _{X'}\equiv _Z -(K_{X'}+B'+\nu ^*G)$ is nef and big (over $Z$), then $-(K_{X'}+B'+\nu ^*G)\equiv _Z \Delta '$ where 
$(X',B'+\Delta ')$ is sub-klt. If $\Delta =\nu _*\Delta '$, then $(X,B+\Delta)$ is klt and it suffices to set $X^+:={\rm Proj}_Z R(X, K_X+B+\Delta )$ by \cite[Theorem 1.3]{DHP22}.

\end{proof}
%%%%%%%%%%%%%%%%%%%%%%%%%%%%%%%

\begin{proof}[Proof of Theorem \ref{t-pl}]
This follows immediately from Theorems \ref{t-ext} and \ref{t-flip}.
\end{proof}

%%%%%%%%%%%%%%%%%%%%%%%%%%%%%%%%%%%%%%%%%%%%%%%%%%%%%%%%%%%%%%%%%%%%%%
%%%%%%%%%%%%%%%%%%%%%%%%%%%%%%%%%%%%%%%%%%%%%%%%%%%%%%%%%%%%%%%%%%%%%%

\section{Existence of flips and divisorial contractions in dimension 3}
Suppose that $X$ is a normal compact analytic variety in Fujiki's class $\mcC$, and $\alpha\in H^{1,1}_{\rm BC}(X)$ is nef and big but not K\"ahler, then by \cite[Theorem 2.30]{DHP22}, ${\rm Null}(\alpha)$ is non-empty. Recall that ${\rm Null}(\alpha)$ is the union of (positive dimensional) analytic subvarieties $Z\subset X$ such that $\alpha ^{\dim Z}\cdot Z=0$. If $X$ is K\"ahler, then from Lemma \ref{lem:as-locus-is-analytic} and Theorem \ref{thm:non-kahler-equal-null} it follows that $\Null(\alpha)$ is a closed analytic subset of $X$.
Suppose that $\alpha ^\perp \cap \overline{\rm NA}(X)=R$ is an extremal ray, then we say that $R$ is divisorial if $\dim {\rm Null}(\alpha)=\dim X-1$ (i.e. ${\rm Null}(\alpha)$ contains a divisor) and $R$ is of flipping type if $\dim {\rm Null}(\alpha)<\dim X-1$ (i.e. ${\rm Null}(\alpha)$ contains no divisors). 
If $\alpha =[K_X+B+\bbeta _X+\omega]$, where $(X,B+\bbeta )$ is a glc pair and $\omega$ is K\"ahler form, then it is expected that ${\rm Null}(\alpha)$ is given a by the union of all curves $C$ such that $[C]\in R$, i.e. $\alpha\cdot C=0$. 

The result below shows that flipping contractions and flips exist in dimension $3$ for \emph{strongly $\mbQ$-factorial} pairs. 

\begin{definition}\label{def:strongly-q-factorial}
    Let $X$ be a normal compact analytic variety. Then $X$ is called \emph{strongly $\mbQ$-factorial} if for every reflexive sheaf $\mcL$ of rank 1, there is a positive integer $m\in\mbZ^+$ such that $\mcL^{[m]}:=(\mcL^{\otimes m})^{**}$ is a line bundle.  
\end{definition}
Note that complex manifolds are strongly $\mbQ$-factorial, and this property is preserved under the steps of MMP, for more details see Lemmas 2.3--2.5 of \cite{DH20}.

\begin{theorem}\label{t-flipcont}  
Let $(X,B)$ be a strongly $\mbQ$-factorial compact K\"ahler 3-fold klt pair such that 
\begin{enumerate}
   \item $K_X+B$ is pseudo-effective,
    \item $\alpha =[K_X+B+\beta]$ is nef and big for some K\"ahler form $\beta$, and
\item $\alpha ^\perp \cap \overline{\rm NA}(X)=R$ is an extremal ray of flipping type.
\end{enumerate} 
Then the flipping contraction $f:X\to Z$ and the flip $X\dasharrow X^+$ exist, and there is a K\"ahler class $\alpha _Z$ on $Z$ such that $\alpha=f^*\alpha _Z$.
\end{theorem}

\begin{proof}
Let $\nu:X'\to X$ be a log resolution of $(X, B)$. We may assume that $\alpha ':=\nu ^*\alpha =\omega '+G'$, where $G'\geq 0$ is $\nu$-exceptional and $\omega '$ is a K\"ahler class (see the arguments in the proof of Proposition \ref{p-1}). Let $\bomega':=\overline{\omega}'$ (see Subsection \ref{d-gpair}), and $\Delta':=\nu ^{-1}_*B+{\rm Ex}(\nu)$, then $(X',\Delta' )$ is dlt and we may write  $K_{X'}+\Delta '=\nu ^*(K_X+B)+E$, where $E\geq 0$ and $\Supp(E)=\Ex(\nu)$.
Note that $K_{X'}+\Delta '+t\omega '$ is K\"ahler for $t\gg 0$.
Thus
 the following hypothesis are satisfied:
\begin{enumerate}
\item $\alpha '$ is nef and $\omega '$ is a modified K\"ahler class, 
\item $\alpha '=\omega '+G'$, where $G'\geq 0$ is supported on $\lfloor \Delta '\rfloor$,
    \item $K_{X'}+\Delta '+t\omega '+a\alpha '$ is nef for some $t>0$ and $a>0$, 
    \item $(X',\Delta '+t\bomega ')$ is gdlt, and
    \item $X'$ is strongly $\mbQ$-factorial.
\end{enumerate}
Note that even though $\omega'$ is a K\"ahler class here, in the rest of proof we will only use the modified K\"ahler property of $\omega'$, which is preserved by steps of the minimal model program. %because we will run a MMP towards the end of this proof and after first contraction $\omega'$ will become modified K\"ahler but not necessarily K\"ahler.\\

If $S'$ is any component of $\lfloor \Delta '\rfloor$, then by adjunction we can write 
\[K_{S'}+\Delta _{S'}+t\omega' _{S'}=(K_{X'}+\Delta '+t\bomega'_{X'})|_{S'},\]
where $\Delta _{S'}:={\rm Diff}_{S'}(\Delta ')$.
Then $(S',\Delta _{S'})$ is a dlt surface and $\omega' _{S'}={\bomega'_{X'}}|_{S'}$ %$\omega' _{S'}=\overline{({\omega'}|_{S'})}$. Let $\omega_{S'}$ denote the trace of $\bomega'_{S'}$ on $S'$; then $\omega_{S'}$ 
is a big class, since $\bomega'_{X'}$ is a modified K\"ahler class. 

Let $\alpha _{S'}:=\alpha' |_{S'}$ so that $\alpha _{S'}$ is nef and define
 \[\tau _{S'}:={\rm inf}\{s\geq 0\;|\; K_{S'}+\Delta _{S'}+s\omega _{S'}+a\alpha _{S'}\ {\rm is\ nef\ for\ some\ }a>0\}.\]
 \begin{claim}\label{c-S}
     If $\tau _{S'}>0$, then $K_{S'}+\Delta _{S'}+\tau _{S'}\omega _{S'}+a'\alpha _{S'}$ is nef for $a'\gg 0$ and there is a $(K_{S'}+\Delta _{S'})$-negative extremal ray $R_{S'}$ of $\overline{\rm NA}(S')$ such that
     \[(K_{S'}+\Delta _{S'}+\tau _{S'}\omega _{S'})\cdot R_{S'}=\alpha _{S'} \cdot R_{S'}=0.\]
 \end{claim}
 \begin{proof}
     Since $\omega _{S'}$ is big and $\tau _{S'}>0$, then there are finitely many $(K_{S'}+\Delta _{S'}+\frac {\tau _{S'}}2\omega _{S'})$-negative extremal rays in $\overline{\rm NA}(S')$ (see \cite[Corollary 2.32]{DHY23}). Let $\Sigma _1,\ldots , \Sigma _k$ be curves that span those rays % which are also $\alpha _{S'}$-trivial. Suppose that $\Sigma _i\cdot \alpha _{S'}=0$ for $1\leq i\leq r$ and    $\Sigma _i\cdot \alpha _{S'}>0$ for $ $r+1\leq i \leq k$.
     and let \[\Lambda:=\{i\;:\; \alpha_{S'}\cdot \Sigma_i=0 \mbox{ for some } 1\<i\<k\}.\] We define
     \[\sigma _{S'}:={\rm min}\left\{s\geq \frac{\tau _{S'}}2\;|\;(K_{S'}+\Delta _{S'}+s\omega _{S'})\cdot \Sigma _i\geq 0\mbox{ for all } i\in\Lambda\right\}.\] 
    If $\Lambda=\emptyset$, i.e. $\alpha _{S'}\cdot \Sigma _i> 0$ for $i=1,\ldots, k$, then we define $\sigma _{S'}:=\frac{\tau _{S'}}2$.
     % \footnote{Om: This does not make sense, but we do need to give an argument what happens when $\alpha_{S'}\cdot\Sigma_i>0$ for all $i\>1$, i.e. $\Lambda=\emptyset$. Once we have this argument, I think the rest of the proof of this claim maybe shortened by noting that $\omega_{S'}$ is a nef class by \cite[Cor. 2.27]{DHY23}, so $\Sigma_i$'s are all $K_{S'}+\Delta_{S'}$-negative curves. CH: I think it makes sense. Wee are defining it this way if $\alpha _{S'}\cdot \Sigma _i\ne 0$ for $i=1,\ldots, k$ and then show this is impossible.}.
     Clearly 
     $\sigma _{S'}\leq \tau _{S'}$ and if $\sigma _{S'}= \tau _{S'}$, then there is a curve $\Sigma _i$ for some $i\in\Lambda$ such that $(K_{S'}+\Delta _{S'}+\tau _{S'}\omega _{S'})\cdot \Sigma _i=\alpha _{S'}\cdot \Sigma _i=0$.
     % Since $K_{S'}+\Delta _{S'}+t\omega _{S'}+a\alpha_{S'}$ is nef and $t\geq \tau _{S'}$, 
     Since $\Sigma_i$ is $(K_{S'}+\Delta_{S'}+\frac{\tau_{S'}}{2}\omega_{S'})$-negative, it follows that $(K_{S'}+\Delta_{S'})\cdot\Sigma_i<0$ and the claim holds. Therefore, it suffices to show that $\sigma _{S'}< \tau _{S'}$ is impossible. To the contrary assume that $\sigma _{S'}< \tau _{S'}$. Then for $a'\gg 0$ we have \[(K_{S'}+\Delta _{S'}+\sigma _{S'}\omega _{S'}+a'\alpha _{S'})\cdot \Sigma _i\geq 0\qquad {\rm for}\ i=1,\ldots, k.\]
     For $0<\epsilon \ll 1$ and $a'\gg 0$, we claim that $K_{S'}+\Delta _{S'}+(\tau _{S'}-\epsilon )\omega _{S'}+a'\alpha _{S'}$ is non-negative on \[\overline {\rm NA}(S')_{K_{S'}+\Delta _{S'}+\frac {\tau _{S'}}2\omega _{S'}\geq 0}.\] 
     Indeed, $\alpha _{S'}$ is nef and $K_{S'}+\Delta _{S'}+(\tau _{S'}-\epsilon )\omega _{S'}+a'\alpha _{S'}$ is a convex linear combination of $K_{S'}+\Delta _{S'}+\frac {\tau _{S'}}2\omega _{S'}$ and the nef classes $K_{S'}+\Delta _{S'}+t\omega _{S'}+a\alpha _{S'}$ and $\alpha _{S'}$. More specifically,  if $\eta _{s,a}:= K_{S'}+\Delta _{S'}+s\omega _{S'}+a\alpha _{S'}$ and $\lambda:=\frac {t-\tau _{S'}+\epsilon}{t-\tau _{S'}/2}$, then we can write \[\eta_{\tau _{S'}-\epsilon ,a'} =\lambda  \eta_{{\tau _{S'}}/2,0}+(1-\lambda) \eta_{t,a} +(a'-(1-\lambda)a)\alpha _{S'}.\]
     Note that $\lambda >0$,  $1-\lambda=\frac {\tau _{S'}/2-\epsilon}{t-\tau _{S'}/2}>0$, since $0<\eps\ll 1$, and $a'-(1-\lambda)a \geq 0$ for $a'\gg 0$.
     Since  $K_{S'}+\Delta _{S'}+(\tau _{S'}-\epsilon )\omega _{S'}+a'\alpha _{S'}$ is non-negative on the $\Sigma _i$ for all $1\<i\<k$, it is non-negative on 
     \[ \overline {\rm NA}(S')=\overline {\rm NA}(S')_{K_{S'}+\Delta _{S'}+\frac {\tau _{S'}}2\omega _{S'}\geq 0}+\sum _{i=1}^k \mathbb R^+[\Sigma _i].\]
     Thus $K_{S'}+\Delta _{S'}+(\tau _{S'}-\epsilon )\omega _{S'}+a'\alpha _{S'}$ is nef, which is a contradiction to the definition of $\tau_{S'}$, and the claim follows.
     Note that  $K_{S'}+\Delta _{S'}+\tau _{S'}\omega _{S'}+a'\alpha _{S'}$ is nef for all $a'\gg 0$, as $\alpha_{S'}$ is nef.
 \end{proof}~\\

Let $\tau:={\rm max}\{\tau _{S'}\}$ as $S'$ runs through all the components of  $\lfloor \Delta '\rfloor$. Then $K_{S'}+\Delta _{S'}+\tau\omega _{S'}+a'\alpha _{S'}$ is nef for each component $S'$ of  $\lfloor \Delta '\rfloor$ and $a'\gg 0$, as $K_{S'}+\Delta_{S'}+t\omega_{S'}+a'\alpha_{S'}$ is nef and $\tau_{S'}\<\tau\<t$ for every component $S'$ of $\lrd\Delta'\rrd$.

\begin{claim} If $\tau >0$, then $K_{X'}+\Delta '+\tau \omega'+a\alpha'$ is nef for some $a>0$.
%We have $\tau ={\rm inf}\{ s\geq 0|K_{X'}+\Delta '+s\omega'+a\alpha'\ {\rm is\ nef\ for\ some\ }a>0\}$.
\end{claim}
\begin{proof}
    Suppose that $K_{X'}+\Delta '+\tau \omega'+a'\alpha'$ is not nef for a given $a'\gg 0$, then by \cite[Theorem 2.36 and Remark 2.37]{DHP22} there is a subvariety $Z'\subset X'$ such that $(K_{X'}+\Delta '+\tau \omega'+a'\alpha ')|_{Z'}$ is not pseudo-effective. %Clearly $\lambda \leq \tau <t$ and so $t-\lambda>0$. 
Since
\[K_{X'}+\Delta '+\tau \omega'+a'\alpha ' \equiv  K_{X'}+\Delta '+t \omega'+(a'-t+\tau)\alpha '+(t-\tau)G' \]
and $K_{X'}+\Delta '+t \omega'+(a'-t+\tau)\alpha '$ is nef for $a'\gg 0$,
 then $Z'$ is contained in the support of $G'$ and hence in a component $S'$ of $\lfloor \Delta '\rfloor$ and this contradicts what we have shown above (see Claim \ref{c-S}).
\end{proof}

If $\tau >0$, then from Claim \ref{c-S} it follows that there is a component $S'$ of $\lfloor \Delta '\rfloor $ and a $(K_{S'}+\Delta _{S'})$-negative  extremal ray $R_i=\mbR^+\cdot[\Sigma_i]$ such that  $(K_{X'}+\Delta '+\tau \omega')\cdot R_i=\alpha'\cdot R_i=0$. Let $\pi : S'\to T$ denote the corresponding contraction. Since $\omega'\cdot R_i>0$ and  $\alpha '\cdot R_i=(\omega '+G ')\cdot R_i=0$, it follows that $G'\cdot \Sigma_i<0$. Thus $\Sigma_i$ is contained in a component of $G'$, say $S''$ such that $S''\cdot\Sigma_i<0$. Since $\Supp(G')=\lrd\Delta'\rrd$, $\Sigma_i$ is contained in a $(K_{S''}+\Delta_{S''})$-negative extremal face of $\NA(S'')$, where $K_{S''}+\Delta_{S''}:=(K_{X'}+\Delta')|_{S''}$. This face contains a negative extremal ray $R''_i=\mbR^+\cdot[\Sigma''_i]$ such that $S''\cdot R''_i<0$. Thus replacing $S'$ by $S''$ and $R_i$ by $R''_i$ we may assume that $S'\cdot R_i<0$, and hence $-S'|_{S'}$ is $\pi$-ample.
%and \hl{so we may assume that $S'\cdot R_i<0$}\footnote{Om: There may be a small problem here!}
 %Consider the induced map $\phi:\overline{\rm NA}(S')\to \overline{\rm NA}(X')$. Then $\phi (R_i)$ is a $K_{X'}+\Delta '$-negative extremal ray of $\phi(\overline{\rm NA}(S'))$ and so $F=\phi ^{-1}(\phi (R_i))$ is a $K_{S'}+\Delta _{S'}$-negative  extremal face of $\overline{\rm NA}(S')$,  and so there is a corresponding contraction $\pi : S'\to T$. Since $\omega '\cdot R>0$ and $\alpha '\cdot R=(\omega '+G ')\cdot R<0$, it follows that $G'\cdot R<0$ and so we may assume that $S'\cdot R<0$ and hence $-S'|_{S'}$ is $\pi$-ample. 
Now write $\Delta'=S'+\Delta''$ so that $S'\not\subset \Supp(\Delta'')$. Since $(X', \Delta'+\bomega')$ is gdlt, for $0<\eps\ll 1$ we have $(X', S'+(1-\eps)\Delta''+\bomega')$ is gplt and $(K_{X'}+S'+(1-\eps)\Delta''+\omega')\cdot R_i<0$. Thus by Theorem \ref{t-ext}, there exists 
a bimeromorphic morphism $p : X' \to Z'$ such that $p|_{S'} = \pi$ and $p|_{X'\setminus S'}$ is an isomorphism.
\begin{claim}\label{clm:contraction}
   The morphism $p$ is either a divisorial or flipping contraction, and in the latter case the corresponding flip exists. 
\end{claim}
\begin{proof}
    Since $-(K_{X'}+\Delta ')$ is $p$-ample, $Z'$ has rational singularities by \cite[Lemma 2.44]{DH20}. Let $\gamma \in H^{1,1}_{\rm BC}(X')$ such that $\gamma \cdot R_i=0$, then $\gamma \cdot C=0$ for any $p$-exceptional curve. This is because all $p$-exceptional curves are contained in $S'$ and $\rho (S'/T)=1$.
    By \cite[Lemma 3.3]{HP16}, there is a class $\gamma _{Z'}\in H^{1,1}_{\rm BC}(Z')$ such that $p^*\gamma _{Z'}=\gamma$. It follows that $\rho (X'/Z')=1$.
    This shows that 
 $p$ is either a divisorial or flipping contraction. By Theorem \ref{t-flip}, if it is a flipping contraction, then the corresponding flip exists. %\begin{claim}    $Z'$ is a K\"ahler variety.   \end{claim}   \begin{proof}  \end{proof} Note that as $-S'$ is $p$-ample, $p$ is a projective morphism and therefore 
 %To that end, let $V\subset Z'$ be a positive dimensional subvariety of $Z'$.  {\color{red}  Note that $\alpha'=p^*\alpha_{Z'}$. If $Z'\not\subset p(\Ex(p))$, let $V'$ be the strict transform of $V$. Then $(\alpha_{Z'})^{\dim V}\cdot V={\alpha'}^{\dim V'}\cdot V'>0$. If $V\subset p(\Ex(p))$, then $\dim V=1$ (this possible only when $p$ is a divisorial contraction). In this case $p^{-1}V\to V$ is projective morphism, so there is a curve $V'\subset p^{-1}V$ such that $f'(V')=V$. Then again $d((\alpha_{Z'})^{\dim V}\cdot V)={\alpha'}^{\dim V'}\cdot V'>0$, since $V'$ is not contracted by $p$; here $d$ is the generic degree of the finite morphism $V'\to V$. Thus from \cite[Lemma 6.3]{DH20} or more generally from \cite[Theorem 2.30]{DHP22} it follows that $\alpha_{Z'}$ is a K\"ahler class.  }\footnote{Om: These is not correct, $\Null(\alpha')$ is larger than $\Ex(p)$. We need argue that $R_i$ is a negative extremal ray of $\NA(X')$, so there is a nef class $\beta'$ on $X'$ whose null locus is $\Ex(p)$ so that $\beta'=p^*\alpha_{Z'}$. CH: It should be "easy" to fix since the morphism is projective and so the relative cone theorem holds.}
 \end{proof}

We may then replace $X'$ by the image of this flip or divisorial contraction. Observe that, at this stage $X'$ is in Fujiki's class $\mcC$ and not necessarily K\"ahler, however, by Lemma \ref{lem:fujiki-to-kahler-surface} the components of $\lrd\Delta'\rrd$ are still K\"ahler surfaces with $\mbQ$-factorial rational singularities. So we can repeat the above arguments.

\begin{claim}\label{c-m} Fix $0<\epsilon  \ll 1$. After finitely many steps, we obtain a bimeromorphic map $\phi: X'\dasharrow X^+$ such that for any $0<t\leq \epsilon$ \[K_{X^+}+\Delta ^++t\omega ^+ +a\alpha ^+=\phi _*(K_{X'}+\Delta '+t\omega '+a\alpha ')\] is nef for some $a>0$ (depending on $t$).
\end{claim} 
\begin{proof}
    Properties (1-5) stated at the beginning of our proof continue to be satisfied after each flip or divisorial contraction. So we may repeat the procedure obtaining an $\alpha '$-trivial, $(K_{X'}+\Delta ')$-minimal model program. Since 
    termination of flips holds by Theorem \ref{thm:termination}, after finitely many steps we may assume that $\tau=0$. Thus we obtain an $\alpha '$-trivial $(K_{X'}+\Delta ')$-MMP $X'\dasharrow X^+$ such that $K_{X^+}+\Delta ^++\epsilon \omega ^++a\alpha ^+$ is nef for some $\epsilon >0$, $a>0$ 
    and \[{\rm inf}\{ s\geq 0\;|\; K_{X^+}+\Delta ^++s\omega^++a\alpha^+\ {\rm is\ nef\ for\ some\ }a>0\}=0.\]
   The claim now follows easily by taking convex linear combinations. 
\end{proof}

Let $\psi:X\dasharrow X^+$ and $\phi:X'\dasharrow X^+$ be the induced bimeromorphic map,  $U:=X\setminus\Null(\alpha)$ and $X'_{U}:=\nu^{-1}(U)$. Next we will show that every step of this MMP is vertical over $U$.

% {\color{red} 
% Let $X'=X_0\bir X^1\bir\cdots \bir X^n=X^+$ be the $\alpha'$-trivial MMP constructed above and $f^i:X^i\to Z^i$ be the $i$-th extremal contraction of this MMP and $X^i\bir X^{i+1}$ is the corresponding divisorial contraction or flip (in case of divisorial contraction $X^{i+1}\cong Z^i$). Note that $f^i$ is defined over $X$, since every $f^i$-exceptional curve is $\nu^i$-exceptional; let $\nu^i:X^i\to X$ be induced morphism. Suppose that $f^i_U:=f^i|_{X^i_U}$ and $g^i_U:=g^i|_{X^{i+1}_U}$, where $g^i:X^{i+1}\to Z^i$ is the induced morphism.

% \begin{claim}
%     $f^i_U$ and $g^i_U$ are proper morphisms for all $0\<i\<n$.
% \end{claim}

% \begin{proof}
%     Let $h^i:Z^i\to X$ be the induced morphism. Then he claim follows from the commutativity of morphisms: $\nu^i=h^i\circ f^i$ and $\nu^{i+1}=h^i\circ g^i$.
% \end{proof}

\begin{claim}\label{c-p}
The map $\phi _U=:\phi |_{X'_U}$ is a proper bimeromorphic map defined over $U$, i.e. there exist proper bimeromorphic morphisms $X'_U\to U$ and $X^+_U\to U$ commuting with $\phi_U$.
\end{claim} 
\begin{proof}
  
Suppose that $X'=:X^0\dasharrow X^1\dasharrow \ldots \dasharrow X^n=:X^+$ is the MMP obtained above and $\phi ^i:X'\dasharrow X^i$ are the induced bimeromorphic maps. Proceeding by induction, it suffices to show that if
$\phi ^i_U:=\phi ^i|_{X'_U}:X'_U\dasharrow X^i_U$ is a proper bimeromorphic map defined over $U$, then so is $\phi ^{i+1}_U$. We will denote by $\nu ^i_U:X^i_U\to U$ the corresponding proper morphism.
Since $X^i\dasharrow X^{i+1}$ is an $\alpha ^i$-trivial flip or divisorial contraction,  if $f^i: X^i\to Z^i$ is the corresponding contraction, and $C$ is a contracted curve, then $\alpha ^i\cdot C=0$. Suppose that $C_U:=C\cap X^i_U\ne \emptyset$. 
Let $p:W\to X'$, $q:W\to X^i$ be the normalization of the graph of $\phi ^i$.
If $C_U$ is not $\nu ^i_U$ vertical, then let $C'\subset W$ be a curve dominating $C$ so that $q_*C'=dC$ for some $d>0$. 
Then
$\nu _* p_*C'\ne 0$ and so $\nu _* p_*C'\cdot \alpha \ne 0$ which is a contradiction since then
\[0=C\cdot \alpha ^i=\frac 1dC'\cdot q^* \alpha ^i=p_*(\frac 1dC')\cdot  \alpha '=\nu _* p_*(\frac 1dC')\cdot \alpha \ne 0 .\]  Thus every such $C_U$  is $\nu ^i_U$ vertical.  It follows that $f^i_U:=f^i|_{X^i_U}$ is a (proper bimeromorphic) morphism over $U$ and hence $X^{i+1}_U\to U$ is a proper bimeromorphic map over $U$. 
  
\end{proof}

   \begin{claim}\label{c-s}
       The divisors contracted by $\phi: X'\dasharrow X^+$ coincide with the set of $\nu: X'\to X$ exceptional divisors and hence $\psi:X\dasharrow X^+$ is a small bimeromorphic map of  strongly $\mbQ$-factorial varieties.
   \end{claim}
   \begin{proof} 
   We will first show that ${\rm Supp} N(K_{X'}+\Delta '+a\alpha ')={\rm Ex}(\nu)$ for any $a\gg 0$. Recall that for any pseudo-effective $(1,1)$ class $\gamma$, $N(\gamma)$ is defined as the negative part of the Boucksom-Zariski decomposition of $\gamma$, see Definition \ref{def:mod-nef}.\\
   Now since $\alpha$ is big, $K_X+B+a\alpha$ is big for $a\gg 0$. Moreover,  since $K_{X'}+\Delta '+a\alpha ' =\nu ^*(K_X+B+a\alpha)+E$, where ${\rm Supp}(E)={\rm Ex}(\nu)$, from \cite[Lemma A.5]{DHY23} it follows that ${\rm Supp} N(K_{X'}+\Delta '+a\alpha ' )\supset {\rm Ex}(\nu)$.
Since $\alpha '=G'+\omega '$ where $\omega '$ is K\"ahler and $\Supp(G')={\rm Ex}(\nu)$,  $\omega ^*:=\omega '+\frac 1 a (K_{X'}+\Delta ')$ is also K\"ahler for $a\gg 0$ and so $K_{X'}+\Delta '+a\alpha ' \equiv a\omega ^*+aG'$. Thus ${\rm Supp} N(K_{X'}+\Delta '+a\alpha ')\subset {\rm Supp} (N(G'))\subset {\rm Ex}(\nu)$ and hence ${\rm Supp} N(K_{X'}+\Delta '+a\alpha ')={\rm Ex}(\nu)$ for any $a\gg 0$.
Since $\omega'$ is a modified K\"ahler class, ${\rm Supp} N(K_{X'}+\Delta '+a\alpha '+\epsilon \omega ')={\rm Ex}(\nu)$ for all $0<\epsilon \ll 1$ (e.g. see the proof of \cite[Claim 3.18]{DHY23}). 
     Since $\phi: X'\dasharrow X^+$ is  a $(K_{X'}+\Delta '+\epsilon \omega '+a\alpha ')$-MMP, then $K_{X^+}+\Delta ^++\epsilon \omega ^+ +a\alpha ^+$ is nef, and by \cite[Theorem A.11]{DHY23}, $\phi$ contracts ${\rm Ex}(\nu)$.
   \end{proof}
By \cite[Lemma 2.32]{DH20}, $\psi_*:H^{1,1}_{\BC}(X)\to H^{1,1}_{\BC}(X^+)$ is an isomorphism. Let $R^\perp\subset N^1(X)$ be the codimension 1 subspace of classes $\gamma \in N^1(X)$ such that $\gamma \cdot R=0$. Let  $W$ be the normalization of the graph of $\psi:X\bir X^+$ and $p:W\to X$ and $q:W\to X^+$ the induced morphisms. Since $\psi$  is not a morphism, then there is a curve $\tilde C\subset W$ such that $p_*\tilde C=0$ and $C^+=q_* \tilde C\ne 0$. It follows that $\alpha ^+\cdot C^+=0$. Let $R^+\subset N_1(X^+)$ be the ray spanned by $C^+$.
If $\gamma \in R^\perp$, then $\nu ^* \gamma \cdot \Sigma \equiv 0$ for any curve $\Sigma$ such that $\alpha ' \cdot \Sigma =0$ where $\alpha '=\nu ^*\alpha $.
Since the minimal model program $X'\dasharrow X^+$ is $\alpha'$-trivial, i.e. it only contracts $\alpha'$-trivial curves, it follows that it is also $\gamma ':=\nu ^* \gamma $-trivial and hence that $\gamma ^+\cdot C^+=0$ where $\gamma ^+=\phi _* \gamma '\in H^{1,1}_{\rm BC}(X^+)$.
Therefore, $\psi _*(R^\perp)=(R^+)^\perp$, and hence every $p$-exceptional curve $\tilde C'$ satisfies $q_*\tilde C'\in R^+$.
Since $\psi_*$ is an isomorphism and $(K_X+B)\cdot R\ne 0$, then $(K_{X^+}+B^+)\cdot R^+\ne 0$. Suppose that $(K_{X^+}+B^+)\cdot R^+< 0$, then 
$(K_{X^+}+B^++t\omega^+)\cdot R^+< 0$ for $0<t\ll 1$. Since 
$\alpha ^+\cdot R^+=0$, $(K_{X^+}+B^++t\omega^++a\alpha ^+)\cdot R^+< 0$ for any $a>0$, contradicting Claim \ref{c-m}. Therefore $(K_{X^+}+B^+)\cdot R^+>0$.

Let $\eta :=\alpha +\delta (K_{X}+B)$ for some $0<\delta \ll 1$ and  $\eta ^+:=\psi _* \eta\in H^{1,1}_{\BC}(X^+)$.
Consider $E:=p^*\eta -q^*\eta ^+$, then we claim that $-E$ is $p$-ample. Indeed, if  $\tilde C\subset W$ is a $p$-vertical curve, then $q_*\tilde C\ne 0$, and as we have seen above that $[q_*\tilde C]\in R^+$, so $\eta ^+\cdot q_*\tilde C>0$ and hence $-E\cdot \tilde C>0$. It follows by the negativity lemma that $E\geq 0$ and the support of $E$ contains $\Ex(p)$. But then ${\rm Supp}(E)=\Ex(p)$. Clearly $\Supp(E)\subset \Ex(q)$, as $\psi$ is an isomorphism in codimension $1$. 
We claim that $\Ex(q)=\Supp(E)=\Ex(p)$. Indeed, let $\tilde C\subset W$ be a $q$-vertical curve. Then $p_*C\neq 0$ and in fact $p_*C\subset \Null(\alpha)$ as $\psi:X\bir X^+$ is an isomorphism on $X\setminus\Null(\alpha)$. Thus $E\cdot \tilde C=\eta\cdot p_*\tilde C=(\alpha+\delta(K_X+B))\cdot p_*\tilde C<0$, and hence $\Ex(q)\subset \Supp(E)$. 

We will now show that $-E|_{{\rm Supp}(E)}$ is ample. Let $E_i$ be a component of $\Ex(p)=\Ex(q)$. We have $C_i:=p(E_i)$ and $C_i^+:=q(E_i)$ are curves and $\theta:=(p|_{E_i}, q|_{E_i}):E_i\to C_i\times C_i^+$ is finite. Let $\tilde C_i$ (resp. $\tilde C_i^+$) be a curve on $E_i$ dominating $C_i\times \{x\}$ (resp. $\{x\}\times C_i^+$) for general $x\in C_i^+$ (resp. $x\in C_i$), then $q_* \tilde C_i=0$ (resp. $p_*\tilde C_i^+=0$). Thus  $[C_i^+]\in R^+$ and $C_i\cdot \alpha =0$ so that  $[C_i]\in R$. It follows that $-\eta|_{C_i^+}$ (resp. $\eta|_{C_i}) $ is K\"ahler. Since $\theta $ is finite, then  $-E|_{E_i}\equiv (q^*\eta ^+-p^*\eta)|_{E_i}=\theta^*(-\eta|_{C_i},\eta^+|_{C^+_i})$ is ample and so 
$-E|_{{\rm Supp}(E)}$ is ample. Perturbing the coefficients of $E$ slightly we may assume that $E$ is a $\mbQ$-Cartier divisor. But then we can contract $E$ to a point by Lemma \ref{l-cont}. 

Let $g:W\to Z$ be the induced bimeromorphic morphism, then we claim that we have induced morphisms $f:X\to Z$ and $f^+:X^+\to Z$. To this end, if $C\subset W$ is a $q$ (respectively $p$) exceptional curve, then $C$ is contained in ${\rm Supp}(E)=\Ex(p)=\Ex(q)$. %and so $\eta \cdot p_*C<0$ (respectively,  $\eta ^+\cdot q_* C>0$).Since at least one of $p_*C,q_*C$ is non-zero, then \[ E\cdot C=(p^*\eta -q^*\eta ^+)\cdot  C <0\] and so $C$ is contained in the support of $E$. 
Since $g$ contracts $E$ to a point, then
$g_*C=0$ and so by the rigidity lemma $X\to Z$ and $X^+\to Z$ are morphisms.

Note that by construction $f:X\to Z$ contracts a non-empty set of $\alpha$-trivial curves (i.e. curves in $R$). We claim that in fact $f$ contracts every curve in $R$. Let $C\subset X$ be a curve in $R$ so that $\alpha \cdot C=0$. If $C$ is not contracted by $f$, then $p$ is an isomorphism over the general points of $C$ and we let $\tilde C:=p^{-1}_*C$ and $C^+:=q_* \tilde C$. Then $\alpha ^+\cdot C^+=0$.
Since $K_{X^+}+B^++t\omega ^++a\alpha ^+$ is nef for any $0<t\ll 1$ and some $a\gg 0$, then
\[(K_{X^+}+B^++t\omega ^+)\cdot C^+=(K_{X^+}+B^++t\omega ^++a\alpha ^+)\cdot C^+\geq 0\]
and taking the limit we have $(K_{X^+}+B^+)\cdot C^+\geq 0$.
It follows that
\[0\leq E\cdot \tilde C=(p^*\eta -q^*\eta ^+)\cdot \tilde C=(K_X+B)\cdot C-(K_{X^+}+B^+)\cdot C^+<0.\]
This is impossible and so $C$ is contracted by $f$.
 It follows that  $f$ is $(K_X+B)$-negative. By \cite[Lemma 2.44]{DH20}, $Z$ has rational singularities and by \cite[Lemma 3.3]{HP16} $f^*H^{1,1}_{\rm BC}(Z)=R^\perp\subset H^{1,1}_{\rm BC}(X)$.
Thus $\rho (X/Z)=1$.
Finally, since $f$ contracts the set of $\alpha$-trivial curves, we have $\alpha =f^*\alpha _Z$. 

We will now check that $\alpha _Z$ is a K\"ahler class. 
To this end, let $V\subset Z$ be a positive dimensional subvariety of $Z$.    Note that $ \dim f(\Ex(f))=0$. Let $V'$ be the strict transform of $V$. Then $(\alpha_{Z})^{\dim V}\cdot V={\alpha}^{\dim V}\cdot V'>0$ and $V'$ is not contained in ${\rm Null}(\alpha)$.  Thus from \cite[Theorem 2.29]{DHP22} it follows that $\alpha_{Z}$ is a K\"ahler class. 
Since $\psi _*$ is an isomorphism, then $\rho (X^+/Z)=1$.
We have already seen that $f^+$ contracts a $(K_{X^+}+B^+)$-positive curve and so $K_{X^+}+B^+$ is $f^+$-ample.
Finally, it is  easy to see that $(f^+)^*\alpha _Z+\delta (K_{X^+}+\Delta ^+)$ is K\"ahler for $0<\delta\ll 1$, and hence $X^+$ is a K\"ahler $3$-fold.

\end{proof}~\\

Next we prove the existence of divisorial contractions in dimension $3$ for strongly $\mbQ$-factorial pairs.
\begin{theorem}\label{t-divcont} 
Let $(X,B)$ be a strongly $\mbQ$-factorial compact K\"ahler $3$-fold klt pair such that 
\begin{enumerate}
   \item $K_X+B$ is pseudo-effective,
    \item $\alpha =[K_X+B+\beta]$ is nef and big for some K\"ahler form $\beta$, and
\item $\alpha ^\perp \cap \overline{\rm NA}(X)=R$ is an extremal ray of divisorial type.
\end{enumerate} 
Then the divisorial contraction $f:X\to Z$ exists and there is a K\"ahler class $\alpha _Z$ on $Z$ such that $\alpha = f^*\alpha _Z$.
\end{theorem}
\begin{proof} Recall that by definition $\alpha ^\perp \cap \overline{\rm NA}(X)=R$ is an extremal ray of divisorial type if and only if $\dim {\rm Null}(\alpha)=2$. 
We begin by observing the following.
\begin{claim}\label{clm:2-dim-null-locus}
    There is an irreducible divisor $S$ such that ${\rm Null}(\alpha)=S$, $S$ is a Moishezon surface covered by family of $\alpha$-trivial curves and $S\cdot R<0$.
\end{claim}
\begin{proof}
Let $\omega$ be a K\"ahler class on $X$. Since $\alpha$ is big, $\alpha-\eps\omega$ is big for $0<\epsilon \ll 1$, and we have the Boucksom-Zariski decomposition $\alpha -\epsilon \omega =\sum s_jS_j+P$ where $P$ is modified nef and big and in particular the restriction of $P$ to any surface is pseudo-effective. Let $S\subset\Null(\alpha)$ be a divisor (a surface) and $S'\to S$ the minimal resolution. Then $\alpha |_{S'}$ is nef but not big (as $(\alpha |_{S'})^2=0$), therefore $S$ coincides with some component of $\sum s_jS_j$, which for simplicity we denote by $S_1$.
    We claim that $S$ is Moishezon. To this end, consider $b={\rm mult}_{S}(B)$ and \[\left(1+\frac {1-b}{s_1}\right)\alpha \vert_{S'}=\left(K_X+B+\beta +\left(\frac {1-b}{s_1}\right)\left(\sum s_jS_j+P+\epsilon \omega\right)\right)\Big\vert_{S'}.\]
     Since ${\rm mult}_{S}\left(B+\frac {1-b}{s_1}\sum s_jS_j\right)=1$, we have $\left(K_X+B +\frac {1-b}{s_1}\sum s_jS_j\right)\big\vert_{S'}=K_{S'}+B_{S'}$ where $B_{S'}\geq 0$.
       %Suppose first that $\alpha |_{S'}\equiv 0$. 
       Then 
    \begin{equation}\label{eqn:restriction-to-resolution}
         \left(1+\frac {1-b}{s_1}\right)\alpha |_{S'}=K_{S'}+B_{S'}+\left(\beta +\frac {1-b}{s_1}(P+\epsilon \omega)\right)\Big\vert_{S'}.
    \end{equation} 
    Since $B_{S'}\geq 0$, $P|_{S'}$ is pseudo-effective, $(\beta +\frac {1-b}{s_1}(P+\epsilon \omega))|_{S'}$ is big and $\alpha |_{S'}$ is not big, it follows that $K_{S'}$ is not pseudo-effective, and hence $S'$ is projective by classification and $H^2(S', \mcO_{S'})=0$ (e.g. see \cite[Lemma 2.43]{DH20}). In particular, $S$ is Moishezon. 
    
    We claim that $S$ is covered by an analytic family of $\alpha$-trivial curves $\{C_t\}$.
    Fix $A_{S'}$ an ample divisor on $S'$. Since $H^2(S', \mcO_{S'})=0$, $\alpha |_{S'}$ is represented by an $\mbR$-divisor. 
    Since $\alpha |_{S'}$ is nef, it is in the cone $\overline {NM}_1(S')$ of movable curves. By \cite[Theorem 1.3]{Ara10} and \cite[Theorem 1.9]{Das20}, for any $\epsilon >0$, we have
a decomposition \[\alpha |_{S'}\equiv C_{\epsilon}+\sum _{i=1}^k\lambda _iM_i,\]
where $C_{\epsilon}\in (\overline{NE}_1(S')_{(K_{S'})\geq 0}+\overline {NM}_1(S'))_{(K_{S'}+\epsilon A_{S'})\geq 0}$, $\lambda _i>0$, $(K_{S'}+\epsilon A_{S'})\cdot C_{\epsilon}\geq 0$ and the $M_i$ are movable curves. If $\alpha |_{S'}\cdot M_i=0$, then the pushforward of $M_i$ onto $S$ defines a movable curve on $S$ which is $\alpha$-trivial and the claim follows. Therefore, we may assume that $\alpha |_{S'}\cdot M_i>0$ for all $i$.
Now, $0=(\alpha |_{S'})^2=\alpha |_{S'}\cdot (C_{\epsilon}+\sum _{i=1}^k\lambda _iM_i)$ and since $\alpha |_{S'}$ is nef then $k=0$, i.e. $\alpha |_{S'}\equiv C_{\epsilon}$. But then $\alpha |_{S'}\cdot (K_{S'}+\epsilon A_{S'})\geq 0$. Taking the limit as $\epsilon \to 0$, it follows that
$\alpha |_{S'}\cdot K_{S'}\geq 0$. Since $\alpha |_{S'}$ is nef, then $\alpha |_{S'}\cdot \gamma\geq 0$ for any pseudo-effective class $\gamma$. Thus $\alpha |_{S'}\cdot B_{S'}\geq 0$, $\alpha |_{S'}\cdot P|_{S'}\geq 0$, and $\alpha |_{S'}\cdot \omega |_{S'}\geq 0$. It follows from \eqref{eqn:restriction-to-resolution} and $(\alpha'|_{S'})^2=0$ that all these terms are equal to $0$. But then $\alpha |_{S^\nu}\cdot \omega |_{S^\nu}=0$ where $S^\nu\to S$ is the normalization. Since $\omega |_{S^\nu}$ is K\"ahler, then $\alpha |_{S^\nu}\equiv 0$. This contradicts the inequality above $\alpha |_{S'}\cdot M_i>0$, and so the claim holds.
    %and covered by an analytic family of $\alpha$-trivial curves $\{C_t\}$.
    %we consider the nef reduction map, i.e. a resolution $g:S'\to S$ and a morphism $h:S'\to T$ where $\dim T=0,1$, $\alpha_{S'}:=\alpha|_{S'}$ is trivial on the fibers of $h$ and non-trivial on non-vertical curves through a general point $s\in S'$. Note that $\dim T=0$ iff $\alpha_{S'}\equiv 0$.  

Let $\{C_t\}_{t\in T}$ be an analytic family of $\alpha$-trivial curves on $S$ covering $S$ as proved in the claim above. Then $[C_t]\in R$ for all $t\in T$ and we have
\[0>-\epsilon \omega\cdot C_t =(\alpha -\epsilon \omega)\cdot C_t =\left(\sum s_jS_j+P\right)\cdot C_t\geq s_{1}S\cdot C_t\]
as $S_i\cdot C_t\geq 0$ for $i\neq 1 $ and $S=S_1$. Thus $S\cdot C_t<0$, and hence $S\cdot R<0$ as $R=\mbR^+\cdot[C_t]$. In particular, if $C$ is any curve with $[C]\in R$, then $S\cdot C<0$ and hence $S={\rm Null}(\alpha)$.
\end{proof}
 Let $b:={\rm mult}_S(B)$ and $\Delta:=B+(1-b)S$. 
 \begin{claim} There exists a dlt model $\mu :X'\to X$ of $(X,\Delta)$. In particular 
\begin{enumerate}
    \item $(X',\Delta ')$ is $\Q$-factorial and dlt, where $\Delta':=\mu ^{-1}_*\Delta +{\rm Ex}(\mu)$, 
    \item $K_{X'}+\Delta '$ is nef over $X$ so that $\mu ^*(K_X+\Delta )-(K_{X'}+\Delta ')\geq 0$, and
    \item $K_{X'}+\Delta '=\mu ^*(K_X+B)+(1-b)S'+\sum a_jE_j$, where $S'=\mu ^{-1}_*S$ and $a_j>0$ for all $j$. 
    \end{enumerate}
    \end{claim}
    \begin{proof}
        Let $\mu :X'\to X$ be a projective log resolution and $\Delta':=\mu ^{-1}_*\Delta +{\rm Ex}(\mu)$, then $(X',\Delta ')$ is dlt. We may run the $(K_{X'}+\Delta ')$-MMP over $X$ (the existence of this MMP is well known since $\dim X=3$ and $\mu$ is projective, but see also \cite[Theorem 1.4]{DHP22} for the case $\dim X>3$). By \cite[Theorem 3.3]{DO23}, we may assume that the corresponding MMP terminates and replacing $(X',\Delta ')$ by the output of this MMP, we may assume that $K_{X'}+\Delta '$ is nef over $X$.
        By the negativity lemma, $\mu ^*(K_X+\Delta )-(K_{X'}+\Delta ')\geq 0$ and so (1) and (2) hold.
        (3) follows since $(X,B)$ is klt.
        
    \end{proof}
%(see eg. \cite[Theorem 1.27]{Fuj22}).  \[K_{X'}+\Delta '=\mu ^*(K_X+B)+(1-b)S'+\sum c_jE_j\]  where $S'=\mu ^{-1}_*S$ and $c_j>0$ (since $(X,B)$ is klt. Let $\eta'$ be a K\"ahler class on $X'$ and note that by \cite[Lemma 2.27]{DH20}, there exists a $\mu$-exceptional $\mbR$-divisor $F$ such that $\eta'+F\equiv _X0$. Thus $-F\equiv_X \eta '$ is ample over $X$ and by the negativity lemma $F\geq 0$. In particular, ${\rm Ex}(\mu)={\rm Supp}(F)$. Perturbing the coefficients of $F$ we may assume that $F$ is a $\mbQ$-divisor. It follows that the non-klt locus of $(X',\Delta ')$ is equal to $S'\cup {\rm Ex}(\mu)$. By (2), the image of ${\rm nklt}(X',\Delta ')$ is contained in ${\rm nklt}(X,\Delta )=S$. Thus $X'_U\to U$  is an isomorphism where $U:=X\setminus S=X\setminus\Null(\alpha)$.

We will now run the $\alpha'$-trivial $(K_{X'}+\Delta')$-minimal model program, where $\alpha '=\mu ^*\alpha$. 
Note that if $U:=X\setminus S=X\setminus\Null(\alpha)$, then $(U,\Delta |_U)$ is klt and so $X'_U\to U$  is a small bimeromorphic $(K_{X'}+\Delta ')$-trivial morphism. Since the $\mu$-exceptional locus is divisorial, then $X'_U\to U$  is an isomorphism.
If $C$ is a curve not contained in $W':=X'\setminus X'_U$, then $\alpha'\cdot C>0$. It follows easily that every $\alpha' $-trivial step of the $(K_{X'}+\Delta ')$-MMP will involve only curves contained in the complement of $X'_U$, and hence the restriction of this MMP to $X'_U$ is an isomorphism. 

Let $n(\alpha )$ be the nef dimension of the restriction of $\alpha$ to the normalization $S^\nu\to S$, see \cite[Definition 2.7]{BCE02}. Since $S$ is a Moishezon surface covered by a family of $\alpha$-trivial curves (as shown above), it follows that $0\leq n(\alpha )\leq 1$. If $n(\alpha )=0$, then $\alpha |_{S^\nu}\equiv 0$ and if $n(\alpha )=1$, then the nef reduction map $S^\nu \to T$ is a morphism to a smooth (projective) curve (see \cite[2.4.4]{BCE02}). Observe that if $n(\alpha )=0$, then the nef reduction map $S^\nu \to T$ is a morphism to a point and in this case the existence of the divisorial contraction $f:X\to Z $ follows from Corollary \ref{c-1}. We however give an alternate argument below which proves both of the cases: $n(\alpha)=0$ and $n(\alpha)=1$,  simultaneously. 

%{\bf Assume that $n(\alpha )=0$}, that is assume that $\alpha |_{S^\nu}\equiv 0$, then $\alpha '|_{\mu ^*S}\equiv 0$ i.e. $\alpha '\cdot C=0$ for any curve contained in $W'=S'\cup {\rm Ex}(\mu)=X'\setminus X'_U$. 
 If $\gamma\in H^{1,1}_{\BC}(X)$, then we will say that $\gamma\equiv_\alpha 0$ if and only if $\gamma\cdot C=0$ for every $\alpha$-trivial curve $C\subset X$, i.e. $\gamma\cdot C=0$ whenever $\alpha \cdot C=0$ for a curve $C\subset X$.
We  have that $(K_X+B-sS)\cdot R=0$ for some $s>0$ (as both $K_X+B$ and $S$ are $R$-negative), and hence $K_X+B-sS\equiv_\alpha 0$.
It follows that 
$\mu ^* (K_X+B-sS)\equiv _{\alpha '}0$ and hence that \[K_{X'}+\Delta '\equiv_{\alpha '} \Theta':=(1-b)S'+\sum a_jE_j+s\mu ^*S\geq 0,\] where %$\mathcal E'=(1-b)S'+\sum c_jE_j+s\mu ^*S\geq 0$ and 
${\rm Supp}(\Theta')=\lfloor \Delta '\rfloor$. 

\begin{claim}
    We can run the $\alpha '$-trivial $(K_{X'}+\Delta ')$-MMP 
    \[X'=X^0\dasharrow X^1 \dasharrow \ldots \dasharrow X^m\]
    so that if $\phi^i:X'\dasharrow X^i$, then
\begin{enumerate}
\item each flip and divisorial contraction is $\alpha'$-trivial, hence $\alpha ^i=\phi ^i_*\alpha'$ is nef,
\item  $\phi ^i|_{X'_U}$ is an isomorphism over $U$ and we denote by $X^i_U$ its image,
\item $(X^i,\Delta ^i:=\phi ^i_*\Delta ')$ is dlt,  $X^i$ is strongly $\mbQ$-factorial, and $\lfloor \Delta ^i\rfloor \subset W^i:=X^i\setminus X^i_U$, 
\item $K_{X^i}+\Delta ^i-\Theta^i\equiv_{\alpha ^i} 0$ and ${\rm Supp}(\Theta^i)=\lfloor \Delta ^i\rfloor$, and
    \item if $S^m$ is a component of $\lfloor \Delta ^m\rfloor$ then every $K_{S^m}+\Delta _{S^m}:=(K_{X^m}+\Delta ^m)|_{S^m}$ negative extremal ray $R^m$ is $\alpha _{S^m}:=\alpha _{X^m}|_{S^m}$-positive, i.e. $\alpha _{S^m}\cdot R^m>0$.
\end{enumerate}
    
\end{claim}
\begin{proof}
We will proceed by induction on $i$. Assume that we have already constructed $X'\dasharrow X^1 \dasharrow \ldots \dasharrow X^i$ with the required properties.
Suppose that there is a component $S^i$ of $\lfloor \Delta ^i\rfloor$ and a $K_{S^i}+\Delta _{S^i}:=(K_{X^i}+\Delta ^i)|_{S^i}$-negative extremal ray $R^i=\R^+[\Sigma ^i]$ which is $\alpha _{S^i}:=\alpha _{X^i}|_{S^i}$ non-positive and hence trivial, i.e. $\alpha _{S^i}\cdot R_i=0$.
Since $K_{X^i}+\Delta ^i\equiv_{\alpha ^i} \Theta ^i$ and ${\rm Supp}(\Theta^i)=\lfloor \Delta ^i\rfloor$, then $\Theta^i \cdot \Sigma ^i<0$ and hence, there is a component $\bar S^i$ of $\lfloor \Delta ^i\rfloor$ such that $\bar S^i\cdot \Sigma ^i<0$ 
and so $\Sigma ^i\subset \bar S^i$. Clearly  $\Sigma ^i$ is also $K_{\bar S^i}+\Delta _{\bar S^i}:=(K_{X^i}+\Delta ^i)|_{\bar S^i}$-negative and hence is contained in an $\alpha _{\bar S^i}$-trivial $(K_{\bar S^i}+\Delta _{\bar S^i})$-negative extremal face. Let $\bar \Sigma ^i$ be a curve spanning an extremal ray in this face, we may assume that $\bar S^i\cdot \bar \Sigma ^i<0$. Let $\pi :\bar S^i\to V$ be the corresponding extremal contraction,
then $-\bar S^i $ and $-(K_{\bar S^i}+\Delta _{\bar S^i})$ are relatively ample over $V$ and so by Theorem \ref{t-ext} there is a contraction $p:X^i\to Z^i$ such that $p|_{\bar S^i}=\pi$ and $p|_{X^i\setminus \bar S^i}$ an isomorphism. If $p$ is a flipping contraction, then the flip exists by Theorem \ref{t-flip}.
We let $X^i\dasharrow X^{i+1}$ be the induced bimeromorphism (be it a flip or divisorial contraction).
By construction $p$ is $\alpha ^i$-trivial and $p|_{X^i_U}$ is the identity. In particular, since $Z^i$ has rational singularities (this follows from \cite[Lemma 2.44]{DH20}), $\alpha ^i\equiv p^*\alpha _{Z^i}$ where $\alpha _{Z^i}$ is nef.  (1-3)$_{i+1}$ are easily seen to hold.
Since $K_{X^i}+\Delta ^i-\Theta^i\equiv_{\alpha ^i} 0$ and $p$ is $\alpha ^i$-trivial, it also follows that \[K_{Z^i}+\Delta _{Z^i}-\Theta_{Z^i}:=p_*(K_{X^i}+\Delta ^i-\Theta^i)\equiv_{\alpha _{Z^i}} 0\] and hence, pulling back to $X^{i+1}$ we have $K_{X^{i+1}}+\Delta ^{i+1}-\Theta^{i+1}\equiv_{\alpha ^{i+1}} 0$ and (4)$_{i+1}$  holds.
By termination of flips Theorem \ref{thm:termination}, after finitely many steps, we obtain $X\dasharrow X^m$ such that 
(5) holds.
\end{proof}
Now recall the nef reduction morphism $S^\nu\to T$.  Note that if $P$ is a component of $\mu ^* S$, then there is an induced morphism $\sigma _P:P\to T$ such that for any curve $C$ on $P$ we have $\alpha '\cdot C=0$ if and only if $C$ is vertical over $T$ i.e. $\sigma _{P,*}C=0$. Note that if $n(\alpha|_{S^\nu})=0$, then $T$ is a point and every curve is vertical over $T$. Since $\phi ^i:X'\dasharrow X^i$ is $\alpha '$-trivial, it follows easily that if $\phi ^i_*P=P^i\ne 0$, then the induced bimeromorphic map $P\dasharrow P^i$ is defined over $T$ and in particular $P^i\dasharrow T$ is a morphism.
\begin{claim}
    If $n(\alpha|_{S^\nu} )=0$, then the map $\phi^m:X'\dasharrow X^m$ contracts $\mu ^*S$ and if $n(\alpha|_{S^\nu} )=1$, then the map 
    $\phi^m:X'\dasharrow X^m$ contracts $S'$ and every component $E_j$ of $\mu ^*S$ such that $n(\alpha '|_{E_j})=1$. Thus  if $n(\alpha |_{S^\nu})=0$, then $\phi :X\dasharrow X^m$ is a bimeromorphic map that contracts $S$ (and no other divisors) and if $n(\alpha|_{S^\nu})=1$, then $\phi :X\dasharrow X^m$ contracts only $S$ and may extract some divisors of $X'$ such that $n(\alpha '|_{E_j})=0$.
\end{claim}
\begin{proof}
    Recall that $F'\geq 0$ is a $\mu$-exceptional $\Q$-divisor such that $-F'$ is $\mu$-ample. 
    Let $\Xi':=\mu ^*S+\epsilon F'$ for $0<\epsilon \ll 1$.
    We claim that for $t\gg 0$ we have that $(-\Xi'+t\alpha ')|_{S'}$ and $(-\Xi'+t\alpha ')|_{E_j}$ are K\"ahler for every component $E_j$ of ${\rm Ex}(\mu)$.
    To see that $(-\Xi'+t\alpha ')|_{S'}$ is K\"ahler, note that $-S|_{S^\nu}\equiv_{\alpha|_{S^\nu}}  \frac {-1}s(K_X+B)|_{S^\nu}\num_{\alpha|_{S^\nu}}\frac{1}{s}\beta|_{S^\nu}$ is ample over $T$. Now if $\dim T=0$, then $-S|_{S^\nu}$ ample and $\alpha|_{S^\nu}\num 0$, thus $(-S+t\alpha)|_{S^\nu}$ is K\"ahler for any $t>0$. If $\dim T=1$, then $\alpha|_{S^\nu}=\vphi^*\gamma$ for some K\"ahler class $\gamma$ on $T$ (see \cite[Proposition 2.11]{BCE02}), where $\vphi:S^\nu\to T$ is the nef reduction map. Thus $(-S+t\alpha)|_{S^\nu}$ is K\"ahler for $t\gg 0$. Since $-F'|_{S'}$ is $\mu|_{S'}$-ample and hence ample over $S^\nu$, we have $(-\mu ^*S+t\alpha '-\epsilon F')|_{S'}$ is K\"ahler for $0<\eps\ll 1$.
    To see that $(-\Xi'+t\alpha ')|_{E_j}$ is K\"ahler, we consider 2 cases. If $\mu (E_j)$ is a point, then $(-\mu ^*S+t\alpha ')|_{E_j}\equiv 0$ and $ (-\epsilon F')|_{E_j}$ is K\"ahler and the claim follows. Therefore, we may assume that the normalization $V_j\to \mu (E_j)$ is a smooth curve. If $\alpha |_{V_j}\equiv 0$, then $\mu (E_j)$ is a curve in $R$ and hence $-S|_{V_j}$ is ample and so $(-\mu ^*S-\epsilon F')|_{E_j}$ is ample for $0<\eps\ll 1$, and $\alpha '|_{E_j}\equiv 0$. The claim follows.
    Finally, if $\alpha |_{V_j}\ne  0$, then $\alpha |_{V_j}$ is K\"ahler so that $(-S+t\alpha)|_{V_j}$ is K\"ahler for $t\gg 0$ and hence $(-\mu^*S+t\alpha'-\epsilon F)|_{E_j}$ is K\"ahler  for $0<\eps\ll 1$ and the claim follows.

    We claim that $\phi^m$ is $\Xi'$ non-positive, i.e. if $p:W\to X'$ and $q:W\to X^m$ is the normalization of the graph of $\phi^m$, then we claim that $p^*\Xi'-q^*(\phi ^m_*\Xi')\geq 0$. Recall that $p$ and $q$ are isomorphisms over $X_{U}'=\mu^{-1}U=\mu^{-1}(X\setminus S)$, since so is $\phi^m$. We now check that $-p^* \Xi'$ is nef over $X^m$. To this end consider a curve $C$ on $W$ such that $q_*C=0$, then $p_*C\ne 0$ is contained in $\Supp(\mu^*S)$ and $\alpha '\cdot p_*C=\alpha ^m\cdot q_*C=0$.
    But then  $-p^* \Xi'\cdot C=-\Xi'\cdot p_*C>0$, and in particular, $-p^*\Xi'$ is nef over $X^m$.
    %As $-\mathcal F'$ is ample on every component of $\mu ^{-1}(S)=X'\setminus X'_U$, it follows that $-p^* \mathcal F'$ is nef over $X^m$. 
    Let $\Xi^m:=q_* p^*\Xi'=\phi^m_*\Xi'$, then $q^*\Xi^m-p^*\Xi'$ is $q$-exceptional and $q$-nef so that by the negativity lemma, $p^*\Xi'-q^*\Xi^m\geq 0$; this proves our claim.\\
    
    Suppose now that $n(\alpha|_{S^\nu})=0$ and $\phi^m _*(\mu ^{*}S)\ne 0$ (resp. $n(\alpha|_{S^\nu})=1$ and there is a component 
    $E$ of $S'+\sum E_j$ such that $n(\alpha '|_{E})=1$ and $E^m:=\phi^m_*E\neq 0$).
Let $\lambda$ be the smallest constant such that $\Theta^m-\lambda \Xi^m\leq 0$ (resp. ${\rm mult}_{E^m}(\Theta^m-\lambda \Xi^m)\leq 0$,  where $E^m=\phi^m_*E$). Thus in both cases, there is component $E$ of $S'+\sum E_j$ such that ${\rm mult }_{E^m}(\Theta^m-\lambda \Xi^m)=0$, and  there is a family  of $\alpha ^m$-trivial curves $\{C^m_t\}$ covering $E^m$. We claim that $(K_{X^m}+\Delta ^m)\cdot C^m_t\geq 0$.  
    If this is not the case, then \[ [C^m_t]\in (\alpha ^m) ^\perp\cap \overline{\rm NA}(E^m)_{K_{E^m}+\Delta _{E^m}<0 },\] 
    where $K_{E^m}+\Delta _{E^m}=(K_{X^m}+\Delta ^m)|_{E^m}$. Since $\alpha^m$ is nef, by the cone theorem on $E^m$, there is a $(K_{E^m}+\Delta _{E^m})$-negative $\alpha^m$-trivial extremal ray, contradicting (5) above.
    If $C_t$ is the strict transform of $C^m_t$ on $E=(\phi^m) ^{-1}_*E^m$, then 
    \[0\geq (\Theta^m-\lambda \Xi^m)\cdot C^m_t= (K_{X^m}+\Delta ^m-\lambda \Xi^m) \cdot C^m_t\geq -\lambda \Xi^m \cdot C^m_t\geq -\lambda \Xi'\cdot C_t>0\]
    where the first inequality follows by our definition of $\lambda$, the second as $(K_{X^m}+\Delta ^m-\Theta ^m)\equiv _{\alpha ^m}0$, the third was observed above, the fourth as $p^*\Xi'- q^*\Xi^m\geq 0$ and ${\rm mult}_{E^m}(p^*\Xi'- 
    q^*\Xi^m)=0$ and the last as $(-\Xi'+t\alpha ')|_E$ is ample. This is the required contradiction and the claim is proven.

\end{proof}
\begin{claim}
    The map $\phi^m:X'\dasharrow X^m$ contracts every component of $\mu ^*S$ dominating $T$, and hence if $n(\alpha )=0$, then $\phi :X\dasharrow X^m$ is a bimeromorphic map that contracts $S$ (and no other divisors) and if $n(\alpha )=1$, then $\phi :X\dasharrow X^m$ contracts only $S$ and may extract some divisors of $X'$ which are contained in $\mu ^{-1}(S)$ that do not dominate $T$.
\end{claim}
\begin{proof}
    Recall that $F'\geq 0$ is a $\mu$-exceptional $\Q$-divisor such that $-F'$ is $\mu$-ample. Since $-S|_{S^\nu}\equiv_\alpha  \frac {-1}s(K_X+B)|_{S^\nu}$ %$\equiv_\alpha \frac {1}s\omega _{S^\nu}$
    is ample over $T$, then $\Xi':=\mu ^*S+\epsilon F'$ is a $\Q$-divisor such that $-\Xi'|_{S'}$ and $-\Xi'|_{E_j}$ are ample over $T$ for every component $E_j$ of ${\rm Ex}(\mu)$ and all $0<\epsilon \ll 1$. 
    
    We claim that $\phi^m$ is $\Xi'$ non-positive, i.e. if $p:W\to X'$ and $q:W\to X^m$ is the normalization of the graph of $\phi^m$, then we claim that $p^*\Xi'-q^*(\phi ^m_*\Xi')\geq 0$. Recall that $p$ and $q$ are isomorphisms over $X_{U}'=\mu^{-1}U=\mu^{-1}(X\setminus S)$, since so is $\phi^m$. We now check that $-p^* \Xi'$ is nef over $X^m$. To this end consider a curve $C$ on $W$ such that $q_*C=0$, then $p_*C\ne 0$ is contained in $\Supp(\mu^*S)$ and $\alpha '\cdot p_*C=\alpha ^m\cdot q_*C=0$.
    But then $p_*C$ is vertical over $T$ and so $-p^* \Xi'\cdot C=-\Xi'\cdot p_*C>0$; in particular, $-p^*\Xi'$ is nef over $X^m$.
    %As $-\mathcal F'$ is ample on every component of $\mu ^{-1}(S)=X'\setminus X'_U$, it follows that $-p^* \mathcal F'$ is nef over $X^m$. 
    Let $\Xi^m:=q_* p^*\Xi'=\phi^m_*\Xi'$, then $q^*\Xi^m-p^*\Xi'$ is $q$-exceptional and $q$-nef so that by the negativity lemma, $p^*\Xi'-q^*\Xi^m\geq 0$; this proves our claim.\\
    
    Suppose now that $\phi^m _*(\mu ^{*}S)$ dominates $T$ i.e. that $\phi^m _*(\mu ^{*}S)\to T$ is surjective and let $\lambda$ be the smallest constant such that $(\Theta^m-\lambda \Xi^m)_\textsubscript{hor}\leq 0$,  where $(\ldots )_\textsubscript{hor}$ denotes the horizontal components over $T$ i.e. those components that dominate $T$. Then there is a component $E^m$ of $\phi^m _*(\mu ^{*}S)$ dominating $T$  such that ${\rm mult }_{E^m}(\Theta^m-\lambda \Xi^m)=0$, and  there is a family  $C^m_t$ (contained in the fibers over $T$) of $\alpha ^m$-trivial curves covering $E^m$. We claim that $(K_{X^m}+\Delta ^m)\cdot C^m_t\geq 0$.  
    If this is not the case, then \[ [C^m_t]\in (\alpha ^m) ^\perp\cap \overline{\rm NA}(E^m)_{K_{E^m}+\Delta _{E^m}<0 }\] where $K_{E^m}+\Delta _{E^m}=(K_{X^m}+\Delta ^m)|_{E^m}$. Since $\alpha^m$ is nef, by the cone theorem on $E^m$, there is a $(K_{E^m}+\Delta _{E^m})$-negative $\alpha^m$-trivial extremal ray, contradicting (5) above.
    If $C_t$ is the strict transform on $E=(\phi^m) ^{-1}_*E^m$, then 
    \[0\geq (\Theta^m-\lambda \Xi^m)\cdot C^m_t= (K_{X^m}+\Delta ^m-\lambda \Xi^m) \cdot C^m_t\geq -\lambda \Xi^m \cdot C^m_t\geq -\lambda \Xi'\cdot C_t>0\]
    where the first inequality follows by our definition of $\lambda$, the second as $(K_{X^m}+\Delta ^m-\Theta ^m)\equiv _{\alpha ^m}0$, the third was observed above, the fourth as $p^*\Xi'- q^*\Xi^m\geq 0$ and ${\rm mult}_{E^m}(p^*\Xi'- q^*\Xi^m)=0$ and the last as $-\Xi'|_E$ is ample over $T$. This is the required contradiction and the claim is proven.
\end{proof}

We will now show that in fact $\phi ^m$ contracts $\mu ^*S$. To this end, it suffices to show that $\Theta^m=0$ or equivalently that $q^*\Theta^m=0$.
Since $q^*\Theta^m$ is effective and exceptional over $X$, by the negativity lemma it suffices to show that $q^*\Theta^m$ is nef over $X$. Let $C\subset W$ be a curve such that $\mu _*p_*C=0$. If $q_*C=0$, then $C\cdot q^*\Theta^m=0$. If instead $q_*C\ne 0$, then
       $\alpha ^m\cdot q_*C=\alpha \cdot \mu _* p_*C=0$. Clearly $q^*\Theta^m\cdot C\geq 0$ if $q_*C$ is not contained in any component $P$ of the support $\Theta^m$ (and hence of the support of $\lfloor \Delta ^m\rfloor$). So assume that $q_*C$ is contained in a component $P$ of $\Theta^m$. 
       Since $\Theta^m\equiv _{\alpha ^m}K_{X^m}+\Delta ^m$ and $(K_{X^m}+\Delta ^m)|_P$ is non-negative on $\alpha ^m$-trivial curves (by (5) above), then $q^*\Theta^m\cdot C=(K_{X^m}+\Delta ^m)\cdot q_*C \geq 0$. Therefore $q^*\Theta^m$ is nef over $X$ and hence $\Theta^m=0$ as observed above.

We will now show that the induced bimeromorphic map $\phi: X\dasharrow X^m$ is a morphism. Let $p:W\to X$ and $q:W\to X^m$ be a resolution of the graph of $\phi$. Pick $\omega ^m$ a K\"ahler class on $X^m$ and $E_1,\ldots , E_r$ the $p$-exceptional divisors generating $H^{1,1}_{\BC}(W)/p^*H^{1,1}_{\BC}(X)$ as in \cite[Lemma 2.32]{DH20}. Then we may write $q^*\omega ^m+\sum e_iE_i\equiv _X0$ for some $e_i\in\mbR$, and hence there is a $\omega\in H^{1,1}_{\BC}(X)$ such that $p^*\omega \equiv q^*\omega ^m+\sum e_iE_i$. %Note that $\sum e_iE_i\geq 0$ by the negativity lemma.
Pick $r\in \R$ such that $(\omega +rS)\cdot R=0$, then $p^*(\omega +rS)\equiv q^*\omega ^m+\sum e_iE_i+rp^*S$. Let $C$ be a $q$-exceptional curve. If $p_*C=0$, then $(q^*\omega ^m+\sum e_iE_i+rp^*S)\cdot C=(\omega+rS)\cdot p_*C=0$. If $p_*C\neq 0$, then 
$p^*\alpha \cdot C=q^*\alpha ^m\cdot C=0$ and so $p_*C\in R$ and hence $(\omega +rS)\cdot p_*C=0$ so that $(q^*\omega ^m+\sum e_iE_i+rp^*S)\cdot C=0$. Recall that $\Theta^m=0$, in particular, $\phi$ contracts $S$ and doesn't extract any divisor. Thus
we have shown that $\sum e_iE_i+rp^*S$ is  $q$-exceptional and $q$-numerically trivial. Then by the negativity lemma $\sum e_iE_i+rp^*S=0$ and so $p^*(\omega +rS)\equiv q^*\omega ^m$. But then, if $C$ is $p$-exceptional and not $q$-exceptional, we have
\[0=C\cdot p^*(\omega +rS)=C\cdot q^*\omega ^m=q_*C \cdot \omega ^m>0\]
which is impossible. Since no such curves exist, then $\phi: X\dasharrow X^m$ is a morphism which we will now denote by $f:X\to Z$.

As we have seen above, $f|_U$ is an isomorphism and $f$ contracts a non-empty set of $\alpha$-trivial curves. Since $\overline{\rm NA}(X)\cap \alpha ^\perp=R$, then $f$ contracts the set of all curves in $R$ (and no other curves).
Since $K_X+B+\beta \equiv _Z 0$, then
$-(K_X+B)$ is $f$-ample. Then from \cite[Lemma 2.44]{DH20} it follows that $Z$ has rational singularities.
By \cite[Lemma 3.3]{HP16}, $\rho (X/Z)=1$ and $\alpha \equiv f^* \alpha _Z$ for some $\alpha _Z\in H^{1,1}_{\rm BC}(Z)$. By \cite[Theorem 2.29]{DHP22}, $\alpha _Z$ is K\"ahler
if and only if for any positive dimensional subvariety  $W\subset Z$,  $\alpha _Z^{\dim W}\cdot W>0$ holds.
If $W\not \subset f(S)$, then the strict transform $W'\subset X$ is not contained in $S$ and so $\alpha _Z^{\dim W}\cdot W=\alpha ^{\dim W'}\cdot W'>0$. Therefore, we may assume that $W=f(S)$ is a (projective) curve. Since $f$ is a projective morphism, $S$ is projective, and thus there is a (projective curve) $W'\subset S$ such that $f(W')=W$. Then $\alpha\cdot W'>0$, since by our construction above a curve in $X$ is contracted by $f$ if and only it is $\alpha$-trivial. Thus by the projection formula $\alpha_Z\cdot W>0$; this completes our proof.

\end{proof}

We conclude this section by proving the existence of a flipping contraction for non $\mbQ$-factorial $3$-folds.  
\begin{lemma}\label{l-mod}
    Let $(X,B)$ be a compact K\"ahler $3$-fold klt pair, and $\alpha$ a nef and big class such that ${\rm Null}(\alpha)=C$ is a finite union of curves. Then there exists a bimeromorphic morphism of normal K\"ahler varieties $\nu:X'\to X$ and an effective $\nu$-exceptional $\mbQ$-Cartier divisor $\Psi \geq 0$ such that $\omega '= \nu ^*\alpha -\Psi$ is K\"ahler and $\Psi=\nu ^{-1}(C)$. In particular, $-\Psi|_{{\rm Supp}(\Psi)}$ is ample.
\end{lemma}

% \begin{proof}[Alternate Proof]\footnote{Om: This argument does not give $\nu(\Supp\Psi)=C$.}
%     By Theorem \ref{thm:non-kahler-equal-null}, $E^{as}_{nK}(\alpha)=\Null(\alpha)=C$, and hence $\alpha$ is a modified K\"ahler class by definition, see \cite[Definition 2.2(i)]{Bou04}. Then the result follows from \cite[Lemma 2.35]{DH20}.
% \end{proof}

\begin{proof}
    Fix a K\"ahler form $\omega$ on $X$, and let $\mu:\tilde X\to X$ be a log resolution of $(X, B)$ and $C$. By Lemma \ref{lem:as-locus-is-analytic} and Theorem \ref{thm:non-kahler-equal-null}  there is a positive current $T\in \alpha$ with weak analytic singularities such that $E_+(T)=C$, and $T \geq \epsilon \omega$ for some $\epsilon>0$. By Lemma \ref{lem:resolution-of-current-singularities} we may assume that we can write $\mu ^* \alpha\equiv [\tilde\Theta]+\Phi$, where $\Phi\geq 0$ is an effective $\mbR$-divisor such that $\mu (\Supp \Phi)=C$, and $\tilde \Theta$ is a closed positive $(1,1)$ current such that $\tilde\Theta\geq \epsilon \mu^*\omega$ and the class $[\tilde \Theta-\epsilon \mu ^*\omega ]\in H^{1,1}_{\BC}(\tilde X)$ is  nef.
    In particular, $-\Phi$ is nef over $X$, and thus from the negativity lemma it follows that ${\rm Supp }(\Phi)=\mu ^{-1}(C)$.
    %Clearly, $\Supp(\Phi)\subset \mu^{-1}(C)$; for the reverse inequality, if there is a point $p\in \mu^{-1}(C)\setminus\Supp(\Phi)$, then let $\Gamma\subset X'$ be a vertical curve passing through $p$ such that $\Gamma\cap \Supp(\Phi)\neq \emptyset$ (such a curve exists as $\mu (\Supp(\Phi) )=C$). Thus $\Phi\cdot \Gamma>0$, contradicting the fact that $-\Phi$ is nef over $X$. 
  % , then there is a $\mu$-exceptional curve $D\subset \tilde X$ which properly intersects  ${\rm Supp}(\Phi)$. But then $\Phi\cdot C>0$, contradicting the fact that $-\Phi $  is nef over $X$. %\footnote{Om: Is this equality correct? I see that $\Supp(\Psi)\subset \mu^{-1}(C)$, but if there is a point $p\in \mu^{-1}(C)$ such that for all vertical curves $\Gamma$  in $\mu^{-1}(C)$ through $p$, $[\tilde\theta]\cdot\Gamma=0$ hold, then $\Psi\cdot \Gamma=0$; this does not imply that $\Gamma\subset \Supp(\Psi)$. }

    Let $\Delta:=\mu ^{-1}_*B+{\rm Ex}(\mu)$, then \[K_{\tilde X}+\Delta \equiv _X G:=\sum _{i=1}^k g_iG_i\geq 0,\]
    where ${\rm Supp }(G)={\rm Ex}(\mu)$, as $(X,B)$ is klt. Let $E=\sum_{i=1}^k e_iG_i $, where $e_i>0$, $E$ is $\mu$-exceptional, and  $-E$ is $\mu$-ample. Perturbing the coefficients of $E$, we may assume that $e_1,\ldots ,e_k$  are linearly independent over $\Q(g_1,\ldots, g_k)$. 
    
    We now run the $[\tilde \Theta]$-trivial $(K_{\tilde X}+\Delta)$-Minimal Model Program over $X$ with scaling of $-E$ as in \cite[Theorem 2]{Kol21c}. This means that at each step we only contract $(K_{\tilde X}+\Delta)$-negative extremal rays over $X$ that are $[\tilde \Theta]$-trivial.
    All steps of this MMP exists by \cite{Nak87}, \cite{DHP22} and \cite[Theorem 1.2]{Fuj22cc}.
    Note that since each step of this mmp is $E$ negative, the contracted locus is contained in the support of $E$ and hence in the support of $\lfloor \Delta\rfloor$.
Since the mmp for surfaces holds, termination follows easily from the usual arguments for special termination  \cite{Fuj07}.

    Let $\phi:\tilde X\dasharrow X'$ be the output of this MMP and $\nu :X'\to X$ the induced map. 
Following \cite[Theorem 2]{Kol21c}, we may assume that there is a $\nu$-exceptional $\mbQ$-Cartier divisor $F\geq 0$ such that %$-F$ is ample over $X$, and hence 
$-F+\lambda [\Theta']$ is K\"ahler over $X$ for some $\lambda>0$, where $\Theta':=\phi_*\tilde\Theta$. Note that every step of this MMP  is also $[\tilde \Theta-\epsilon \mu ^*\omega]$-trivial, and hence the class $[\Theta '-\epsilon \nu ^* \omega] $ is nef.

    Let $U=X\setminus C$, $\tilde U=\mu ^{-1}(U)$, and $U'=\nu ^{-1}(U)$, then $\Phi|_{\tilde U}=0$ and so $[\tilde \Theta]|_{\tilde U}\equiv _U 0$, and hence over $U$ this is the usual relative MMP. Since $G'|_{U'}\geq 0$, where $G'=\phi _*G$, and $G'|_{U'}$ is nef, by the negativity lemma it follows that $G'|_{U'}= 0$. Thus, there are no $\nu |_{U'}$-exceptional divisors. As $-F|_{U'}$ is relatively ample, we have $U'=U$.
    
    It then follows that, for $0<\delta\ll 1$ the class of 
    \[(1+\lambda \delta)\Theta '-\delta F=\epsilon \nu ^*\omega-\delta(F-\lambda \Theta ')+(\Theta'-\epsilon \nu ^*\omega)\]
    is K\"ahler (as $[\Theta '-\epsilon \nu ^* \omega]$ is a nef class and $[\epsilon \nu ^*\omega-\delta (F-\lambda \Theta')]$ is a K\"ahler class for $0<\delta\ll 1$). We let $\omega ':=[\Theta '-\frac\delta {1+\lambda \delta}F]$ and $\Psi =\Phi '+\frac\delta {1+\lambda \delta}F$. Then
     \[\nu ^*\alpha =\phi _*(\mu ^*\alpha)=\phi _*([\tilde \Theta] +\Phi )=[\Theta']+\Phi '=\omega '+\Psi,\] 
     and the claim follows.  
     Note that perturbing $\Psi$ and $\omega '$, we may assume that $\Psi$ is a $\Q$-divisor.
\end{proof}

\begin{theorem}\label{thm:contraction-of-small-locus}
    Let $(X,B)$ be a compact K\"ahler $3$-fold klt pair and $K_X+B$ pseudo-effective. Suppose that there is a K\"ahler class $\omega$ such that $\alpha =K_X+B+\omega$ is a nef and big class but not K\"ahler, and ${\rm Null}(\alpha)=C$ is a finite union of curves. Then there exists a bimeromorphic morphism $f:X\to Z$ of normal compact K\"ahler varieties and a K\"ahler 
    class $\alpha _Z$ on $Z$ such that $\alpha= f^*\alpha_Z$; in particular, $\Ex(f)=\Null(\alpha)$.
\end{theorem}
\begin{proof}
    By Lemma \ref{l-mod}, there is a bimeromorphic morphism of normal compact K\"ahler varieties $\nu :X'\to X$ and an effective $\nu$-exceptional $\Q$-divisor $\Psi$ such that $\nu (\Psi)=C$ and $-\Psi|_{{\rm Supp}(\Psi)}$ is ample.
    By Lemma \ref{l-cont}, there is a bimeromorphic morphsim $g:X' \to Z$ to a normal compact analytic variety $Z$ such that $g({\rm Supp}(\Psi ))$ is a finite set of points and $X\setminus\Supp(\Psi)\to Z\setminus g(\Supp(\Psi))$ is an isomorphism. By the rigidity lemma
(see \cite[Lemma 4.1.13]{BS95}), we obtain a bimeromorphic morphism $f : X \to Z$ such that $f\circ\nu=g$. In particular, $\Ex(f)=C=\Null(\alpha)$. Thus, $-(K_X+B)$ is $f$-ample, hence $Z$ has rational singularities by \cite[Lemma 2.44]{DH20}. Then by \cite[Lemma 3.3]{HP16}, there is an $(1,1)$ class $\omega_Z\in H^{1,1}_{\BC}(Z)$  such that $\alpha=f^*\omega_Z$. Finally, by the projection formula and \cite[Theorem 2.29]{DHP22} it follows that $\omega_Z$ is a K\"ahler class. 
\end{proof}

\section{Cone Theorem and Minimal Models}\label{sec:cone-and-mm}
In this section we will establish the cone theorem and existence of minimal models for klt pairs $(X, B)$ of dimension $3$ such that $K_X+B$ is pseudo-effective.\\

Let $K\subset \R ^n$ be a compact convex set. An \textit{extremal point} of $K$ is a point  $x$ such that if $x=tx_1+(1-t)x_2$ for $x_i\in K$, then $x=x_1$ or $x=x_2$; equivalently $K\setminus \{x\}$ is convex. An \textit{exposed point} of $K$ is a point $x\in K$ such that $\{x\}=K\cap \{f=0\}$ where $\{f=0\}$ is a supporting hyperplane for $K$. Every exposed point is extremal but not vice versa. We let ${\rm ext}(K)$ be the set of extremal points, ${\rm exp}(K)$ the set of exposed points and we denote by ${\rm conv}(X)$ the convex hull of a set $X\subset \R ^n$ and ${\rm cl}(X)$ the closure of $X$. For a detailed a discussion on these topics see \cite[Chapter 1]{HW20}. We need the following result:
\begin{theorem}\label{t-conv}    Let $K\subset \R ^n$ be a non-empty compact convex set. Then
\[{\rm cl(conv(exp}(K)))={\rm conv( ext}(K))=K.\]  
In particular, the sets ${\rm exp(K)}$ and ${\rm ext(K)}$ are both non-empty.
\end{theorem}

\begin{proof}
    The equality ${\rm conv( ext}(K))=K$ is well known and due to Minkowski (see \cite[Theorem 1.21]{HW20}), and the equality ${\rm cl(conv(exp}(K)))=K$ follows from \cite[Theorem 1.23]{HW20}.
    \end{proof}
    
   \begin{proof}[Proof of Theorem \ref{t-cone}] By standard arguments (see the proof of \cite[Corollary 5.3]{DHP22}), passing to a strongly $\Q$-factorial model (which exists by \cite[Lemma 2.28]{DH20}), we may assume that $X$ is strongly $\Q$-factorial. Let $K$ be a compact convex slice of $\overline{\rm NA}(X)$ of Euclidean dimension $n-1$, where $n=\dim_{\mbR} H^{1,1}_{\BC}(X)$, i.e. the cone generated by $K$ is equal to $\NA(X)$, and $R$ an extremal ray corresponding to an exposed point $\{x\}\in K$. Then by definition, there is a $(1,1)$ class $\alpha\in H^{1,1}_{\BC}(X)$ such that $K$ is contained in the half-space $\alpha_{\>0}:=\{\gamma\in N_1(X): \alpha\cdot\gamma\>0\}$ and $\{x\}=K\cap\alpha^\bot$. In particular, $R=\overline{\rm NA}(X)\cap \alpha ^\perp$. Since $\alpha$ is non-negative on $\overline{\rm NA}(X)$, it is nef. If $(K_X+B)\cdot R<0$, then for $0<a\ll 1$, it follows easily that $\beta :=\alpha -a(K_X+B)$ is positive on $\overline{\rm NA}(X)\setminus\{0\}$, and hence K\"ahler (e.g. see the proof of \cite[Claim 3.25]{DHY23}). Thus $\alpha =a(K_X+B+\frac 1 a \beta)$. Replacing $\alpha$ by $\frac 1 a\alpha$ and letting $\omega = \frac 1 a \beta$ we may assume that $\alpha \equiv K_X+B+\omega$, where $\omega$ is a K\"ahler class. By Theorems \ref{t-flipcont} and \ref{t-divcont}, there is a corresponding flipping or divisorial contraction and hence $R=\R ^+[\Gamma ]$ for some curve $\Gamma$ on $X$. By \cite[Theorem 1.23]{DO23}, we may assume that $\Gamma$ is a rational curve and $0>(K_X+B)\cdot \Gamma \geq -6$. Since $X$ is a K\"ahler variety, by a Douady space argument there are at most countably many such extremal rays $\{\mbR^+\cdot[\Gamma _i]\}_{i\in I}$.
 Let $V:=\overline{\rm NA}(X)_{K_X+B \geq 0}+\sum _{i\in I}\mathbb R ^+[\Gamma _i]$. Following \cite[Lemma 6.1]{HP16} it suffices to show that $\overline{\rm NA}(X) = \overline V$. To the contrary assume that $\NA(X)\setminus \overline V\neq\emptyset$. Since $\overline V$ is a convex set (as so is $V$) and $\NA(X)$ is generated by $K$, from Theorem \ref{t-conv} it follows that there is an exposed point $x\in K$ not contained in $\overline V$. Let $R$ be the $(K_X+B)$-negative extremal ray corresponding to the exposed point $x\in K$. Then from what we have seen above, there is a rational curve $\Gamma$ such that $0<-(K_X+B)\cdot \Gamma\<6$ and $R=\mbR^+\cdot[\Gamma]$, but all such extremal rays are contained in $\overline V$ by definition of $V$, this is a contradiction.
 % By , if the inclusion 
 % $\overline V\subset \NA(X)$ is strict, then there is a $K_X+B$-negative extremal ray $R\in \NA(X)$ not contained in $\overline V$ corresponding to an exposed point $x\in K$. By what we have seen above, $R=\R ^+[\Gamma _i]$. Therefore $\overline V= \NA(X)$ and the claim follows.

% The rest of the proof is standard (see e.g. \cite[Lemma 61.]{HP16} and the proof of \cite[Theorem 1.6]{DHY23}) and we omit it.
 \end{proof}

% \begin{corollary}\label{cor:dlt-contraction}\footnote{Om: This is new.}
%     Let $(X, B)$ be a $\mbQ$-factorial compact K\"ahler $3$-fold dlt pair such that $K_X+B$ pseudo-effective. If $R$ is a $(K_X+B)$-negative extremal ray, then there is a projective bimeromorphic morphsim $f:X\to Z$ with connected fibers such that a curve $C\subset X$ is contracted by $f$ if and only if $[C]\in R$. 
% \end{corollary}
%Finally we are ready to prove Theorem \ref{thm:contractions}.
\begin{proof}[Proof of Theorem \ref{thm:contractions}]
Choose $0<\eps\ll 1$ so that $\omega+\eps B$ is K\"ahler. Thus replacing $B$ by $(1-\eps)B$ and $\omega$ by $\omega+\eps B$ we may assume that $(X, B)$ is a klt pair. By \cite[Lemma 2.27]{DH20} there is a small projective bimeromorphic morphism $g:X'\to X$ such that $X'$ is strongly $\mbQ$-factorial and $K_{X'}+B'=g^*(K_X+B)$. Let $\alpha':=g^*\alpha$ and run a $\alpha'$-trivial $(K_{X'}+B')$-MMP using the cone Theorem \ref{t-cone} and the contraction Theorems \ref{t-flipcont} and \ref{t-divcont}. Note that each step of this MMP preserves strong $\mbQ$-factoriality by \cite[Lemma 2.5]{DH20}. This MMP terminates by Theorem \ref{thm:termination}. Let $\phi:X'\bir X^n$ be the end result of this MMP, $B^n:=\phi_*B'$ and $\alpha^n:=\phi_*\alpha'$. We claim that $\Null(\alpha^n)$ does not contain any surface. Assume to the contrary that there is a surface $S^n\subset \Null(\alpha^n)$. Let $S\subset \Null(\alpha)$ be the strict transform of $S^n$ under the induced map $\psi:=\phi\circ g^{-1}:X\bir X^n$. We note that $S$ exists, since $\psi$ does not extract any divisor. Then from the proof of Claim \ref{clm:2-dim-null-locus} it follows that $\Null(\alpha)=S$, and $S$ is covered by an analytic family of $\alpha$-trivial curves $\{C_t\}_{t\in T}$ such that $S\cdot C_t<0$ for all $t\in T$. Now recall that $\alpha=K_X+B+\omega$, where $\omega$ is a K\"ahler class; let $\omega^n:=\psi_*\omega$ and $C^n_t:=\psi_*C_t$. Then $\{C^n_t\}_{t\in T}$ is a family of $\alpha^n$-trivial curves covering $S^n$. %such that $S^n\cdot C^n_t<0$ for all $t\in T$. 
Now since $\alpha^n=K_{X^n}+B^n+\omega^n$, we have $(K_{X^n}+B^n)\cdot C^n_t=(K_{X^n}+B^n)|_{S^n}\cdot C^n_t=-\omega^n|_{S^n}\cdot C^n_t<0$, since $\omega^n|_{S^n}$ is big as $\omega^n$ is a modified K\"ahler class. This is a contradiction to the fact that $(K_{X^n}+B^n)\cdot C\>0$ for all $\alpha^n$-trivial curves $C$.

 Now by Theorem \ref{thm:non-kahler-equal-null}, $\Null(\alpha^n)=E^{as}_{nK}(\alpha^n)$ is an analytic subset of $X^n$, and hence $\Null(\alpha^n)$ is either an empty set or a finite union of curves. First assume that $\Null(\alpha^n)$ is a finite union of curves. In this case by Theorem \ref{thm:contraction-of-small-locus}, there is a small proper bimeromorphic morphism $h:X^n\to Z$ to a normal analytic variety $Z$ and a K\"ahler class $\alpha_Z\in H^{1,1}_{\BC}(Z)$ such that $\alpha^n=h^*\alpha_Z$ and $\Ex(h)=\Null(\alpha^n)$. Let $W$ be a resolution of the graph of $\phi:X'\bir X^n$ and $p:W\to X'$ and $q:W\to X^n$ are projection morphisms; then $(g\circ p)^*\alpha=(h\circ q)^*\alpha_Z$. Let $C\subset W$ a curve contracted by $g\circ p$, then $(h\circ q)^*\alpha_Z\cdot C=(g\circ p)^*\alpha\cdot C=0$. Since $\alpha_Z$ is K\"ahler, this implies that $C$ is also contracted by $h\circ q$. Then by the rigidity lemma \cite[Lemma 4.1.13]{BS95}, there is a morphism $f:X\to Z$ such that $f\circ (g\circ p)=h\circ q$. In particular, $\alpha=f^*\alpha_Z$, and thus a curve $C\subset X$ is contracted by $f$ if and only if $\alpha \cdot C=0$, i.e. $[C]\in R$. The only thing that remains to be shown in this case is that $f$ is projective. To see this, observe that $-(K_X+B)\num_f \omega$, and hence $-(K_X+B)$ is a $\mbQ$-Cartier $f$-ample divisor on $X$, thus $f$ is projective.

 Now if $\Null(\alpha^n)=\emptyset$, then by \cite[Theorem 2.29]{DHP22}, $\alpha^n$ is a K\"ahler class on $X^n$. We rename $X^n$ by $Z$; then by a similar argument as above it follows that there is a projective bimeromorphic morphism $f:X\to Z$ such that $\alpha=f^*\alpha^n$. In particular, a curve $C\subset X$ is contracted by $f$ if and only if $\alpha\cdot C=0$, i.e. $[C]\in R$.

 Finally, if $f$ is a divisorial contraction, then by a similar argument as in the proof of \cite[Proposition 3.36]{KM98} it follows that $Z$ is $\mbQ$-factorial, and $(Z, f_*B)$ is dlt (see \cite[Lemma 3.38]{KM98}). If $f$ is a flipping contraction, then the flip $f^+:X^+\to Z$ is defined as $X^+:=\mbox{Projan}\oplus_{m\>0}f_*\mcO_X(md(K_X+B))$, where $d$ is the Cartier index of $K_X+B$; the finite generation of this ring is due to Shokurov, see \cite[Theorem 4.3]{CHP16}, also, for a more general result in all dimensions see \cite[Theorem 1.3]{DHP22}.

\end{proof}

\begin{proof}[Proof of Theorem \ref{t-mmp-mfs}] %Since $X$ is K\"ahler, there is a K\"ahler form $\omega$ such that $[K_X+B+\omega]$ is K\"ahler. Perturbing $\omega$, we may assume that $[\omega ]\in H^{1,1}_{\rm BC}(X)$ is very general. By Lemma \ref{l:mmp-scaling} we may run the $K_X+B$ mmp with scaling of $\omega$ say $\phi :X\dasharrow X_m$. If $K_X+B$ is pseudo-effective, then $K_X+B+t\omega$ is big for any $t>0$ and so we obtain $\lambda _m=0$ and $K_{X_m}+B_m$ is nef. Otherwise,. $K_{X_m}+B_m$ is not pseudo-effective and hence not nef so that $\lambda _m>0$. By  \cite{HP15} and \cite[Theorem 1.2]{DH20}, in this case there is a Mori fiber space $X_n\to Z$.
Suppose that $K_X+B$ is not nef, then by Theorem \ref{t-cone} and its proof, there is a $(K_X+B)$-negative extremal ray $R$ with nef supporting class $\alpha$ such that $R=\overline{\rm NA}(X)\cap \alpha ^\perp$, and the corresponding flipping or divisorial contractions exist by Theorem \ref{thm:contractions}. Let $(X,B)\dasharrow (X_1,B_1)$ be the corresponding flip or divisorial contraction. We may replace $(X,B)$ by $(X_1,B_1)$ and repeat the above procedure. By termination of flips Theorem \ref{thm:termination}, after finitely many steps we obtain the required minimal model. 
\end{proof}

\bibliographystyle{hep}
\bibliography{4foldreferences}
\end{document}